\documentclass[11pt,a4]{article}
\usepackage{latexsym}
\usepackage{amsmath}
\usepackage{amssymb}
\usepackage{amsthm}
\usepackage{amscd}
\usepackage{color}
\definecolor{dblue}{rgb}{0,0,0.45}
\definecolor{red}{rgb}{0.7,0,0}

\usepackage{color}

\numberwithin{equation}{section}

\newtheorem{theorem}{Theorem}[section]
\newtheorem{lemma}[theorem]{Lemma}
\newtheorem{corollary}[theorem]{Corollary}
\newtheorem{proposition}[theorem]{Proposition}

\theoremstyle{definition}

\newtheorem{remark}[theorem]{Remark}
\newtheorem{definition}[theorem]{Definition}
\newtheorem{example}[theorem]{Example}

\theoremstyle{remark}

\title{A thought on generalized Morrey spaces}
\author{Yoshihiro Sawano}
\begin{document}

\maketitle

\tableofcontents

\begin{abstract}
Morrey spaces can complement the boundedness properties
of operators that Lebesgue spaces can not handle.
Morrey spaces which we have been handling are called classical Morrey spaces.
However,
classical Morrey spaces are not totally enough to describe the boundedness properties.
To this end, we need to generalize parameters $p$ and $q$,
among others $p$.
\end{abstract}

\section{Introduction}
\label{s2}

This note studies the function spaces
which Guliyev, Mizuhara and Nakai defined 
in \cite{Guliyev94, Mizuhara91,Nakai94}.
Let $0<q \le p<\infty$.
We define the classical Morrey space ${\mathcal M}^p_q({\mathbb R}^n)$
to be the set of all measurable functions $f$ for which the quantity
\[
\|f\|_{{\mathcal M}^p_q}
\equiv
\sup_{x \in {\mathbb R}^n, r>0}
|Q(x,r)|^{\frac1p-\frac1q}
\left(\int_{Q(x,r)}|f(y)|^q\,dy\right)^{\frac1q}
\]
is finite,
where $Q(x,r)=\{y \in {\mathbb R}^n\,:\,
|x_1-y_1|,|x_2-y_2|,\ldots,|x_n-y_n| \le r\}$.
This space goes back to the papers
\cite{Morrey38, Morrey66} by Morrey
and the subsequent paper \cite{Peetre69} by Peetre.
To describe the endpoint case or to describe
the intersection space, sometimes it is useful
to generalize the parameter $p$:
let us suppose that $p$ comes from the function $t^{\frac{n}{p}}$.
So, we envisage the situation
where $t^{\frac{n}{p}}$ is replaced by a general function $\varphi$.
More precisely,
it is known as the Adams theorem that
$I_\alpha$ maps 
${\mathcal M}^p_q({\mathbb R}^n)$
to
${\mathcal M}^s_t({\mathbb R}^n)$
whenever $1<q \le p<\infty$ and $1<t \le s<\infty$
satisfy
\[
\frac{1}{s}=\frac{1}{p}-\frac{\alpha}{n}, \quad
\frac{q}{p}=\frac{t}{s}.
\]
Here $I_\alpha$ is the fractional integral operator
given by (\ref{eq:170426-99})
for a nonnegative measurable function $f:{\mathbb R}^n \to [0,\infty]$.
However, it is known that
we can not take $s=\infty$. 
In fact, $I_\alpha$ fails to be bounded 
from ${\mathcal M}^{n/\alpha}_q({\mathbb R}^n)$ to $L^\infty({\mathbb R}^n)$
for all $1\le q \le \frac{n}{\alpha}$.
To compensate for this failure,
we can use generalized Morrey spaces.
We will define the generalized Morrey norm
by
\[
\|f\|_{{\mathcal M}^\varphi_q}
\equiv
\sup_{x \in {\mathbb R}^n, r>0}
\frac{1}{\varphi(r)}\left(
\frac{1}{|Q(x,r)|}
\int_{Q(x,r)}|f(y)|^q\,dy\right)^{\frac1q}<\infty.
\]
Here $\varphi:(0,\infty) \to [0,\infty)$
is a function which does not vanish at least at some point.
Another advantage of generalized Morrey spaces
is that we can cover many function spaces related to Lebesgue spaces.

Although we do not consider the direct applications 
of generalized Morrey spaces to PDEs,
generalized Morrey spaces can be applied to PDEs.
Applications to the nondivergence elliptic differential equations
can be found in 
\cite{FHS17, GGA12,KMR14,Softova13-2,WNTZ12}.
Applications to the parabolic differential equations
can be found in 
\cite{ZJSZ10}.
See \cite{GuOm16, Softova13-1}
for the parabolic oblique derivative problem.
See \cite{KNS00} for applicatitons to Schr\"{o}dinger equations.
We refer to \cite{LMPS12} for the application to singular
integral equations.

This note is organized as follows:
We collect some preliminary facts
in Section \ref{s6}
In Section \ref{s3},
we define generalized Morrey spaces
and then discuss the structure of generalized Morrey spaces.
The conditions on $\varphi$ will be important.
So we discuss them quite carefully.
Section \ref{s4} is a detailed discusssion
of the boundedness properties of the operators.
Here we handle
the Hardy--Littlewood maximal operators,
the Riesz potentails,
the Riesz transforms
and 
the fractional maximal operators
as well as their generalizations.
Section \ref{s5} is a survey of
other related function spaces.

Here we list a series of (somewhat standard) notation we use in this note.
\begin{enumerate}
\item
Let $(X,d)$ be a metric space.
We denote by $B(x,r)$
the {\it ball centered at $x$ of radius $r$}.
Namely, we write
\[
B(x,r) \equiv
\{y \in {\mathbb R}^n\,:\, d(x,y)<r\}
\]
when $x \in {\mathbb R}^n$
and $r>0$.
\index{$B(x,r)$}
Given a ball $B$,
we denote by $c(B)$ its {\it center} 
and by $r(B)$ its {\it radius}.
In the Euclidean space ${\mathbb R}^n$,
we write $B(r)$ instead of $B(o,r)$,
where $o\equiv (0,0,\ldots,0)$.
\index{$B(r)$}
\item
A metric measure space
is a pair of a metric space $(X,d)$
and a measure $\mu$ such that
any open set is measurable.
\item
By a \lq \lq cube" we mean a compact cube in ${\mathbb R}^n$
whose edges are parallel to the coordinate {\it axes}.
The metric closed ball defined 
by $\ell^\infty$ is called a {\it cube}.
If a cube has center $x$ and radius $r$,
we denote it by $Q(x,r)$.
Namely, we write
\[
Q(x,r) \equiv
\left\{y=(y_1,y_2,\ldots,y_n) \in {\mathbb R}^n\,:\,
\max_{j=1,2,\ldots,n}|x_j-y_j| \le r\right\}
\]
when $x=(x_1,x_2,\ldots,x_n) \in {\mathbb R}^n$
and $r>0$.
\index{$Q(x,r)$}
{}From the definition of $Q(x,r)$,
its volume is $(2r)^n$.
We write $Q(r)$ instead of $Q(o,r)$.
Given a cube $Q$,
we denote by $c(Q)$ the {\it center of $Q$}
and by $\ell(Q)$ the {\it sidelength of $Q$}:
$\ell(Q)=|Q|^{1/n}$,
where $|Q|$ denotes the volume of the cube $Q$.
\item
Given a cube $Q$ and $k>0$,
$k\,Q$ means the {\it cube concentric to $Q$
with sidelength $k\,\ell(Q)$}.
Given a ball $B$ and $k>0$,
we denote by $k\,B$
the {\it ball concentric to $B$
with radius $k\,r(B)$}.
\item
Let $A,B \ge 0$.
Then $A \lesssim B$ and $B \gtrsim A$ mean
that there exists a constant $C>0$
such that $A \le C B$,
where $C$ depends only on the parameters
of importance.
The symbol $A \sim B$ means
that $A \lesssim B$ and $B \lesssim A$
happen simultaneously,
while $A \simeq B$ means that there exists
a constant $C>0$ such that $A=C B$.
\item
When we need to emphasize or
keep in mind that
the constant $C$
depends on the parameters
$\alpha,\beta,\gamma$ etc
Instead of $A\lesssim B$,
we write
$A\lesssim_{\alpha,\beta,\gamma,\ldots}B$.
\end{enumerate}

\section{Preliminaries}
\label{s6}

\subsection{Lebesgue spaces and integral inequalities}

Here we recall Lebesgue spaces and the integral inequalities.
\begin{definition}[Lebesgue space]\label{defi:150824-57}
\index{Lebesgue space@Lebesgue space}
Let $(X,{\mathcal B},\mu)$ be a measure space
and $0 < p \le \infty$.
\begin{enumerate}
\item
The Lebesgue norm
$L^p(\mu)$-(quasi-)norm of a measurable function $f$ is given
by
\begin{align*}
\| f \|_p
&\equiv
\| f \|_{L^{p}(\mu)}
\equiv
\left(
\int_X |f(x)|^p\,d\mu(x)
\right)^{\frac{1}{p}}, \, p<\infty \\
\| f \|_\infty
&\equiv
\| f \|_{L^\infty(\mu)}
\equiv
\sup\{ \lambda>0 \, : \,
|f(x)| \le \lambda \mbox{ for $\mu$-a.e. $x \in X$ }
\}, \, p=\infty.
\end{align*}
\item
Define
\begin{align*}
L^p(\mu)
&\equiv
\{ f:X \to \overline{\mathbb K} \, : \,
f \mbox{ is measurable and }
\| f \|_p<\infty
\}/\sim.
\end{align*}
Here the equivalence relation $\sim$
is defined by:
\begin{equation}
f \sim g \Longleftrightarrow f=g \mbox{ a.e.}
\end{equation}
and below omit this equivalence in defining function spaces.
\end{enumerate}
As usual if $(X,{\mathcal B},\mu)$ is the Lebesgue measure,
then we write
$L^p({\mathbb R}^n)=L^p(\mu)$
and
$\|\cdot\|_{{\rm W}L^p}=\|\cdot\|_{{\rm W}L^p(\mu)}$.
\index{$L^p(\mu)$}
\end{definition}

\begin{theorem}[H\"{o}lder's inequality]Let
$\displaystyle
0< p,q,r \le \infty
$
satisfy
$\displaystyle
\frac{1}{r}=\frac{1}{p}+\frac{1}{q}.
$
Then for $f \in L^p(\mu)$ and $g \in L^q(\mu)$,
$\displaystyle
\|f g\|_{L^r(\mu)} \le \|f\|_{L^p(\mu)}\|g\|_{L^q(\mu)}.
$
See {\rm\cite[Theorem 2.4]{AdFo-text-03}}
for example.
\end{theorem}

In addition to $L^p(\mu)$ with $0<p \le \infty$,
it is convenient to define $L^0(\mu)$:
The space $L^0(\mu)$
denotes the set of all measurable functions
considered modulo the difference on the set of measure zero.
\subsection{Integral operators}

We will study 
the Hardy--Littlewood maximal operators,
the Riesz potentails,
the Riesz transforms
and 
the fractional maximal operators.
So we recall the definition together with the boundedness
properties on $L^p({\mathbb R}^n)$.
For Theorems \ref{thm:1208-2}, \ref{thm:maximal},
\ref{thm:CZ-operator1}, \ref{thm:CZ-operator2},
\ref{thm:fractional}
we refer to \cite{Grafakos-text-08, Sawano-text-18, Stein-text-93, Stein-text-70}.
\begin{definition}[Hardy--Littlewood maximal operator]
For $f \in L^0({\mathbb R}^n)$,
define a function $M f$ by
\begin{equation}\label{eq:maximal operator}
M f(x)\equiv
\sup_{B \in {\mathcal B}}\frac{\chi_B(x)}{|B|}\int_B |f(y)|{\rm d}y
\quad (x \in {\mathbb R}^n).
\end{equation}
The mapping $M:f \mapsto Mf$
is called the {\it Hardy--Littlewood maximal operator}.
\index{Hardy--Littlewood maximal operator@Hardy--Littlewood maximal operator}
\end{definition}

We have the weak-$(1,1)$ boundedness of $M$.
\begin{theorem}[Hardy--Littlewood maximal inequality]
\label{thm:1208-2}
For $f \in L^1({\mathbb R}^n), \lambda>0$,
\[
\lambda|\,\{M f>\lambda \}\,|
\le 3^n\|f\|_1.
\]
\index{Hardy--Littlewood maximal inequality@
Hardy--Littlewood maximal inequality}
\end{theorem}

For $1<p \le \infty$
we have the strong-$(p,p)$ boundedness.
\begin{theorem}[$L^p({\mathbb R}^n)$-inequality]
\label{thm:maximal}
Let $1<p<\infty$. Then 
\[
\| M f \|_p
\le
\left(\frac{p\,2^p\cdot 3^n}{p-1}\right)^{\frac1p}
\| f \|_p
\]
for all $f \in L^0({\mathbb R}^n)$.
\index{Lp-inequality of the Hardy--Littlewood maximal operator@$L^p({\mathbb R}^n)$-inequality of the Hardy--Littlewood maximal operator}
\end{theorem}

We move on to the singular integral operators.
Among others we consider the Calder\'{o}n--Zygmund operators.

Here 
we are interested in the following type 
of the singular integral operators:
\begin{definition}[Singular integral operator]
A {\it singular integral operator}
is an $L^2({\mathbb R}^n)$-bounded linear operator
$T$
that comes with 
a function $K \in C^1({\mathbb R}^n \setminus \{0\})$
satisfying the following conditions:
\begin{enumerate}
\item
$(${\it Size condition}$)$
\index{size condition@size condition}
For all $x \in {\mathbb R}^n$,
\begin{equation}\label{eq:size condition}
|K(x)| \lesssim |x|^{-n}.
\end{equation}
\item
$(${\it Gradient condition}$)$
\index{gradient condition@gradient condition}
For all $x \in {\mathbb R}^n$,
\begin{equation}\label{eq:gradient condition}
|\nabla K(x)| \lesssim |x|^{-n-1}.
\end{equation}
\item
Let $f$ be an $L^2({\mathbb R}^n)$-function.
For almost all $x \notin {\rm supp}(f)$,
\begin{equation}\label{eq:081112-401}
Tf(x)=\int_{{\mathbb R}^n}K(x-y)f(y){\rm d}y.
\end{equation}
The function $K$ is called the {\it integral kernel} of $T$.
\end{enumerate}
\index{integral kernel@integral kernel}
\index{singular integral operator@singular integral operator}
\label{defi:Singular integral operator}
\end{definition}
More generally we consider the following type of
singular integral operators:
\begin{definition}\label{defi:CZ}
\index{Calderon-Zygmund operator@Calderon-Zygmund operator}
A (generalized) Calder\'{o}n-Zygmund operator
is an $L^2({\mathbb R}^n)$-bounded linear operator $T$,
if it satisfies the following conditions:
\index{generalized Calderon-Zygmund operator@(generalized) Calder\'{o}n-Zygmund operator}
\begin{enumerate}
\item
There exists a measurable function $K$
such that for all $L^\infty({\mathbb R}^n)$-functions
$f$
with compact support we have
\begin{equation}
\label{eq:sing1}
Tf(x)=\int_{{\mathbb R}^n}K(x,y)f(y)\,dy
\mbox{ for all } x \notin {\rm supp}{f}.
\end{equation}
\item
The kernel function $K$ satisfies the following estimates.
\begin{align}
\label{eq:sing3}
|K(x,y)| \lesssim |x-y|^{-n},
\end{align}
if $x \ne y$
and
\begin{align}
\label{eq:sing2}
|K(x,z)-K(y,z)|
+
|K(z,x)-K(z,y)|
\lesssim \frac{|x-y|}{|x-z|^{n+1}}\,,
\end{align}
if $0<2|x-y|<|z-x|$.
\end{enumerate}
\end{definition}

As a typical example we consider the 
Riesz transform:
\begin{example}\label{example:180916-62}
\index{Riesz transform@Riesz transform}
Let $j=1,2,\ldots,n$.
The singular integral operators,
which are represented by the $j$-th Riesz transform 
given by,
\begin{equation}\label{eq:170228-2}
R_j f(x)
\equiv
\lim_{\varepsilon \downarrow 0}
\int_{{\mathbb R}^n \setminus B(x,\varepsilon)}
\frac{x_j-y_j}{|x-y|^{n+1}}f(y){\rm d}y,
\end{equation}
are integral operators with singularity.
\end{example}

We have the weak-$(1,1)$ boundedness as before.\begin{theorem}\label{thm:CZ-operator1}
Suppose that $T$ is a CZ-operator.
Then $T$ is weak-$(1,1)$ bounded, that is
\begin{equation}
\left|\left\{ \, |Tf| > \lambda \right\}\right|
\lesssim \frac{1}{\lambda}\int_{{\mathbb R}^n} |f(x)|\,dx
\end{equation}
for all $f \in L^1({\mathbb R}^n) \cap L^2({\mathbb R}^n)$.
\end{theorem}
We have the strong boundedness
for $1<p<\infty$.
\begin{theorem}\label{thm:CZ-operator2}
Let $1<p<\infty$ and $T$ be a generalized singular integral operator,
which is initially defined on $L^2({\mathbb R}^n)$.
Then $T$ can be extended to $L^p({\mathbb R}^n)$
for all $1<p<\infty$ so that
\begin{equation}
\| \, T f \|_p \lesssim_p\,\| f \|_p
\end{equation}
for all $f \in L^p({\mathbb R}^n)$.
\end{theorem}

As a different type of singular integral operators,
we consider the fractional integral operators.
Let $0<\alpha<n$.
Let $I_\alpha$ be the {\it fractional maximal operator}
given by
\begin{equation}\label{eq:170426-99}
I_\alpha f(x)
\equiv
\int_{{\mathbb R}^n}\frac{f(y)}{|x-y|^{n-\alpha}}{\rm d}y
\quad (x \in {\mathbb R}^n)
\end{equation}
for a nonnegative measurable function $f:{\mathbb R}^n \to [0,\infty]$.

The following theorem is known
as the Hardy--Littlewood--Sobolev theorem.
Generalized Morrey spaces
can be used to refine this theorem.
\begin{theorem}[Hardy--Littlewood--Sobolev]\label{thm:fractional}
{\rm \cite{Hardy--Littlewood-1,Hardy--Littlewood-2,Sobolev38}}
Let $0<\alpha<n$.
\begin{enumerate}
\item
We have
\begin{equation}
\lambda^{\frac{n}{n-\alpha}}
|\{ x \in {\mathbb R}^n\,:\, |I_\alpha f(x)|>\lambda \}|
\lesssim \|f\|_1
\end{equation}
for all $f \in L^1({\mathbb R}^n)$ and $\lambda>0$.
\item
Assume that the parameters $\displaystyle 1<p<\frac{n}{\alpha}$
and $\displaystyle \frac{1}{q}=\frac{1}{p}-\frac{\alpha}{n}$.
Then we have $\| I_\alpha f\|_q \lesssim \| f \|_p$
for all $f \in L^p({\mathbb R}^n)$.
\end{enumerate}
\end{theorem}

Let $0 \le \alpha<n$.
The fractional maximal operator $M_\alpha$
of order $\alpha$ is defined by
\[
M_\alpha f(x) \equiv
\sup_{Q \in {\mathcal Q}}\ell(Q)^{\alpha-n}
\chi_Q(x)\int_Q|f(y)|\,dy
\]
for a measurable function $f$.
Note that 
$M_\alpha$ maps ${\mathcal M}^{n/\alpha}_1({\mathbb R}^n)$ boundedly to
$L^{\infty}({\mathbb R}^n)$.
As an opposite endpoint case,
we have the following boundedness:
\begin{theorem}\label{thm:180925-1}
Let $0<\alpha<n$.
Then
\begin{equation}
\lambda^{\frac{n}{n-\alpha}}
|\{ x \in {\mathbb R}^n\,:\, M_\alpha f(x)>\lambda \}|
\lesssim \|f\|_1
\end{equation}
for all $f \in L^1({\mathbb R}^n)$.
\end{theorem}

Using the boundedess of $M$,
we can prove the boundedness of $M_\alpha$.
Here for the sake of convenience for readers
we supply the proof.
\begin{theorem}\label{thm:180925-2}
Let $1<p<q<\infty$ and $0<\alpha<n$
satisfy $\frac1q=\frac1p-\frac{\alpha}{n}$.
Then
$\|M_\alpha f\|_{L^q} \lesssim \|f\|_{L^p}$
for all $f \in L^p({\mathbb R}^n)$.
\end{theorem}

\begin{proof}
By the monotone convergence theorem
we may assume that $f \in L^\infty_{\rm c}({\mathbb R}^n)$.
For $\lambda>0$ we let
$\Omega_\lambda\equiv\{M_\alpha f>\lambda\}$.
We observe that
\begin{align*}
(\|M_\alpha f\|_{L^q})^q
=q
\int_0^\infty \lambda^{q-1}|\Omega_\lambda|\,d\lambda
\le q
\int_0^\infty \lambda^{q-1-\frac{n}{n-\alpha}}(\|f\chi_{\Omega_\lambda}\|_{L^1})^{\frac{n}{n-\alpha}}\,d\lambda.
\end{align*}
We choose $u>0$ so that
$
n(1-u)=p\alpha.
$
Then we have
\[
\|f\chi_{\Omega_\lambda}\|_{L^1}
\le
\left(\|f\chi_{\Omega_\lambda}\|_{L^{\frac{n u}{n-\alpha}}}\right)^{u}
(\|f\|_{L^{p}})^{1-u}.
\]
Thus, we have
\begin{align*}
(\|M_\alpha f\|_{L^q})^q
&\lesssim
\int_0^\infty \lambda^{q-1-\frac{n}{n-\alpha}}\left(
\|f\chi_{\Omega_\lambda}\|_{L^{\frac{n u}{n-\alpha}}}\right)^{\frac{n u}{n-\alpha}}
(\|f\|_{L^{p}})^{\frac{n(1-u)}{n-\alpha}}\,d\lambda\\
&\simeq
(\|f\|_{L^{p}})^{\frac{n(1-u)}{n-\alpha}}
\int_{{\mathbb R}^n} |f(x)|^{\frac{n u}{n-\alpha}}M_\alpha f(x)^{q-\frac{n}{n-\alpha}}\,dx.
\end{align*}
Observe that
\begin{align*}
\left(q-\frac{n}{n-\alpha}\right)
\times
\left(\frac{(n-\alpha)p}{n u}\right)'
&=
\left(q-\frac{n}{n-\alpha}\right)
\times
\frac{(n-\alpha)p}{(n-\alpha)p-n u}\\
&=
\frac{(n-\alpha)p q-n p}{(n-\alpha)p-n u}\\
&=
\frac{q(p-1)n}{(p-1)n}.
\end{align*}
By the H\"{o}lder inequality, we obtain
\begin{eqnarray*}
\int_{{\mathbb R}^n} |f(x)|^{\frac{n u}{n-\alpha}}M_\alpha 
f(x)^{q-\frac{n}{n-\alpha}}\,d\mu(x)
\le
(\|f\|_{L^p})^{\frac{n u}{n-\alpha}}
(\|M_\alpha f\|_{L^q})^{q-\frac{n}{n-\alpha}}.
\end{eqnarray*}
As a result,
\begin{eqnarray*}
(\|M_\alpha f\|_{L^q})^q
\lesssim
(\|f\|_{L^{p}})^{\frac{n}{n-\alpha}}
(\|M_\alpha f\|_{L^q})^{q-\frac{n}{n-\alpha}}.
\end{eqnarray*}
If we arrange this inequality,
we obtain
\[
\|M_\alpha f\|_{L^q}
\lesssim
\|f\|_{L^{p}}.
\]
\end{proof}

In addition to the linear operators above
we sometimes consider
the commutator generated by BMO and these operators.
Here we recall the BMO (bounded mean oscillation) as follows:
\begin{definition}[${\rm BMO}({\mathbb R}^n)$ (space)]
\label{defi:170428-161}
Define
\[
\| f \|_{{\rm BMO}}
\equiv 
\sup_{Q \in {\mathcal Q}}\frac{1}{|Q|}\int_Q|f(y)-m_Q(f)|{\rm d}y
=\sup_{Q \in {\mathcal Q}}m_Q(|f-m_Q(f)|)
\]
for $f \in L^1_{\rm loc}({\mathbb R}^n)$.
\index{$m_Q$}
The \lq \lq norm" $\| \star \|_{{\rm BMO}}$
is called
the ${\rm BMO}({\mathbb R}^n)$ {\it norm}.
The space
${\rm BMO}({\mathbb R}^n)$
collects
all $f \in L^1_{\rm loc}({\mathbb R}^n)$
such that
$\|f\|_{{\rm BMO}}$ is finite.
Usually ${\rm BMO}({\mathbb R}^n)$ is considered
modulo the constant functions.
\index{${\rm BMO}({\mathbb R}^n)$ (space)}
\end{definition}

As the following theorem shows,
the BMO functions have high local integrability.
\begin{theorem}[John--Nirenberg inequality]
\label{thm:John--Nirenberg}
For any $\lambda>0$, a cube
$Q$ 
and a nonconstant ${\rm BMO}({\mathbb R}^n)$-function
$b$,
\[
\left|
\left\{
x \in Q \, : \, |b(x)-m_Q(b)|>\lambda
\right\}
\right\|
\lesssim_n
\exp\left(-\frac{D\lambda}{\| b \|_{{\rm BMO}}}\right),
\]
where $D$ depends only on the dimension.
\index{John--Nirenberg inequality@John--Nirenberg inequality}
\end{theorem}
We will consider the commutators
generated by BMO and these operators.
For the properties of these commutators
we refer to \cite{Grafakos-text-09}.
Let $a \in {\rm BMO}({\mathbb R}^n)$.
Then the we can consider the operators
\[
f \in L^\infty_{\rm c}({\mathbb R}^n)
\mapsto 
[a,T]f=a \cdot T f-T[a \cdot f], \quad
[a,I_\alpha]f=a \cdot I_\alpha f-I_\alpha[a \cdot f]
\]
despite the ambiguity which the definition of BMO$({\mathbb R}^n)$
causes.

Thanks to the following theorem,
we can extend our results 
of generalized Morrey spaces
by mimicking the proof of the boundedness
for the similar operators.
However, due to the similarity we do not consider
the boundedness and the definition of these operators.
Instead, we give some references when needed.
\begin{theorem}\label{thm:180925-3}
Let $0<\alpha<n$, $a \in {\rm BMO}({\mathbb R}^n)$
and $T$ be a generalized Calder\'{o}n--Zygmund operator.
\begin{enumerate}
\item
Let $1<p<\infty$.
We have
$\|[a,T]f\|_{L^p} \lesssim \|a\|_{\rm BMO}\|f\|_{L^p}$
for all $f \in L^\infty_{\rm c}({\mathbb R}^n)$.
\item
Let $1<p<q<\infty$ satisfy $\frac1q=\frac1p-\frac{\alpha}{n}$.
Then we have
$\|[a,I_\alpha]f\|_{L^q} \lesssim \|a\|_{\rm BMO}\|f\|_{L^p}$
for all $f \in L^\infty_{\rm c}({\mathbb R}^n)$.
\end{enumerate}
\end{theorem}
For the proof of Theorem \ref{thm:180925-3}
we refer to \cite{Grafakos-text-09}.

\section{Generalized Morrey spaces}
\label{s3}

Likewise we can replace $q$ by a function $\Phi$.

\subsection{Definition of generalized Morrey spaces}
\label{subsection:Definition of generalized Morrey spaces}

From the observation above, we are led to the following definition:
\begin{definition}\label{defi:180414-2}
Let $0<q<\infty$ and $\varphi:(0,\infty) \to [0,\infty)$
be a function which does not satisfy $\varphi \equiv 0$.
\begin{enumerate}
\item
\cite{Nakai94}
Define the generalized Morrey space
${\mathcal M}^\varphi_q({\mathbb R}^n)$
to be the set of all measurable functions $f$ such that
\begin{equation}\label{eq:131115-111}
\|f\|_{{\mathcal M}^\varphi_q}
\equiv
\sup_{x \in {\mathbb R}^n, r>0}
\varphi(r)\left(
\frac{1}{|Q(x,r)|}
\int_{Q(x,r)}|f(y)|^q\,dy\right)^{\frac1q}<\infty.
\end{equation}
\index{M phi q@${\mathcal M}^\varphi_q({\mathbb R}^n)$}
\item
\cite{Guliyev94}
Define the weak generalized Morrey space
${\rm w}{\mathcal M}^{\varphi}_q({\mathbb R}^n)$
to be the set of all measurable functions $f$ such that
\begin{align*}
\|f\|_{{\rm w}{\mathcal M}^\varphi_q}
\equiv 
\sup_{\lambda>0}
\|\lambda\chi_{(\lambda,\infty)}(|f|)\|_{{\mathcal M}^\varphi_q}<\infty.
\end{align*}
\index{w M phi q@${\rm w}{\mathcal M}^\varphi_q({\mathbb R}^n)$}
\end{enumerate}
\end{definition}
Although we torelate the case where $\varphi(t)=0$
for some $t>0$,
it turns out that there is no need to consider such possibility.

Before we go into more details,
clarifying remarks may be in order.
\begin{remark}
In \cite{Nakai94}
Nakai defined
the generalized Morrey space
${\mathcal M}^\varphi_q({\mathbb R}^n)$
to be the set of all measurable functions $f$ such that
\begin{equation}\label{eq:131115-1118}
\|f\|_{{\mathcal M}^\varphi_q}
\equiv
\sup_{x \in {\mathbb R}^n, r>0}
\frac{1}{\varphi(r)}\left(
\frac{1}{|Q(x,r)|}
\int_{Q(x,r)}|f(y)|^q\,dy\right)^{\frac1q}<\infty,
\end{equation}
or more generally
\[
\|f\|_{{\mathcal M}^\varphi_q}
\equiv
\sup_{x \in {\mathbb R}^n, r>0}
\frac{1}{\varphi(x,r)}\left(
\frac{1}{|Q(x,r)|}
\int_{Q(x,r)}|f(y)|^q\,dy\right)^{\frac1q}<\infty.
\]
See also
\cite{EUG12, Guliyev09, GAKS11, KoMi06, Nakai14, TaHe13, WNTZ12}.
Here we follow the notation by
Sawano, Sugano and Tanaka
\cite{SST09-1, SST11-1, SST11-2},
for example.
See also \cite{MiOh14-1}.
See \cite{AkKu14, DGS15, PeSa11,Sa(N)13-1,Sa(N)13-2,Shirai06-2} for another notation of generalized Morrey spaces,
for example.
Despite the difference of the notation,
the idea of defining the predual space depends on \cite{KoMi06,Shirai06-2}. 
\end{remark}

\begin{remark}
Some prefer to call $\varphi$ the weight
(see \cite{GoMu11} for example),
while
some prefer to call $\|\cdot \times w\|_{{\mathcal M}^\varphi_q}$
the weighted generalized Morrey norm
(see \cite{GuOm14} for example).
\end{remark}

We can recover the Lebesgue space $L^q({\mathbb R}^n)$
by letting $\varphi(t)\equiv t^{\frac{n}{q}}$ for $t>0$
as we have mentioned.
To compare Morrey spaces with generalized Morrey spaces,
we sometimes call Morrey spaces classical Morrey spaces.
In addition to the function $\varphi(t)=t^{\frac{n}{p}}$,
we consider the following typical functions:
\begin{example}\label{example:180611-1}\
\begin{enumerate}
\item
The function $\varphi \equiv 1$
generates $L^\infty({\mathbb R}^n)$
thanks to the Lebesgue convergence theorem.
\item
Let $0<q<\infty$ and $a \in {\mathbb R}$.
Define
$\varphi(t)=t^{\frac{n}{q}}(\log(e+t))^a$
for $t>0$.
We remark that
${\mathcal M}^\varphi_q({\mathbb R}^n) \ne \{0\}$
if and only if $a \le 0$.
In fact,
we have
\[
\|f\|_{{\mathcal M}^\varphi_q}
=
\sup_{x \in {\mathbb R}^n, r>0}
(\log(e+r))^a\|f\|_{L^q(Q(x,r))}.
\]
Thus, if $f$ is a measurable function
such that $\|f\|_{{\mathcal M}^\varphi_q}<\infty$ and if $a>0$,
then
we have
$\|f\|_{L^q(Q(x,r))}=0$
for any cube $Q(x,r)$.
Thus, $f=0$ a.e..
Convesely if $a \le 0$,
then $L^q({\mathbb R}^n) \subset 
{\mathcal M}^\varphi_q({\mathbb R}^n)$.
\item
Let $0<q \le p_1<p_2<\infty$.
Then
$\varphi(t)=t^{\frac{n}{p_1}}+t^{\frac{n}{p_2}}$,
$t>0$
can be used to express
the intersection
of ${\mathcal M}^{p_1}_q({\mathbb R}^n) 
\cap {\mathcal M}^{p_2}_q({\mathbb R}^n)$.
In general,
for $0<q<\infty$ and $\varphi_1,\varphi_2:(0,\infty)\to[0,\infty)$
satisfying $\varphi_1(t_1) \ne 0$ and $\varphi(t_2) \ne 0$
for some $t_1,t_2>0$
${\mathcal M}^{\varphi_1}_q({\mathbb R}^n)
\cap{\mathcal M}^{\varphi_2}_q({\mathbb R}^n)
={\mathcal M}^{\varphi_1+\varphi_2}_q({\mathbb R}^n)$
with equivalence of norms.
\item
Let $0<q \le p_1<p_2<\infty$.
Let $\varphi(t)=
\chi_{{\mathbb Q} \cap (0,\infty)}(t)t^{\frac{n}{p_1}}
+
\chi_{(0,\infty) \setminus {\mathbb Q}}(t)t^{\frac{n}{p_2}}
$
for $t>0$.
Then we have
${\mathcal M}^\varphi_q({\mathbb R}^n)={\mathcal M}^{p_1}_q({\mathbb R}^n) 
\cap {\mathcal M}^{p_2}_q({\mathbb R}^n)$
and for any $f \in L^0({\mathbb R}^n)$
\[
\|f\|_{{\mathcal M}^\varphi_q}
=
\max\{
\|f\|_{{\mathcal M}^{p_1}_q},
\|f\|_{{\mathcal M}^{p_2}_q}
\}.
\]
\item
We can consider the norm
\[
\sup_{x \in {\mathbb R}^n, r\in(0,1)}
|Q(x,r)|^{\frac1p}\left(
\frac{1}{|Q(x,r)|}
\int_{Q(x,r)}|f(y)|^q\,dy\right)^{\frac1q}
\]
for $0<q<p<\infty$.
In fact,
we take
$\varphi(t)=t^{\frac{n}{p}}\chi_{[0,1]}(t)$
for $t>0$.
\item
Likewise
we can consider the norm
\[
\sup_{x \in {\mathbb R}^n, r\in[1,\infty)}
|Q(x,r)|^{\frac1p}\left(
\frac{1}{|Q(x,r)|}
\int_{Q(x,r)}|f(y)|^q\,dy\right)^{\frac1q}
\]
for $0<q<p<\infty$.
In fact,
we take
$\varphi(t)=t^{\frac{n}{p}}\chi_{[1,\infty)}(t)$
for $t>0$.
\item
Let $0<q<\infty$
The uniformly locally $L^q$-integrable space
$L^q_{\rm uloc}({\mathbb R}^n)$ is the set of all measurable functions
$f$ for which
\[
\sup_{x \in {\mathbb R}^n, r\in(0,1)}
\left(\int_{Q(x,r)}|f(y)|^q\,dy\right)^{\frac1q}
=
\sup_{x \in {\mathbb R}^n}
\left(\int_{Q(x,1)}|f(y)|^q\,dy\right)^{\frac1q}
\]
is finite.
As before, if we let
$\varphi(t)=t^{\frac{n}{q}}\chi_{(0,1]}(t)$
for $t>0$,
then we obtain
$L^q_{\rm uloc}({\mathbb R}^n)
={\mathcal M}^\varphi_q({\mathbb R}^n)$.
We can define 
the weak uniformly locally $L^q$-integrable space
${\rm w}L^q_{\rm uloc}({\mathbb R}^n)$
similarly.
For $f \in L^0({\mathbb R}^n)$,
the norm is given by
\[
\|f\|_{{\rm w}L^q_{\rm uloc}}
\equiv
\sup_{\lambda>0}
\lambda\|\chi_{(\lambda,\infty]}(|f|)\|_{L^q_{\rm uloc}}.
\]
\end{enumerate}
\end{example}

The fifth example deserves a name.
We define the small Morrey space as follows:
\begin{definition}\label{defi:180611-1}
Let $0<q<p<\infty$.
The {\it small Morrey space}
$m^p_q({\mathbb R}^n)$
is the set of all measurable functions
$f$ for which the quantity
\[
\|f\|_{m^p_q}
\equiv
\sup_{x \in {\mathbb R}^n, r\in(0,1)}
|Q(x,r)|^{\frac1p}\left(
\frac{1}{|Q(x,r)|}
\int_{Q(x,r)}|f(y)|^q\,dy\right)^{\frac1q}
\]
is finite.
The {\it weak small Morrey space}
${\rm w}m^p_q({\mathbb R}^n)$
is defined similarly.
For $f \in L^0({\mathbb R}^n)$,
the norm is given by
\[
\|f\|_{{\rm w}m^p_q}
\equiv
\sup_{\lambda>0}
\lambda\|\chi_{(\lambda,\infty]}(|f|)\|_{m^p_q}.
\]
\end{definition}

\begin{example}{\rm \cite[Proposition A]{ErSa12}}\label{ex:180525-1}
Let $x \in {\mathbb R}^n$ and $r>0$.
Then
$$\displaystyle\|\chi_{Q(x,r)}\|_{{\mathcal M}^\varphi_q}
=\sup_{t>0}\varphi(t)\min(t^{-\frac{n}{q}},r^{-\frac{n}{q}}).
$$
In fact, simply observe that
\begin{align*}
\|\chi_{Q(x,r)}\|_{{\mathcal M}^\varphi_q}
&\equiv
\sup_{R>0}
\varphi(R)\left(
\frac{|Q(x,R) \cap Q(x,r)|}{|Q(x,R)|}\right)^{\frac1q}\\
&=\sup_{t>0}\varphi(t)\min(t^{-\frac{n}{q}},r^{-\frac{n}{q}}).
\end{align*}
\end{example}

The following $\min(1,q)$-triangle inequality holds:
\begin{lemma}\label{eq:141021-10}
Let $0<q<\infty$ and $\varphi:(0,\infty) \to (0,\infty)$
be a function.
Then
\[
\|f+g\|_{{\mathcal M}^\varphi_q}{}^{\min(1,q)}
\le
\|f\|_{{\mathcal M}^\varphi_q}{}^{\min(1,q)}
+
\|g\|_{{\mathcal M}^\varphi_q}{}^{\min(1,q)}
\]
for all $f,g \in {\mathcal M}^\varphi_q({\mathbb R}^n)$.
\end{lemma}

\begin{proof}
This is similar to classical Morrey spaces:
Use the $\min(1,q)$-triangle inequality for the Lebesgue space
$L^q({\mathbb R}^n)$.
\end{proof}

\begin{proposition}\label{prop:180525-1}
Let $0<q<\infty$, and let $\varphi:(0,\infty) \to [0,\infty)$
be a function satisfying $\varphi(t_0) \ne 0$ for some $t_0>0$.
Then ${\mathcal M}^\varphi_q({\mathbb R}^n)$
is a quasi-Banach space
and if $q \ge 1$,
then ${\mathcal M}^\varphi_q({\mathbb R}^n)$
is a Banach space.
\end{proposition}

\begin{proof}
The norm inequality follows from Lemma \ref{eq:141021-10}.
The proof of the completeness is a routine,
which we omit.
\end{proof}

Proposition \ref{prop:180525-1}
guarantees that the (quasi-)norm
of ${\mathcal M}^\varphi_q({\mathbb R}^n)$
is complete.
However, it may happen that
${\mathcal M}^\varphi_q({\mathbb R}^n)=\{0\}$.
We check that this extraordinary thing never happens
if $\varphi$ satisfies a mild condition.
\begin{proposition}{\rm \cite[Lemma 2.2(2)]{NNS15}}\label{prop:180611-1}
Let $0<q<\infty$, and let $\varphi:(0,\infty) \to [0,\infty)$
be a function satisfying $\varphi(t_0) \ne 0$ for some $t_0>0$.
Then the following are equivalent:
\begin{enumerate}
\item[$(a)$]
$L^\infty_{\rm c}({\mathbb R}^n)
\subset
{\mathcal M}^\varphi_q({\mathbb R}^n)$.
\item[$(b)$]
${\mathcal M}^\varphi_q({\mathbb R}^n) \ne \{0\}$.
\item[$(c)$] 
$\sup_{t>0}\varphi(t)\min(t^{-\frac{n}{q}},1)<\infty$.
\end{enumerate}
\end{proposition}

\begin{proof}
It is clear that $(a)$ implies $(b)$.

Assume $(b)$.
Then there exists
$f \in {\mathcal M}^\varphi_q({\mathbb R}^n) \setminus \{0\}$.
We may assume that $f(0) \ne 0$ and
that $x=0$ is the Lebesgue point of $|f|^q$.
Then since $f \in L^q_{\rm loc}({\mathbb R}^n)$,
by the Lebesgue differential theorem,
\[
\frac{1}{|Q(r)|}\int_{Q(r)}|f(y)|^q\,dy \sim 1
\]
for all $0<r<1$.
Here the implicit constants depend on $f$.
Thus,
\[
\sup_{0<t \le 1}\varphi(t)
\sim
\sup_{0<t \le 1}\varphi(t)
\left(\frac{1}{|Q(t)|}\int_{Q(t)}|f(y)|^q\,dy\right)^{\frac1q}
\le
\|f\|_{{\mathcal M}^\varphi_q}
<\infty.
\]
Here the implicit constants depend on $f$ again.
Meanwhile
\begin{align*}
\sup_{t \ge 1}t^{-\frac{n}{q}}\varphi(t)
&\lesssim
\sup_{t \ge 1}
\varphi(t)
\left(\frac{1}{|Q((t+1,0,0,\ldots,0),t)|}\int_{Q((t+1,0,0,\ldots,0),t)}|f(y)|^q\,dy\right)^{\frac1q}\\
&\le
\|f\|_{{\mathcal M}^\varphi_q}
<\infty.
\end{align*}
Thus
we conclude $(c)$.

Finally, if $(c)$ holds,
then $\chi_{Q(x,1)} \in
{\mathcal M}^\varphi_q({\mathbb R}^n)$
for any $x \in {\mathbb R}^n$.
Since ${\mathcal M}^\varphi_q({\mathbb R}^n)$
is a linear space,
 $\chi_{Q(x,m)} \in
{\mathcal M}^\varphi_q({\mathbb R}^n)$
for any $x \in {\mathbb R}^n$.
Since any function $g \in L^\infty_{\rm c}({\mathbb R}^n)$
admits the estimate of the form
$|g| \le N\chi_{Q(x,m)}$
for some $m,N \in {\mathbb N}$,
we conclude
$(a)$.
\end{proof}

Here we consider the case where
${\mathcal M}^\varphi_q({\mathbb R}^n)$
is close
to
${\mathcal M}^p_q({\mathbb R}^n)$
in a certain sense.
\begin{example}\label{example:180611-12}
Let $0<q \le p<\infty$ and ${\mathbb B}=(\beta_1,\beta_2) \in {\mathbb R}^2$.
We write
\[
\ell^{\mathbb B}(r)=\ell^{(\beta_1,\beta_2)}(r)
\equiv
\begin{cases}
(1+|\log r|)^{\beta_1}&(0< r \le 1),\\
(1+|\log r|)^{\beta_2}&(1< r<\infty).
\end{cases}
\]
\begin{enumerate}
\item
We set
\[
\varphi(t)=t^{\frac{n}{q}}\ell^{-{\mathbb B}}(t)
\quad (t>0).
\]
Note that
${\mathcal M}^\varphi_q({\mathbb R}^n) \ne \{0\}$
if and only if $\beta_2 \ge 0$.
Indeed,
according to Proposition \ref{prop:180611-1}
the case $\beta_2<0$ must be excluded
in order that
${\mathcal M}^\varphi_q({\mathbb R}^n) \ne \{0\}$.
Convesely if $\beta_2 \ge 0$,
then 
$L^\infty_{\rm c}({\mathbb R}^n) \subset
{\mathcal M}^\varphi_q({\mathbb R}^n)$.
\item
Let $0<q<p<\infty$.
We set
\[
\varphi(t)=t^{\frac{n}{p}}\ell^{-{\mathbb B}}(t)
\quad (t>0).
\]
Then
$L^\infty_{\rm c}({\mathbb R}^n) \subset
{\mathcal M}^\varphi_q({\mathbb R}^n)$.
\item
We set
\[
\varphi(t)=\ell^{-{\mathbb B}}(t)
\quad (t>0).
\]
For $a \in {\mathbb R}$,
we set
$f(x)=(1+|x|)^{-a}$,
$x \in {\mathbb R}^n$.
Let us see that
${\mathcal M}^\varphi_q({\mathbb R}^n)$
is close to $L^\infty({\mathbb R}^n)$
if $\beta_1 \ge 0$.
\begin{enumerate}
\item
By the Lebesgue differentiation theorem,
$\|f\|_\infty \le 0$ if $\beta_1<0$,
so that $f=0$ a.e..
So,
 if $\beta_1<0$,
then 
${\mathcal M}^\varphi_q({\mathbb R}^n)=\{0\}$.
\item
Let $\beta_1 \ge 0>\beta_2$.
Then
$f \in {\mathcal M}^\varphi_q({\mathbb R}^n)$
if and only if $a<0$.
\item
Let $\beta_1,\beta_2 \ge 0$.
Then
$f \in {\mathcal M}^\varphi_q({\mathbb R}^n)$
if and only if $a \le 0$.
\end{enumerate}
\end{enumerate}
\end{example}

We have the following scaling law:
\begin{lemma}{\rm \cite[Lemma 2.5]{NNS15}}\label{lem:150318-1}
Let $0<\eta,q<\infty$
and $\varphi:(0,\infty) \to (0,\infty)$
be a function.
Then $f\in {\mathcal M}^\varphi_q({\mathbb R}^n)$
if and only if 
$|f|^\eta \in {\mathcal M}^{\varphi^\eta}_{q/\eta}({\mathbb R}^n)$.
Furthemore in this case
$\||f|^\eta \|_{{\mathcal M}^{\varphi^\eta}_{q/\eta}}=\|f\|_{{\mathcal M}^\varphi_q}{}^\eta$.
\end{lemma}

\begin{proof}
We content ourselves with showing the equality
$\||f|^\eta \|_{{\mathcal M}^{\varphi^\eta}_{q/\eta}}=\|f\|_{{\mathcal M}^\varphi_q}{}^\eta$.
We calculate
$$
\||f|^\eta\|_{{\mathcal M}_{q/\eta}^{\varphi^\eta}}
= \sup\limits_{Q=Q(a,r)}
\frac{\varphi(r)^\eta \|f\|_{L^{q}(Q)}{}^\eta}{|Q|^{\frac{\eta}{q}}} 
=
\|f\|_{{\mathcal M}_q^\varphi}{}^{\eta}
$$
by using 
$
\||f|^\eta\|_{{\frac{q}{\eta}}}=\|f\|_{{q}}{}^\eta.
$
\end{proof}
The nesting property holds like classical Morrey spaces.
\begin{lemma}\label{lem:180307-1}
Let $0< q_1 \le q_2 <\infty$ and $\varphi:(0,\infty) \to (0,\infty)$
be a function.
Then
${\mathcal M}^\varphi_{q_2}({\mathbb R}^n)
\subset
{\mathcal M}^\varphi_{q_1}({\mathbb R}^n)$.
\end{lemma}

\begin{proof}
The proof of Lemma \ref{lem:180307-1}
hinges on the H\"{o}lder inequality.
Let $f \in L^0({\mathbb R}^n)$.
We write out the norms fully:
\begin{align}\label{eq:150709-1}
\| f \|_{{\mathcal M}^p_{q_1}}
&\equiv
\sup_{x \in {\mathbb R}^n, \, r>0}
\varphi(r)
\left(
\frac{1}{|B(x,r)|}
\int_{B(x,r)}|f(y)|^{q_1}\,dy
\right)^\frac{1}{q_1}\\
\label{eq:150709-2}
\| f \|_{{\mathcal M}^p_{q_0}}
&=
\sup_{x \in {\mathbb R}^n, \, r>0}
\varphi(r)
\left(
\frac{1}{|B(x,r)|}
\int_{B(x,r)}|f(y)|^{q_0}\,dy
\right)^\frac{1}{q_0}
\end{align}
By the H\"{o}lder inequality
(for probability measure), we have
\begin{equation}
\label{eq:150709-3}
\left(\frac{1}{|B(x,r)|}
\int_{B(x,r)}|f(y)|^{q_1}\,dy
\right)^\frac{1}{q_1}
\le
\left(\frac{1}{|B(x,r)|}
\int_{B(x,r)}|f(y)|^{q_0}\,dy
\right)^\frac{1}{q_0}.
\end{equation}
Thus by inserting inequality
(\ref{eq:150709-3}) 
into
(\ref{eq:150709-1})
and
(\ref{eq:150709-2}),
we obtain
${\mathcal M}^\varphi_{q_2}({\mathbb R}^n)
\subset
{\mathcal M}^\varphi_{q_1}({\mathbb R}^n)$.
\end{proof}

\begin{remark}
See \cite{Nakai94,Nakai08-1,SST11-2}
for the discussion of generalized Morrey spaces.
\end{remark}

\subsection{The class ${\mathcal G}_q$}

It will turn out demanding to consider all possible functions
$\varphi$.
We will single out good functions.
We answer the question 
of what functions are good.
\begin{definition}\label{defi:150824-33}
An increasing function $\varphi:(0,\infty) \to (0,\infty)$
is said to belong to the class ${\mathcal G}_q$ if
$
t^{-\frac{n}{q}} \varphi(t)
\ge
s^{-\frac{n}{q}} \varphi(s)
$
for all $0<t \le s<\infty$.
\index{G q@${\mathcal G}_q$}
\end{definition}
Remark that
${\mathcal G}_{q_1} \subset {\mathcal G}_{q_2}$
if $0<q_2<q_1<\infty$.

Here we list a series of the functions in ${\mathcal G}_q$.
\begin{example}\label{example:180611-13}
Let $0<q<\infty$.
\begin{enumerate}
\item
Let $u \in {\mathbb R}$,
and let $\varphi(t)=t^u$
for $t > 0$.
Then $\varphi$
belongs to ${\mathcal G}_q$
if and only if $0 \le u \le \frac{n}q$.
\item
Let $0<u \le \frac{n}q$, $L \gg 1$
and let $\varphi(t)=\dfrac{t^u}{\log(L+t)}$
for $t > 0$.
Then $\varphi$
belongs to ${\mathcal G}_q$.
\item
If $\varphi_1,\varphi_2 \in {\mathcal G}_q$,
then $\varphi_1+\varphi_2, \max(\varphi_1,\varphi_2)\in {\mathcal G}_q$.
\item
Let $0 \le u \ll 1$,
and let
$\varphi(t)=\dfrac{t^u}{\log(e+t)}$
for $t \ge 0$.
Then $\varphi \notin {\mathcal G}_q$
because $\varphi$ is not increasing.
\end{enumerate}
\end{example}

We start with a simple observation that 
any function in ${\mathcal G}_q$ enjoys the doubling property:
\begin{proposition}\label{prop:180611-2}
If $\varphi \in {\mathcal G}_q$ with $0<q<\infty$,
then
$\varphi(r) \le \varphi(2r) \le 2^{\frac{n}{q}}\varphi(r)$
for all $r>0$.
\end{proposition}

\begin{proof}
The left inequality is a consequence
of the fact that $\varphi$ is increasing,
while the right inequality follows from
the fact that $t \mapsto t^{-\frac{n}{q}}\varphi(t)$ is decreasing.
\end{proof}
The next proposition justifies
that we can naturally use the class ${\mathcal G}_q$.
\begin{proposition}{\rm \cite[p. 446]{Nakai00}}\label{prop:Nakai-norm}
Let $0<q<\infty$.
Then
for any function $\varphi:(0,\infty) \to [0,\infty)$
satisfying
$\displaystyle 0<\sup_{t>0}\varphi(t)\min(t^{-\frac{n}{q}},1)<\infty$,
there exists a function
$\varphi^*:(0,\infty) \to (0,\infty) \in {\mathcal G}_q$
such that
${\mathcal M}^\varphi_q({\mathbb R}^n)
={\mathcal M}^{\varphi^*}_q({\mathbb R}^n)$
with equivalence of norms.
\end{proposition}

\begin{proof}
Let $f$ be a measurable function.

We claim that
\begin{equation}\label{eq:180814-1}
\frac1{|Q|}\int_{Q}|f(x)|^q\,dx
\le 2^n
\sup_{Q' \in {\mathcal Q}:\,Q' \subset Q,\ell(Q')=t'}
\frac1{|Q'|}\int_{Q'}|f(x)|^q\,dx
\end{equation}
for any cube $Q \in {\mathcal Q}$ and any positive number $t' \le \ell(Q)$.
In fact, let
\[
N=\left[1+\frac{\ell(Q)}{t'}\right] \in [1,\infty).
\]
We write $\displaystyle Q=\prod_{j=1}^n [a_j,a_j+\ell(Q)]$,
so that
$(a_1,a_2,\ldots,a_n)$ is the \lq \lq bottom" corner of $Q$.
Define
\[
\delta_{m_j}=\begin{cases}
0&m_j<N,\\
N t'-\ell(Q)&m_j=N
\end{cases}
\]
for $m_j \in\{1,2,\ldots,N\}$ and
\[
Q_m=\prod_{j=1}^n \left[a_j+(m_j-1)t'-\delta_{m_j},a_j+m_j t'-\delta_{m_j}\right]
\]
for $m =(m_1,m_2,\ldots,m_n)\in \{1,2,\ldots,N\}^n$.
Then
\[
1 \le \sum_{m =(m_1,m_2,\ldots,m_n)\in \{1,2,\ldots,N\}^n}\chi_{Q_m} \le 2^n
\]
almost everywhere.
As a consequence,
\begin{align*}
\frac1{|Q|}\int_{Q}|f(x)|^q\,dx
&\le
\sum_{m =(m_1,m_2,\ldots,m_n)\in \{1,2,\ldots,N\}^n}
\frac1{|Q|}\int_{Q_m}|f(x)|^q\,dx\\
&\le
\sum_{m =(m_1,m_2,\ldots,m_n)\in \{1,2,\ldots,N\}^n}
\frac{t'{}^n}{|Q|}
\frac1{|Q_m|}\int_{Q_m}|f(x)|^q\,dx.
\end{align*}
Since
\[
\frac{t'N}{\ell(Q)}
=
\frac{t'}{\ell(Q)}
\left[1+\frac{\ell(Q)}{t'}\right] 
\le
\frac{t'}{\ell(Q)}
\left(1+\frac{\ell(Q)}{t'}\right)
\le 
\frac{t'}{\ell(Q)}+1 \le 2,
\]
it follows that
\[
\frac1{|Q|}\int_{Q}|f(x)|^q\,dx
\le 2^n
\max_{m =(m_1,m_2,\ldots,m_n)\in \{1,2,\ldots,N\}^n}
\frac1{|Q_m|}\int_{Q_m}|f(x)|^q\,dx.
\]
Thus,
(\ref{eq:180814-1}) follows.

If we let
$$
\varphi_1(t') \equiv \inf_{t \ge  t'}\varphi(t)
$$
then it is easy to see that
$\varphi_1(t)\le \varphi(t)$ for all $t>0$ and hence
$
\|f\|_{{\mathcal M}^{\varphi_1}_q}
\le
\|f\|_{{\mathcal M}^\varphi_q}
$.
Furthermore,
\begin{align*}
\lefteqn{
\sup_{Q \in {\mathcal Q}}\varphi(\ell(Q))
\left(\frac1{|Q|}\int_{Q}|f(x)|^q\,dx\right)^{\frac1q}
}\\
&\le 2^{\frac{n}{q}}
\sup_{Q \in {\mathcal Q}}\inf_{t' \in (0,\ell(Q)]}
\sup_{Q' \in {\mathcal Q}:\,Q' \subset Q,\ell(Q')=t'}\varphi(\ell(Q))
\left(\frac1{|Q'|}\int_{Q'}|f(x)|^q\,dx\right)^{\frac1q}\\
&\le 2^{\frac{n}{q}}
\sup_{Q \in {\mathcal Q}}\inf_{t' \in (0,\ell(Q)]}
\sup_{Q' \in {\mathcal Q}:\, \ell(Q')=t'}\varphi(\ell(Q))
\left(\frac1{|Q'|}\int_{Q'}|f(x)|^q\,dx\right)^{\frac1q}\\
&= 2^{\frac{n}{q}}
\sup_{Q \in {\mathcal Q}}\inf_{t' \in (0,\ell(Q)]}
\sup_{Q' \in {\mathcal Q}:\, \ell(Q')=t'}\varphi_1(t')
\left(\frac1{|Q'|}\int_{Q'}|f(x)|^q\,dx\right)^{\frac1q}\\
&= 2^{\frac{n}{q}}
\sup_{Q \in {\mathcal Q}}\inf_{t' \in (0,\ell(Q)]}
\sup_{Q' \in {\mathcal Q}:\, \ell(Q')=t'}\varphi_1(\ell(Q'))
\left(\frac1{|Q'|}\int_{Q'}|f(x)|^q\,dx\right)^{\frac1q}\\
&\le 2^{\frac{n}{q}}
\|f\|_{{\mathcal M}^{\varphi_1}_q}.
\end{align*}
Thus it follows that
$$
\|f\|_{{\mathcal M}^{\varphi_1}_q}
\le
\|f\|_{{\mathcal M}^\varphi_q}
\le
2^{\frac{n}{q}}
\|f\|_{{\mathcal M}^{\varphi_1}_q}.
$$
Next, if we let
$$
\varphi^*(t)
\equiv
t^{\frac{n}{q}}\sup_{t' \ge t}\varphi_1(t')t'{}^{-\frac{n}{q}}
=
\sup_{s \ge 1}\varphi_1(s t)s^{-\frac{n}{q}}
\quad (t>0),
$$
then
$\|f\|_{{\mathcal M}^\varphi_q}
=
\|f\|_{{\mathcal M}^{\varphi^*}_q}
$.
In fact,
since $\varphi_1$ is increasing,
$\varphi^*$ is increasing.
From the definitoin of $\varphi^*$,
$\varphi^*(t) \ge \varphi_1(t)$
for all $t>0$ and thus
$\|f\|_{{\mathcal M}^{\varphi_1}_q}
\le
\|f\|_{{\mathcal M}^{\varphi^*}_q}
$.
On the other hand,
for any $r>0$ and $\varepsilon>0$,
we can find $r' \ge r$ such that
\[
\varphi^*(r)
\le
(1+\varepsilon)r^{\frac{n}{q}}\varphi_1(r')r'{}^{-\frac{n}{q}}.
\]
Thus
for any cube $Q(x,r)$,
\begin{align*}
\lefteqn{
\varphi^*(r)
\left(\frac{1}{|Q(x,r)|}\int_{Q(x,r)}|f(y)|^q\,dy\right)^{\frac1q}
}\\
&\le
(1+\varepsilon)\varphi_1(r')
\left(\frac{1}{|Q(x,r')|}\int_{Q(x,r)}|f(y)|^q\,dy\right)^{\frac1q}\\
&\le
(1+\varepsilon)\varphi_1(r')
\left(\frac{1}{|Q(x,r')|}\int_{Q(x,r')}|f(y)|^q\,dy\right)^{\frac1q}\\
&\le
(1+\varepsilon)\|f\|_{{\mathcal M}^{\varphi_1}_q}.
\end{align*}
Taking the supremum over $x$ and $r$,
we have
$\|f\|_{{\mathcal M}^\varphi_q}
\le (1+\varepsilon)
\|f\|_{{\mathcal M}^{\varphi^*}_q}
$.
Since $\varepsilon>0$ is arbitrary,
we conclude
$\|f\|_{{\mathcal M}^\varphi_q}
\le 
\|f\|_{{\mathcal M}^{\varphi^*}_q}
$.
\end{proof}
In view of Proposition \ref{prop:Nakai-norm},
it follows that we can always suppose that
$\varphi \in {\mathcal G}_q$.

\begin{remark}
Some authors suppose that there exists $\delta>0$ such that
\begin{equation}\label{eq:140812-1}
\varphi(r) \le \delta \quad (0<r \le 1)
\end{equation}
and that
\begin{equation}\label{eq:140812-2}
r^{\frac{n}{p}}\varphi(r)>\delta \quad (1<r<\infty).
\end{equation}
\end{remark}

As the next proposition shows,
${\mathcal G}_q$ is a good class
in addition to the nice property
that it naturally arises in generalized Morrey spaces.
\begin{proposition}\label{prop:131029-2}
Let $\varphi \in {\mathcal G}_q$ with $0<q<\infty$.
Then we can find a continuous function
$\varphi^* \in {\mathcal G}_q$ such that
$\varphi^*$ is strictly increasing
and that $\varphi \sim \varphi^*$.
\end{proposition}

\begin{proof}
We look for $\varphi^*$ in a couple of steps.
\begin{enumerate}
\item
Consider
\[
\varphi_0(t) \equiv
t^{\frac{n}{q}}\int_t^{2t} \varphi(s)
\frac{\,ds}{s^{\frac{n}{q}+1}}
=\int_1^2 \varphi(t s)
\frac{\,ds}{s^{\frac{n}{q}+1}}
\quad (t>0).
\]
Since $\varphi$ is increasing,
the function $\varphi_0$ is increasing.
Also, by the fact that $\varphi$ is doubling,
we see that $\varphi_0$ and $\varphi$ are equivalent.
Thus, we may assume that $\varphi$ is continuous.
\item
We define
\[
\psi(t)=\varphi(t)\chi_{(0,1]}(t)+
\left(\frac{1-e^{-t}}{1-e^{-1}}\right)^{\frac{n}{q}}\varphi(1)\chi_{(1,\infty)}(t)
\quad (t>0),
\]
Since we can check that
$t \in (0,1] \mapsto t^{-\frac{n}{q}}\psi(t)$
and
$t \in [1,\infty)\mapsto t^{-\frac{n}{q}}\psi(t)$
is both decreasing,
$t \in (0,\infty) \mapsto t^{-\frac{n}{q}}\psi(t)$
is decreasing.
Likewise we can check that $\psi$ is increasing.
Thus,
$\psi \in {\mathcal G}_q$.
Since
$\psi+\varphi \simeq \varphi$,
we can assume that $\varphi$ is strictly increasing
in $(1,\infty)$ and that $\varphi$ is continuous.
\item
Finally, we consider
\[
\varphi^*(t)=\sum_{k=0}^\infty 
\frac{\varphi(2^k t)}{2^{k N}} \quad
(t>0),
\]
where
$\displaystyle N=\frac{n}{q}+1$.
Then it is easy to check that $\varphi^*$ is
equivalent to $\varphi$ since each term is dominated by
$2^{-k}\varphi$.
Likewise since $\varphi \in {\mathcal G}_q$,
$\varphi^* \in {\mathcal G}_q$.
Finally since $\varphi$ is strictly increasing in $(1,\infty)$,
the function $\varphi(2^k t)$ is strictly increasing
on $(2^{-k},\infty)$
for any $k \in {\mathbb N}$.
Since $k \in {\mathbb N}$ is arbitrary,
$\varphi^*$ is strictly increasing.
\end{enumerate}
\end{proof}

Thus,
we are led to the following definition:
\begin{definition}\label{defi:180611-12}
The class $W$ stands for the set of all continuous functions
$\varphi:(0,\infty) \to (0,\infty)$.
That is,
$W=C((0,\infty),(0,\infty))$.
\end{definition}

We apply what we have obtained to small Morrey spaces.
\begin{example}\label{example:180611-14}
Let $0<q<\infty$.
Let us see
how we modify 
$\psi$ in ${\mathcal M}^{\psi}_q({\mathbb R}^n)$
to obtain the equivalent space
${\mathcal M}^{\varphi}_q({\mathbb R}^n)$
with
$\varphi \in {\mathcal G}_q$.
\begin{enumerate}
\item
Let $\varphi(t)=\max(t^{\frac{n}{p}},1)$
and $\psi(t)=t^{\frac{n}{p}}\chi_{(0,1]}(t)$
for $t>0$.
Then
with the equivalence of norms
$m^p_q({\mathbb R}^n)
={\mathcal M}^{\psi}_q({\mathbb R}^n)
={\mathcal M}^{\varphi}_q({\mathbb R}^n)$.
\item
Let $\varphi(t)=\max(t^a,1)$ with $a \ge \frac{n}{q}$
and $\psi(t)=t^{\frac{n}{q}}\chi_{(0,1]}(t)$
for $t>0$.
Then
with the equivalence of norms
$L^q_{\rm uloc}({\mathbb R}^n)
={\mathcal M}^{\psi}_q({\mathbb R}^n)
={\mathcal M}^{\varphi}_q({\mathbb R}^n)$.
\end{enumerate}
Note that
$\varphi \in {\mathcal G}_q \cap W$
but that
$\psi \in {\mathcal G}_q \setminus W$
in both cases.
\end{example}

So far we have shown that
we may assume that $\varphi \in {\mathcal G}_q$.
As a result, we may assume that $\varphi$ is doubling.
This observation makes the definition of the norm
$\|\cdot\|_{{\mathcal M}^\varphi_q}$ more flexible.
\begin{remark}
In (\ref{eq:131115-111}),
cubes can be replaced with balls;
an equivalent norms will be obtained.
Precisely use the norm given by
\[
\|f\|_{{\mathcal M}^\varphi_q}
\equiv
\sup_{x \in {\mathbb R}^n, r>0}
\varphi(r)\left(
\frac{1}{|B(x,r)|}
\int_{B(x,r)}|f(y)|^q\,dy\right)^{\frac1q}
\]
to go through the same argument
given for the norm defined by means of cubes.
\end{remark}

Related to this definition,
we give some definitions related to the class ${\mathcal G}_q$.
Although we show that it is sufficient
to limit ourselves to ${\mathcal G}_q$,
we still feel that this class is too narrow as
the function of $\varphi(t)=t\log(e+t^{-1})$ shows.
So, it is convenient to relax the condition on $\varphi$.
The following definition will serve to this purpose.
\begin{definition}\label{defi:180310-110}
Let $\ell>0$.
\begin{enumerate}
\item
A function $\varphi:(0,\ell] \to (0,\infty)$ is said to be almost decreasing
if 
$\varphi(s) \lesssim \varphi(t)$
for all $0 \le t<s \le \ell$.
\item
A function $\varphi:(0,\ell] \to (0,\infty)$ is said to be almost increasing
if 
$\varphi(s) \gtrsim \varphi(t)$
for all $0 \le t<s \le \ell$.
\item
A function $\varphi:(0,\infty) \to (0,\infty)$ is said to be almost decreasing
if 
$\varphi(s) \lesssim \varphi(t)$
for all $0 \le t<s < \infty$.
\item
A function $\varphi:(0,\infty) \to (0,\infty)$ is said to be almost increasing
if 
$\varphi(s) \gtrsim \varphi(t)$
for all $0 \le t<s < \infty$.
\end{enumerate}
The implicit constants
in these inequality are called the almost
decreasing/increasing constants.
\index{almost decreasing@almost decreasing}
\index{almost increasing@almost increasing}
\end{definition}

\begin{example}\label{example:180310-114}\
\begin{enumerate}
\item
Let $a,b$ be real parameters with $b \ne 0$.
Let $\varphi_{a,b}(t)=t^a(\log(e+t))^b$ for $t>0$.
Then $\varphi_{a,b}$ is almost increasing for any $a>0$.
If $a< 0$,
then $\varphi_{a,b}$ is almost decreasing.
\item
The function
$\varphi(t)=t+\sin 2t$, $t>0$ is almost increasing
but not increasing.
\item
The function
$\varphi(t)=1+e^t(1-\cos t)$, $t>0$ is not almost increasing
because $\varphi(2\pi m)=1$ and $\varphi(2\pi m+\pi)=1+2e^{2\pi m+\pi}$
for all $m \in {\mathbb N}$.
\item
Let $0<p,q<\infty$ and $\beta_1,\beta_2 \in {\mathbb R}$.
We write
\[
\ell^{\mathbb B}(r)
\equiv
\begin{cases}
(1+|\log r|)^{\beta_1}&(0< r \le 1),\\
(1+|\log r|)^{\beta_2}&(1< r<\infty).
\end{cases}
\]
We set
\[
\varphi(t)=t^{\frac{n}{p}}\ell^{-{\mathbb B}}(t)
\quad (t>0)
\]
as we did in Example \ref{example:180611-12}.
We observe that $\varphi$ is almost increasing
for all $p>0$
and
$\beta_1,\beta_2 \in {\mathbb R}$.
We also note that $t \mapsto \varphi(t)t^{-\frac{n}{q}}$
is almost decreasing
if and only if
$p=q$ and $\beta_1 \le 0 \le \beta_2$
or
$p>q$.
Note that
$\varphi$ is an equivalent to a function
$\psi \in {\mathcal G}_q$
if and only if
either $p=q$, $\beta_1 \le 0 \le \beta_2$ or $p>q$
since we can neglect the effect of "$\log$".
\item
Let $E \subset (0,\infty)$
be a non-Lebegsgue measurable set.
Then
$\varphi=1+\chi_E$
is almost increasing and almost decreasing.
Note that $\varphi$ is not measurable.
\item
A simple but still standard example is as follows:
Let $m \in {\mathbb N}$ and define
\[
\varphi(t)\equiv \frac{t^{\frac{n}{q}}}{l_m(t)}
\quad (t>0),
\]
where $l_m(t)$ is given inductively by:
\[
l_0(t)\equiv t, \quad
l_m(t)\equiv \log(3+l_{m-1}(t)) \quad
(m=1,2,\ldots)
\]
for $t>0$.
The for all $0<q<\infty$ and $m \in {\mathbb N}$,
$\varphi \in {\mathcal G}_q$.
\end{enumerate}
\end{example}

As the following lemma shows,
we can always replace an almost increasing function
with an increasing function.
\begin{lemma}
If a function
$\varphi:(0,\infty) \to (0,\infty)$
is almost increasing with the almost increasing constant $C_0>0$,
then
there exists an increasing function
$\psi:(0,\infty) \to (0,\infty)$
such that
$\varphi(t) \le \psi(t) \le 
C_0\varphi(t)$
for all $t>0$.
\end{lemma}

\begin{proof}
Simply set
$
\psi(t)=\sup_{0<s \le t}\varphi(s), t>0.
$
\end{proof}

Having set down the condition of $\varphi$,
we move on to the norm estimate of the function.
The lemma below gives an estimate for the norm of $\chi_{B(R)}$ in
${\mathcal M}^\varphi_q({{\mathbb R}^n})$.

\begin{lemma}\label{lem7}
{\rm \cite[Lemma 4.1]{KoMi06}}
Let $0<q<\infty$ and $\varphi \in {\mathcal G}_q$.
Then
$
\|\chi_{Q(x,R)}\|_{{\mathcal M}^\varphi_q} 
=\varphi(R)
$
for all $R>0$.
\end{lemma}

\begin{proof}
One method is to
reexamine Example \ref{ex:180525-1}.
Here we give a
direct proof.
It is easy to see that
$
\|\chi_{Q(x,R)}\|_{{\mathcal M}^\varphi_q} 
\ge \varphi(R)$.
To prove the opposite inequality 
we consider
\[
\varphi(r)\frac{1}{|Q(y,r)|^{\frac{1}{q}}}\|\chi_{Q(y,r) \cap Q(x,R)}\|_{q}
\]
for any cube $Q=Q(y,r)$.
When $R \le r$,
then
\[
\varphi(r)\frac{\|\chi_{Q(y,r) \cap Q(x,R)}\|_{q}}{|Q(y,r)|^{\frac{1}{q}}}
\le
\varphi(r)\frac{\|\chi_{Q(x,R)}\|_{q}}{|Q(y,r)|^{\frac{1}{q}}}
\le
\varphi(R)\frac{\|\chi_{Q(x,R)}\|_{q}}{|Q(y,R)|^{\frac{1}{q}}}
=
\varphi(R).
\]
since $\varphi \in {\mathcal G}_q$.
When $R>r$,
then
\[
\varphi(r)\frac{\|\chi_{Q(y,r) \cap Q(x,R)}\|_{q}}{|Q(y,r)|^{\frac{1}{q}}}
\le
\varphi(r)\frac{\|\chi_{Q(y,r)}\|_{q}}{|Q(y,r)|^{\frac{1}{q}}}
=
\varphi(r)
\le
\varphi(R)
\]
again by virtue of the fact that $\varphi \in {\mathcal G}_q$.
\end{proof}

A direct consequence of the above quantitative information is:
\begin{corollary}{\rm \cite[Corollary 2.3]{NNS15}}\label{cor:140820-1}
Let $0<q<\infty$ and $\varphi:(0,\infty) \to (0,\infty)$
be a function in the class ${\mathcal G}_q$.
If $N_0>n/q$,
then
$(1+|\cdot|)^{-N_0}\in {\mathcal M}^\varphi_q({\mathbb R}^n)$.
\end{corollary}

\begin{proof}
Since $\varphi \in {\mathcal G}_q$,
we have $\varphi(t)t^{-\frac{n}{q}} \le \varphi(1)$
for all $t \ge 1$.
we have
\[
\|(1+|\cdot|)^{-N_0}\|_{{\mathcal M}^\varphi_q}
\lesssim
\sum_{j=1}^\infty
\frac{\|\chi_{Q(j)}\|_{{\mathcal M}^\varphi_q}}{\max(1,j-1)^{N_0}}
\le
\sum_{j=1}^\infty
\frac{\varphi(j)}{\max(1,j-1)^{N_0}}
<\infty.
\]
Here for the second inequality
we invoked Lemma \ref{lem7}.
\end{proof}

Now we consider the role of $\varphi$.
We did not tolerate the case where $p=\infty$
when we define
${\mathcal M}^p_q({\mathbb R}^n)$.
If we define ${\mathcal M}^\infty_q({\mathbb R}^n)$
similar to ${\mathcal M}^p_q({\mathbb R}^n)$,
then we have ${\mathcal M}^\infty_q({\mathbb R}^n)=L^\infty({\mathbb R}^n)$
by the Lebesgue differentiation theorem.
The next theorem concerns a situation close to this.

\begin{theorem}{\rm\cite[Proposition 3.3]{Nakai08-1}}\label{thm:150501-1}
Let $0< q<\infty$ and $\varphi \in {\mathcal G}_q$.
Then, the following are equivalent{\rm:}
\begin{enumerate}
\item
$\inf_{t>0} \varphi(t)>0$,
\item
${\mathcal M}^\varphi_q({\mathbb R}^n) \subset L^\infty({\mathbb R}^n)$.
\end{enumerate}
If these conditions are satisfied,
then
$
{\mathcal M}^\varphi_q({\mathbb R}^n)=L^\infty({\mathbb R}^n) \cap
{\mathcal M}^{\varphi-\inf_{t>0} \varphi(t)}_q({\mathbb R}^n)
$
with equivalence of norms.
\end{theorem}

\begin{proof}
If $\inf_{t>0} \varphi(t)>0$,
then by the Lebesgue differentiation theorem we get
\begin{align*}
|f(x)|^q
&=\lim_{r \downarrow 0}
\frac{1}{|B(x,r)|}\int_{B(x,r)}|f(y)|^q\,dy\\
&\le \frac{1}{\inf_{t>0} \varphi(t)^q}
\lim_{r \downarrow 0}
\frac{\varphi(r)^q}{|B(x,r)|}\int_{B(x,r)}|f(y)|\,dy\\
&\le
\frac{1}{\inf_{t>0} \varphi(t)^q}\|f\|_{{\mathcal M}^\varphi_q}{}^q
\end{align*}
for all $f \in {\mathcal M}^\varphi_q({\mathbb R}^n)$
and all Lebesgue points $x$ of $f$.
Therefore, $f \in L^\infty$.

Assume that ${\mathcal M}^\varphi_q({\mathbb R}^n) \subset L^\infty({\mathbb R}^n)$.
Then by the closed graph theorem,
$\|f\|_{\infty} \lesssim \|f\|_{{\mathcal M}^\varphi_q}$
for all $f \in {\mathcal M}^\varphi_q({\mathbb R}^n)$.
If we choose $f=\chi_{B(a,r)}$,
then we have 
$1 \lesssim \varphi(r)$.
This shows that $\inf_{t>0} \varphi(t)>0$.

Finally, by taking $\varphi_1=\inf_{t>0} \varphi(t) $ and $\varphi_2=\varphi-\varphi_1$, we obtain
$
{\mathcal M}^\varphi_q=L^\infty({\mathbb R}^n) \cap
{\mathcal M}^{\varphi-\inf_{t>0} \varphi(t)}_q({\mathbb R}^n)
$
with equivalence of norms
from the general formula
$
{\mathcal M}^{\varphi_1+\varphi_2}_q({\mathbb R}^n)
=
{\mathcal M}^{\varphi_1}_q({\mathbb R}^n)
\cap
{\mathcal M}^{\varphi_2}_q({\mathbb R}^n)
$
with equivalence of norms.
\end{proof}
We also investigate the reverse inclusion
to Theorem \ref{thm:150501-1}.
\begin{theorem}{\rm\cite[Proposition 3.3]{Nakai08-1}}\label{thm:150501-11}
Let $0< q<\infty$ and $\varphi \in {\mathcal G}_q$.
Then, the following are equivalent{\rm:}
\begin{enumerate}
\item
$\sup_{t>0} \varphi(t)<\infty$,
\item
${\mathcal M}^\varphi_q({\mathbb R}^n) \supset L^\infty({\mathbb R}^n)$.
\end{enumerate}
\end{theorem}

\begin{proof}
If $\sup_{t>0} \varphi(t)<\infty$,
then by letting $\psi(t)\equiv \sup_{s>0} \varphi(s), t>0$
we have
$
L^\infty({\mathbb R}^n)
=
{\mathcal M}^{\psi}_q({\mathbb R}^n)
$
with equivalence of norms.
Since
$\varphi \le \psi$,
we have
$
{\mathcal M}^{\psi}_q({\mathbb R}^n)
\subset
{\mathcal M}^{\varphi}_q({\mathbb R}^n)
$.
Thus
$L^\infty({\mathbb R}^n)
\subset
{\mathcal M}^{\varphi}_q({\mathbb R}^n)
$.

Conversely,
if 
$L^\infty({\mathbb R}^n)
\subset
{\mathcal M}^{\varphi}_q({\mathbb R}^n)
$,
then by the closed graph theorem
the embedding norm is finite and hence
$\varphi(r)=\|\chi_{Q(r)}\|_{{\mathcal M}^{\varphi}_q}
\lesssim
\|\chi_{Q(r)}\|_{L^\infty}=1$
for all $r>0$.
Then $\sup\varphi<\infty$.
\end{proof}

By combining these two theorems,
we obtain the following result.
\begin{corollary}{\rm\cite[Proposition 3.3]{Nakai08-1}}\label{cor:180521-1}
Let $0< q<\infty$ and $\varphi \in {\mathcal G}_q$.
Then, the following are equivalent{\rm:}
\begin{enumerate}
\item
$\log \varphi \in L^\infty(0,\infty)$, i.e.,
$0<\inf_{t>0} \varphi(t) \le \sup_{t>0} \varphi(t)<\infty$.
\item
$L^\infty({\mathbb R}^n)
=
{\mathcal M}^{\varphi}_q({\mathbb R}^n)
$,
\end{enumerate}
\end{corollary}

So, if $\log \varphi$ grows or decays slowly we can say that 
${\mathcal M}^{\varphi}_q({\mathbb R}^n)$
is close to $L^\infty({\mathbb R}^n)$.

The next example shows that
when the support of the functions are torn apart,
the norm does not increase even in the case
of generalized Morrey spaces.
\begin{example}\label{prop:140520-1}
Let $0<q<\infty$ and $\varphi \in {\mathcal G}_q$.
Suppose that we have
a collection of cubes $\{Q_j\}_{j=1}^\infty=\{Q(a_j,r_j)\}_{j=1}^\infty$
such that
$\{3Q_j\}_{j=1}^\infty=\{Q(a_j,3r_j)\}_{j=1}^\infty$
is disjoint and 
is contained in $Q=Q(a_0,r_0)$.
Let $\{f_j\}_{j=1}^\infty$ be a collection
of functions in ${\mathcal M}^\varphi_q({\mathbb R}^n)$
satisfying
\begin{gather}
\label{eq:180311-41}
\|f_j\|_{q} \le \varphi(r_0)^{-1}r_j{}^{\frac{n}{q}}, \quad
{\rm supp}(f_j) \subset Q_j,\\
\label{eq:180311-42}
\|f_j\|_{{\mathcal M}^\varphi_q} \le 1
\end{gather}
for each $j \in {\mathbb N}$.
Then
$\displaystyle
f\equiv\sum_{j=1}^\infty f_j \in {\mathcal M}^\varphi_q({\mathbb R}^n).
$
Since $f$ is supported in $Q$,
we have only to consider cubes
contained in $3Q$;
if a cube $Q'$ intersects both $Q$ and ${\mathbb R}^n \setminus 3Q$,
then its triple $3Q'$ engulfs $Q$ and the radius of $Q$ is smaller than
that of $Q'$.
Thus we have
\[
\|f\|_{{\mathcal M}^\varphi_q}
\lesssim
\sup_{x \in {\mathbb R}^n, r>0, \, Q(x,r) \subset 3Q}
\varphi(r)\left(
\frac{1}{|Q(x,r)|}
\int_{Q(x,r)}|f(y)|^q\,dy\right)^{\frac1q}.
\]
We fix a cube $Q(x,r)$ and estimate
\[
\varphi(r)\left(
\frac{1}{|Q(x,r)|}
\int_{Q(x,r)}|f(y)|^q\,dy\right)^{\frac1q}.
\]
We let
\begin{align*}
J_1&\equiv
\{j \in {\mathbb N}\,:\,Q(x,r) \cap Q(a_j,r_j) \ne \emptyset,\,
r_j \le r\},
\\
J_2&\equiv
\{j \in {\mathbb N}\,:\,Q(x,r) \cap Q(a_j,r_j) \ne \emptyset,\,
r_j > r\}.
\end{align*}
Accordingly
we set
\[
{\rm I}_i\equiv
\varphi(r)
\left(
\frac{1}{|Q(x,r)|}
\int_{Q(x,r)}\left|\sum_{j \in J_i}f_j(y)\right|^q\,dy\right)^{\frac1q}
\]
for $i=1,2$.

As for $I_1$, we use 
(\ref{eq:180311-41})
and the fact that $\{Q(a_j,r_j)\}_{j=1}^\infty$ is disjoint:
\[
{\rm I}_1
\le
\varphi(r)\left(
\frac{1}{|Q(x,r)|}
\sum_{j \in J_1}\varphi(r_0)^{-q} r_j{}^n\right)^{\frac1q}
\le
\left(
\frac{1}{|Q(x,r)|}
\sum_{j \in J_1}r_j{}^n\right)^{\frac1q}
\lesssim 1.
\]

As for $I_2$,
we have only to consider only one summand,
since $Q(a_j,3r_j)$ engulfs $Q(x,r)$ if $j \in H_2$:
we can deduce
${\rm I}_2 \le
1$
using 
$(\ref{eq:180311-42})$.
\end{example}

As an application of Example \ref{prop:140520-1},
we present another example.
Denote by $[t]$ the integer part of $t \in {\mathbb R}$.
For positive sequences $\{A_k\}_{k=1}^\infty$ and $\{B_k\}_{k=1}^\infty$
\lq \lq $A_k \sim B_k$ as $k\to \infty$" means that
$\{\log(A_k/B_k)\}_{k=1}^\infty$ is a bounded sequence.

\begin{example}\cite[Lemma 4.1]{Nakai08-1}\label{example:180310-116}
Let
$\varphi \in {\mathcal G}_{q}$ with $0<q<\infty$.
We fix $a \in {\mathbb R}^n$.
Let $\{s_k\}_{k=1}^\infty \subset (0,1]$ be a sequence
which decreases to $0$.

Keeping in mind
\[
\inf_{0<t \le 1}\varphi(t)^{q/n}t^{-1}
=
\varphi(1)^{q/n}, \quad
\inf_{0<t<1}\varphi(t)^{-q/n}=\varphi(1)^{-q/n},
\]
we let
\begin{align*}
\ell_k&\equiv [1+\varphi(s_k)^{q/n}s_k{}^{-1}]\,
(\sim \varphi(s_k)^{q/n}s_k{}^{-1})\\
m_k&\equiv[1+\varphi(s_k)^{-q/n}]\,
(\sim \varphi(s_k)^{-q/n}).
\end{align*}
Then
\begin{equation}\label{eq:180311-31}
\ell_k s_k m_k \sim \varphi(s_k)^{\frac{q}{n}}s_k{}^{-1}s_k\varphi(s_k)^{-\frac{q}{n}}=1,
\quad s_k \downarrow 0
\end{equation} 
as $k \to \infty$.
Hence
\begin{equation}\label{eq:180311-51}
\varphi((\ell_k m_k)^{-1})^{-q}m_k{}^{-n}
\sim
\varphi(s_k)^{-q}s_k{}^n\ell_k{}^n
\sim 1
\end{equation}
as $k \to \infty$.

Devide equally $a+[0,1]^n$ into $\ell_k{}^n$ cubes
to have a non-overlapping collection
$$\{b_{k,j}+[0,\ell_k{}^{-1}]^n\}_{j=1,2,\ldots,\ell_k{}^n}.$$
Furthermore,
we devide $[0,\ell_k{}^{-1}]^n$ equally into $m_k{}^n$ cubes
to obtain a collection
$$\{e_{k,j}+[0,m_k{}^{-1}]^n\}_{j=1,2,\ldots,m_k{}^n}$$
of cubes.
Then
$$
b_{k,j}+[0,\ell_k{}^{-1}]^n
=
\bigcup_{i=1}^{m_k{}^n}
(b_{k,j}+e_{k,i}+[0,(\ell_k m_k)^{-1}]^n).
$$
If we set
\[
g_{k,i,j,a}=
\frac{1}{\varphi((\ell_k m_k)^{-1})}
\chi_{b_{k,j}+e_{k,i}+[0,(\ell_k m_k)^{-1}]^n}.
\]
and
\begin{equation}\label{eq:140520-10}
f_{k,i,a}=f_{k,i}=
\sum_{j=1}^{\ell_k{}^n}
\frac{\chi_{b_{k,j}+e_{k,i}+[0,(\ell_k m_k)^{-1}]^n}}{\varphi((\ell_k m_k)^{-1})}=
\sum_{j=1}^{\ell_k{}^n}
g_{k,i,j,a}.
\end{equation}
Then, 
each 
$f_{k,i,a}$ is supported in $a+[0,1]^n$
and
each
$g_{k,i,j,a}$ is supported in $b_{k,j}+[0,\ell_k{}^{-1}]^n$.
We will show  that
$\{f_{k,i,a}\}_{a \in {\mathbb R}^n, k \in {\mathbb N}, \, i=1,2,\ldots,m_k{}^n}$
forms a bounded set in ${\mathcal M}^\varphi_{q}({\mathbb R}^n)$.
This is achieved
by verifying the assumption of
Example \ref{prop:140520-1}.
Let
$F_1,F_2,\ldots,F_{3^n}$
be a partition of $\{1,2,\ldots,\ell_k{}^n\}$
such that 
$\{b_{k,j}+[-\ell_k{}^{-1},2\ell_k{}^{-1}]^n\}_{j \in F_{l'}{}^n}$
is not overlapping
for each $l'=1,2,\ldots,3^n$.
To check this we need to verify
(\ref{eq:180311-41}) and (\ref{eq:180311-42}).

From Lemma \ref{lem7},
we have (\ref{eq:180311-41}).
Meanwhile,
\[
\|g_{k,i,j,a}\|_{q}
=\varphi((\ell_k m_k)^{-1})(\ell_k m_k)^{-\frac{n}{q}}
\sim \ell_k{}^{-\frac{n}{q}}.
\]
Thus
for each $l'=1,2,\ldots,3^n$
\[
\left\{
\frac{1}{\varphi((\ell_k m_k)^{-1})}\sum_{j \in F_{l'}}
\chi_{b_{k,j}+e_{k,i}+[0,(\ell_k m_k)^{-1}]^n}
\right\}_{a \in {\mathbb R}^n, k \in {\mathbb N}, \, i=1,2,\ldots,m_k{}^n}
\]
is a bounded set.
So
$\{f_{k,i,a}\}_{a \in {\mathbb R}^n, k \in {\mathbb N}, \, i=1,2,\ldots,m_k{}^n}$
forms a bounded set in ${\mathcal M}^\varphi_{q}({\mathbb R}^n)$.
\end{example}

From these examples, we obtain the following conclusion:
\begin{corollary}{\rm \cite[Corollary 4.11]{Nakai08-1}}\label{cor:180521-2}
Let $0<q_1,q_2<\infty$,
and let 
$\varphi_1 \in {\mathcal G}_{q_1}$
and
$\varphi_2 \in {\mathcal G}_{q_2}$.
Assume in addition that $\inf_{t>0}\varphi_1(t)=0$.
Then
$
{\mathcal M}^{\varphi_1}_{q_1}({\mathbb R}^n)
\subset
{\mathcal M}^{\varphi_2}_{q_2}({\mathbb R}^n)
$
if and only if
$q_1 \ge q_2$ and
$\varphi_1 \gtrsim \varphi_2$.
Furthemore,
if $q_1<q_2$,
there exists a compactly supported function
$f \in 
{\mathcal M}^{\varphi_2}_{q_2}({\mathbb R}^n)
\setminus
{\mathcal M}^{\varphi_1}_{q_1}({\mathbb R}^n)$.
In particular, for
$0<q_1 \le p_1<\infty$
and
$0<q_2 \le p_2<\infty$,
$
{\mathcal M}^{p_1}_{q_1}({\mathbb R}^n)
\subset
{\mathcal M}^{p_2}_{q_2}({\mathbb R}^n)
$
if and only if
$q_1 \ge q_2$ and
$p_1=p_2$.
\end{corollary}

\begin{proof}
The \lq \lq if" part is trivial
thanks to Lemma \ref{lem:180307-1}.
Let us concentrate on the \lq \lq only if" part.

Assume $q_1<q_2$.
We employ Example \ref{example:180310-116}
with $q=q_1$ and $\varphi=\varphi_1$.
Then
\[
\left\|\sum_{j \in F_{l'}}
\chi_{b_{k,j}+e_{k,i}+[0,(\ell_k m_k)^{-1}]^n}
\right\|_{L^{q_2}}
\sim
\left(
\ell_k{}^n(\ell_k m_k)^{-n}
\right)^{\frac{1}{q_2}}
=
m_k{}^{-\frac{n}{q_2}}.
\]
As a result,
\begin{align*}
\frac{1}{\varphi((\ell_k m_k)^{-1})}
\left\|\sum_{j \in F_{l'}}
\chi_{b_{k,j}+e_{k,i}+[0,(\ell_k m_k)^{-1}]^n}
\right\|_{L^{q_2}}
&\sim
\frac{1}{\varphi((\ell_k m_k)^{-1})m_k{}^{\frac{n}{q_2}}}\\
&=
m_k{}^{\frac{n}{q}-\frac{n}{q_2}}
\to \infty,
\end{align*}
since
$\inf_{t>0}\varphi_1(t)=0$.
Thus,
${\mathcal M}^{\varphi_1}_{q_1}({\mathbb R}^n)$
contains a function which does not belong to
$L^{q_2}_{\rm loc}({\mathbb R}^n)$.
Consequently, if
$
{\mathcal M}^{\varphi_1}_{q_1}({\mathbb R}^n)
\subset
{\mathcal M}^{\varphi_2}_{q_2}({\mathbb R}^n),
$
then $q_1 \ge q_2$.
Finally,
we must have
$\varphi_1 \gtrsim \varphi_2$
from Lemma \ref{lem7}
if
$
{\mathcal M}^{\varphi_1}_{q_1}({\mathbb R}^n)
\subset
{\mathcal M}^{\varphi_2}_{q_2}({\mathbb R}^n).
$
\end{proof}
When $\inf_{t>0}\varphi_1(t)>0$,
the situation is different.
\begin{example}\label{example:180611-15}
When $\varphi_1$ is constant,
we use Theorem \ref{thm:150501-11} to have the following characterization.
Let $0<q_1,q_2<\infty$,
and let 
$\varphi_1=1 \in {\mathcal G}_{q_1}$
and
$\varphi_2 \in {\mathcal G}_{q_2}$.
Then
$L^\infty({\mathbb R}^n)=
{\mathcal M}^{\varphi_1}_{q_1}({\mathbb R}^n)
\subset
{\mathcal M}^{\varphi_2}_{q_2}({\mathbb R}^n)
$
if and only if
$\sup_{t>0} \varphi(t)_2<\infty$.
\end{example}

For the later purpose we use
the following characterization of 
$\overset{*}{{\mathcal M}}{}^{\varphi}_{q}({\mathbb R}^n)$,
which is defined to be the closure
of 
${\mathcal M}^{\varphi}_{q}({\mathbb R}^n) \cap L^0_{\rm c}({\mathbb R}^n)$
in 
${\mathcal M}^{\varphi}_{q}({\mathbb R}^n)$.
\index{$\overset{*}{{\mathcal M}}{}^{\varphi}_{q}({\mathbb R}^n)$}
\begin{lemma}\label{lem:150526-1}
For $0< q\le p<\infty$, we have
\begin{equation}\label{eq:150526-1}
\overset{*}{{\mathcal M}}{}^{\varphi}_{q}({\mathbb R}^n)
=\left\{ f\in {\mathcal M}^{\varphi}_{q}({\mathbb R}^n): 
\lim\limits_{R\to \infty} 
\|\chi_{{\mathbb R}^n\setminus B(R)} f \|_{{\mathcal M}^{\varphi}_{q}}
=0\right\}.
\end{equation}
\end{lemma}

\begin{proof}
Assume that
$f\in {\mathcal M}^{\varphi}_{q}({\mathbb R}^n)$
satisfies
\[
\lim\limits_{R\to \infty} 
\|\chi_{{\mathbb R}^n\setminus B(R)} f \|_{{\mathcal M}^{\varphi}_{q}}
=0.
\]
Then 
$\displaystyle
f=\lim\limits_{R\to \infty} 
\chi_{B(R)} f
$
in ${\mathcal M}^\varphi_q({\mathbb R}^n)$.
Conversely,
if $f \in \overset{*}{{\mathcal M}}{}^{\varphi}_{q}({\mathbb R}^n)$
and $\varepsilon>0$,
then we can find
$g \in {\mathcal M}^{\varphi}_{q}({\mathbb R}^n) \cap L^0_{\rm c}({\mathbb R}^n)$
such that
$\|f-g\|_{{\mathcal M}^\varphi_q}<\varepsilon$.
By the triangle inequality
\[
\|\chi_{{\mathbb R}^n\setminus B(R)} f \|_{{\mathcal M}^{\varphi}_{q}}
\le
\|\chi_{{\mathbb R}^n\setminus B(R)} g \|_{{\mathcal M}^{\varphi}_{q}}
+
\|f-g \|_{{\mathcal M}^{\varphi}_{q}}.
\]
Let $R>0$ be large enough as to have ${\rm supp}(g) \subset B(R)$.
Then the above inequality reads as:
\[
\|\chi_{{\mathbb R}^n\setminus B(R)} f \|_{{\mathcal M}^{\varphi}_{q}}
\le
\|f-g \|_{{\mathcal M}^{\varphi}_{q}}(<\varepsilon),
\]
as required.
\end{proof}

We also have the following characterization of the space
$\overline{\mathcal M}^\varphi_q({\mathbb R}^n)$,
the closure of 
$L^\infty({\mathbb R}^n) \cap {\mathcal M}^\varphi_q({\mathbb R}^n)$
in ${\mathcal M}^\varphi_q({\mathbb R}^n)$:
\index{$\overline{\mathcal M}^\varphi_q({\mathbb R}^n)$}
\begin{lemma}\label{lem:150805-1}
Let $0<q<\infty$ and $\varphi \in {\mathcal G}_q$.
If $f\in \overline{\mathcal M}{}^\varphi_q({\mathbb R}^n)$, then 
\begin{align}\label{eq:150805-1}
\lim_{R\to \infty}
\|\chi_{\{|f|>R\}}f\|_{{\mathcal M}^\varphi_q}
=0.
\end{align}
\end{lemma}

\begin{proof}
Let $\varepsilon>0$ be fixed.
Since $f\in \overline{\mathcal M}{}^\varphi_q({\mathbb R}^n)$,
we can choose
$g \in L^\infty({\mathbb R}^n) \cap {\mathcal M}{}^\varphi_q({\mathbb R}^n)$
so that
$\|g-f\|_{{\mathcal M}^\varphi_q}<\varepsilon$.
Let $R>\|g\|_{\infty}$
We estimate
\[
\|\chi_{\{|f|>2R\}}f\|_{{\mathcal M}^\varphi_q}
\le
\|f-g\|_{{\mathcal M}^\varphi_q}
+
\|\chi_{\{|f|>2R\}}g\|_{{\mathcal M}^\varphi_q}
\]
If $|g| \le R$ and $|f|>2R$,
then $|g| \le R \le |f-g|$.
Thus
\begin{align*}
\|\chi_{\{|f|>2R\}}f\|_{{\mathcal M}^\varphi_q}
&\le
2\|f-g\|_{{\mathcal M}^\varphi_q}
+
\|\chi_{\{|g|>R\}}g\|_{{\mathcal M}^\varphi_q}\\
&=
2\varepsilon.
\end{align*}
As a result
(\ref{eq:150805-1}) holds.
\end{proof}
The generalized tilde subspace
$\widetilde{\mathcal M}{}^\varphi_q({\mathbb R}^n)$
is defined to be the completion
of $L^\infty({\mathbb R}^n) \cap {\mathcal M}^\varphi_q({\mathbb R}^n)$
in ${\mathcal M}^\varphi_q({\mathbb R}^n)$.
\index{$\widetilde{\mathcal M}{}^\varphi_q({\mathbb R}^n)$}
\begin{proposition}\label{prop:180611-21}
Let $0<q<\infty$ and $\varphi \in {\mathcal G}_q$.
Then
$\widetilde{\mathcal M}{}^\varphi_q({\mathbb R}^n)
=\overline{\mathcal M}{}^\varphi_q({\mathbb R}^n)
\cap
\overset{*}{\mathcal M}{}^\varphi_q({\mathbb R}^n)$.
\end{proposition}

\begin{proof}
It is easy to show that
$\widetilde{\mathcal M}{}^\varphi_q({\mathbb R}^n)
\subset
\overline{\mathcal M}{}^\varphi_q({\mathbb R}^n)
\cap
\overset{*}{\mathcal M}{}^\varphi_q({\mathbb R}^n)$.
Let
$f \in 
\overline{\mathcal M}{}^\varphi_q({\mathbb R}^n)
\cap
\overset{*}{\mathcal M}{}^\varphi_q({\mathbb R}^n)$.
By Lemmas \ref{lem:150526-1} and \ref{lem:150805-1},
we see that
\[
f=\lim_{R \to \infty}\chi_{B(R) \cap \{|f|>R\}}f.
\]
Thus $f \in \widetilde{\mathcal M}{}^\varphi_q({\mathbb R}^n)$.
\end{proof}
Although we can define many other closed subspaces
as we did for Morrey spaces,
we content ourselves with these three definitions,
which we actually use.

\begin{remark}
Here the space $\widetilde{\mathcal M}^\varphi_q({\mathbb R}^n)$
is introduced.
See \cite{AkKu14, DGS15, LoHa16, Sa(N)13-1,Sa(N)13-2} for a different class of closed subspaces.
\end{remark}

\section{Boundedness properties of the operators in generalized Morrey spaces}
\label{s4}

Having set down the boundedness properties of generalized Morrey spaces,
we are now interested in the boundedness properties of the operators.

\subsection{Maximal operator in generalized Morrey spaces and the class ${\mathbb Z}_0$}
\label{subsection:Boundedness of the Hardy--Littlewood maximal operator on generalized Morrey spaces}

After we generalize the parameter $p$
in the space
${\mathcal M}^p_q({\mathbb R}^n)$,
we realize that the boundedness of the maximal operator
is obtained due to the condition on $q \in (1,\infty)$.

\begin{theorem}
\label{thm:140525-3}
{\rm \cite[Theorem 1]{Nakai94}, \cite[Corollary 5.6]{Nakai08-1}, \cite[Theorem 2.3]{Sawano08-1}}
Let $1<q<\infty$ and $\varphi \in {\mathcal G}_q$.
Then
\[
\|M f\|_{{\mathcal M}^\varphi_q}
\lesssim
\|f\|_{{\mathcal M}^\varphi_q}
\]
for all measurable functions $f$.
\end{theorem}

We observe that we did not requere
anything other than $\varphi \in {\mathcal G}_q$
as an evidence of the fact that the parametre $q$
play a central role for the boundedness of the 
Hardy--Littlewood maximal operator.
Theorem \ref{thm:140525-3} extends
the result in \cite{ChFr87}
from Morrey spaces to generalized Morrey spaces.
\begin{proof}
Once we assume $\varphi \in {\mathcal G}_q$,
the proof of this theorem will be an adaptation
of the classical case.
Fix a cube $Q(x,r)$.
We need to prove
\[
\varphi(r)
\left(\frac{1}{|Q(x,r)|}\int_{Q(x,r)}Mf(y)^q\,dy\right)^{\frac1q}
\lesssim
\|f\|_{{\mathcal M}^\varphi_q}.
\]
We let $f_1 \equiv \chi_{Q(x,5r)}f$ and $f_2 \equiv f-f_1$.
We need to prove
\begin{equation}\label{eq:131113-11}
\varphi(r)
\left(\frac{1}{|Q(x,r)|}\int_{Q(x,r)}Mf_1(y)^q\,dy\right)^{\frac1q}
\lesssim
\|f\|_{{\mathcal M}^\varphi_q}
\end{equation}
and
\begin{equation}\label{eq:131113-12}
\varphi(r)
\left(\frac{1}{|Q(x,r)|}\int_{Q(x,r)}Mf_2(y)^q\,dy\right)^{\frac1q}
\lesssim
\|f\|_{{\mathcal M}^\varphi_q}.
\end{equation}
The proof of (\ref{eq:131113-11}) follows
from the boundedness of the Hardy--Littlewood maximal operator
and the fact that $\varphi(5r) \simeq \varphi(r)$
for any $r>0$.
As for (\ref{eq:131113-12}),
we use 
\[
M[\chi_{{\mathbb R}^n \setminus 5Q}f](y)
\lesssim \sup_{R:Q \subset R \in {\mathcal Q}}
\frac{1}{|R|}\int_R|f(z)|\,dz
\quad (y \in Q)
\]
with $Q=Q(x,r)$.
Then by virtue of the fact that $\varphi$ is increasing,
we obtain
\begin{align*}
\varphi(r)
\left(\frac{1}{|Q(x,r)|}\int_{Q(x,r)}Mf_2(y)^q\,dy\right)^{\frac1q}
&\lesssim
\varphi(r)
\times
\sup_{R:Q \subset R \in {\mathcal Q}}
\frac{1}{|R|}\int_R|f(z)|\,dz\\
&\lesssim
\sup_{R:Q \subset R \in {\mathcal Q}}
\varphi(r)
\times
\frac{1}{|R|}\int_R|f(z)|\,dz\\
&\lesssim
\|f\|_{{\mathcal M}^\varphi_1}\\
&\le
\|f\|_{{\mathcal M}^\varphi_q},
\end{align*}
where for the last inequality,
we used the nesting property
of ${\mathcal M}^\varphi_q({\mathbb R}^n), 1 \le q<\infty$.
\end{proof}

As we have seen,
any function $\varphi:(0,\infty) \to (0,\infty)$ will do.
So we have the following boundedness for small Morrey spaces.
\begin{example}\label{example:180611-31}
Let $1<q \le p<\infty$.
Then
$M$ is bounded on $m^p_q({\mathbb R}^n)$
and hence $L^q_{\rm uloc}({\mathbb R}^n)$.
In fact,
$M$ is bounded on ${\mathcal M}^\varphi_q({\mathbb R}^n)
\sim m^p_q({\mathbb R}^n)$
thanks to Theorem \ref{thm:140525-3},
where $\varphi(t)=\max(t^{\frac{n}{p}},1)$,
$t>0$.
\end{example}

Similar to Theorem \ref{thm:140525-3},
we can prove the following theorem:
\begin{theorem}{\rm \cite[Theorem 1]{Nakai94}, \cite[Corollary 6.2]{Nakai08-1}}
\label{thm:140525-333}
Let $1 \le q<\infty$ and $\varphi \in {\mathcal G}_q$.
Then
\[
\|M f\|_{{\rm w}{\mathcal M}^\varphi_q}
\lesssim
\|f\|_{{\mathcal M}^\varphi_q}
\]
for all measurable functions $f$.
\end{theorem}
Once again any function 
${\mathcal G}_q$
will do as long as $\varphi \in {\mathcal G}_q$.
Theorem \ref{thm:140525-333} extends
the result in \cite{ChFr87}
from Morrey spaces to generalized Morrey spaces.
\begin{proof}
Simply resort to the weak-$(1,1)$ boundedness
of $M$ and modify
the proof of Theorem \ref{thm:140525-3}.
\end{proof}

We move on to the vector-valued inequality.
We use the following estimate:
\begin{lemma}\label{lem:150824-21}
For all measurable functions $f$ and cubes $Q$,
we have
\begin{equation}\label{eq:131109-69A}
M[\chi_{{\mathbb R}^n \setminus 5Q}f](y)
\lesssim \sup_{Q \subset R \in {\mathcal Q}}
m_R(|f|)
\quad (y \in Q).
\end{equation}
\end{lemma}

\begin{proof}
We write out $M[\chi_{{\mathbb R}^n \setminus 5Q}f](y)$
in full:
\[
M [\chi_{{\mathbb R}^n \setminus 5Q}f](y)
=
\sup_{R \in {\mathcal Q}}
\frac{\chi_{R}(y)}{|R|}\|f\|_{L^1(R \setminus 5Q)},
\]
where $R$ runs over all cubes.
In order that 
$\chi_{R}(y)\|f\|_{L^1(R \setminus 5Q)}$
be nonzero,
we need to have $y \in R$ and
$R \setminus 5Q \ne \emptyset$.
Thus, $R$ meets both $Q$ and ${\mathbb R}^n \setminus 5Q$.
If $R \in {\mathcal Q}$ is a cube
that meets both $Q$ and ${\mathbb R}^n \setminus 5Q$,
then $\ell(R) \ge 2\ell(Q)$ and $2R \supset Q$.
Thus, (\ref{eq:131109-69A}) follows.
\end{proof}

Unlike the usual maximal inequality,
we need the integral condition
(\ref{eq:Nakai-19}).
\begin{theorem}\label{thm:180524-1}
Let $1<q<\infty$, $1<u \le \infty$ and $\varphi \in {\mathcal G}_q$.
Assume in addition that
\begin{equation}\label{eq:Nakai-19}
\int_r^\infty \frac{ds}{\varphi(s)s}
\lesssim \frac{1}{\varphi(r)} \quad (r>0).
\end{equation}
Then
for all $\{f_j\}_{j=1}^\infty \subset {\mathcal M}^\varphi_q({\mathbb R}^n)$,
\begin{equation}\label{eq:180311-11111}
\left\|\left(\sum_{j=1}^\infty M f_j{}^u\right)^{\frac1u}\right\|_{{\mathcal M}^\varphi_q}
\lesssim
\left\|\left(\sum_{j=1}^\infty |f_j|^u\right)^{\frac1u}\right\|_{{\mathcal M}^\varphi_q}.
\end{equation}
\end{theorem}

\begin{proof}
When $u=\infty$,
the result is clear from $(\ref{thm:140525-3})$.
The proof is essentially the same
as the classical case
except that we truly use (\ref{eq:Nakai-19}).
The proof of the estimate of the inner term
remains unchanged
except in that we need to generalize the parameter $p$
to the function $\varphi$.
Let $f_{j,1}=\chi_{5Q}f_j$ and $f_{j,2}=f_j-f_{j,1}$.
We can handle $f_{j,1}$'s in a standard manner as before.
Going through a similar argument to the classical case
and using Lemma \ref{lem:150824-21},
we will have
\begin{align*}
\frac{\varphi(\ell(Q))}{|Q|}\int_Q
\left(\sum_{j=1}^\infty Mf_{j,2}(y)^u\right)^\frac{1}{u}\,dy
&\lesssim
\sum_{k=1}^\infty
\frac{\varphi(\ell(Q))}{|2^k Q|}\int_{2^k Q}
\left(\sum_{j=1}^\infty |f_j(z)|^u\right)^\frac{1}{u}\,dz.
\end{align*}
If we use the definition of the Morrey norm,
we obtain
\begin{align}
\label{eq:180310-33335}
\frac{\varphi(\ell(Q))}{|Q|}\int_Q
\left(\sum_{j=1}^\infty Mf_{j,2}(y)^u\right)^\frac{1}{u}\,dy
&\lesssim
\sum_{k=1}^\infty
\frac{\varphi(\ell(Q))}{\varphi(2^k \ell(Q))}
\left\|\left(\sum_{j=1}^\infty |f_j|^u\right)^{\frac1u}\right\|_{{\mathcal M}^\varphi_q}.
\end{align}
Since $\varphi \in {\mathcal G}_q$,
we obtain
\[
\sum_{k=1}^\infty
\frac{\varphi(\ell(Q))}{\varphi(2^k \ell(Q))}
\lesssim
\int_{\ell(Q)}^\infty 
\frac{\varphi(\ell(Q))}{\varphi(t)t}\,dt.
\]
If we use (\ref{eq:Nakai-19})
and $\varphi \in {\mathcal G}_q$,
then we have
\[
\sum_{k=1}^\infty
\frac{\varphi(\ell(Q))}{\varphi(2^k \ell(Q))}
\lesssim
1.
\]
Inserting this estimate into (\ref{eq:180310-33335}),
we obtain the counterpart to the classical case.
\end{proof}
Since (\ref{eq:Nakai-19}) is an important condition,
we are interested in its characterization.
In fact, we have the following useful one.
\begin{theorem}\label{thm:180611-16}
Assume that $\varphi:(0,\infty) \to (0,\infty)$ is an almost increasing function.
Then the following are equivalent:
\begin{enumerate}
\item
$\varphi$ satisfies $(\ref{eq:Nakai-19})$.
\item
There exists $m_0 \in {\mathbb N}$
such that
\begin{equation}\label{eq:150219-1}
\varphi(2^{m_0}r)>2\varphi(r)
\end{equation}
for all $r>0$.
\end{enumerate}
\end{theorem}

\begin{proof}
Assume that (\ref{eq:Nakai-19}) holds.
Let $m_0'$ be a positive integer such that
$\varphi(2^{m_0'-1}r)\le 2\varphi(r)$
for all $r>0$.
Thus since $\varphi$ is almost increasing,
\[
\frac{\log m_0'}{2\varphi(r)}
\le \int_r^{2^{m_0'-1}r}\frac{ds}{\varphi(s)s} \lesssim 
\frac{1}{\varphi(r)}.
\]
Since $\varphi(r)>0$, 
we have an upper bound $M$
for $m_0'$.
Thus if we set $m_0=M+1$,
we obtain the desired number $m_0$.

If (\ref{eq:Nakai-19}) holds,
then
\[
\int_r^\infty \frac{ds}{\varphi(s)s}
=
\sum_{j=1}^\infty
\int_{2^{m_0(j-1)}r}^{2^{m_0 j}r}\frac{ds}{\varphi(s)s}
\lesssim
\sum_{j=1}^\infty
\int_{2^{m_0(j-1)}r}^{2^{m_0 j}r}\frac{ds}{2^{j}\varphi(r)s}
\lesssim \frac{1}{\varphi(r)},
\]
as required.
\end{proof}

As we have mentioned,
we need (\ref{eq:Nakai-19})
for the vector-valued maximal inequality.
We give an example showing that (\ref{eq:Nakai-19})
is absolutely necessary:
By no means
(\ref{eq:Nakai-19}) is artificial
as the following proposition shows:\begin{proposition}\label{prop:180311-1}
Let $1<q<\infty$, $1<u<\infty$ and $\varphi \in {\mathcal G}_q$.
Assume in addition that
\begin{equation}\label{eq:180815-2}
\left\|\left(\sum_{j=1}^\infty M f_j{}^u\right)^{\frac1u}\right\|_{{\rm w}{\mathcal M}^\varphi_q}
\lesssim
\left\|\left(\sum_{j=1}^\infty |f_j|^u\right)^{\frac1u}\right\|_{{\mathcal M}^\varphi_q}.
\end{equation}
holds for all sequences of measurable functions
$\{f_j\}_{j=1}^\infty$.
Then
$(\ref{eq:Nakai-19})$
holds.
\end{proposition}

We exclude the case where $u=\infty$,
where $(\ref{eq:180815-2})$ still holds without
$(\ref{eq:Nakai-19})$.

\begin{proof}
Assume to the contrary;
for all $m\in {\mathbb N}\cap [2,\infty)$, there would exist $r_m>0$
such that
$\varphi(2^m r_m) \le 2\varphi(r_m)$ for $r=r_m$. 
Let us consider
$f_j=\chi_{[1,j]}(m)\chi_{B(2^jr_m) \setminus B(2^{j-1}r_m)}$
for $j \in {\mathbb N}$.
Observe 
\[
\left\|\left(\sum_{j=1}^\infty |f_j|^u\right)^{\frac1u}\right\|_{{\mathcal M}^\varphi_q}
=
\left\|\left(
\sum_{j=1}^m \chi_{B(2^jr_m) \setminus B(2^{j-1}r_m)}
\right)^{\frac1u}\right\|_{{\mathcal M}^\varphi_q}.
\]
As a result,
\[
\left\|\left(\sum_{j=1}^\infty |f_j|^u\right)^{\frac1u}\right\|_{{\mathcal M}^\varphi_q}
\le
\|\chi_{B( 2^mr_m)}\|_{{\mathcal M}^\varphi_q}
\le \varphi(2^mr_m)\lesssim \varphi(r_m).
\]
Let $x\in B(r_m)$. 
For $j>m$, we have $M f_j(x)=0$.
Meanwhile for $1\le j\le m$, we have
\begin{align*}
M f_j(x)
&\ge 
\frac{1}{|B(x,2^{j+1}r_m)|}
\int_{B(x,2^{j+1}r_m)}
f_j(y) \ dy \\
&\ge 
\frac{1}{|B(x,2^{j+1}r_m)|}
\int_{B(2^jr_m)}
 \chi_{B(2^jr_m) \setminus B(2^{j-1}r_m) }(y) \ dy \\
&= \frac{2^n-1}{4^n} \ge \frac{1}{4^n}.
\end{align*}
Consequently, 
\[
\left\|\left(\sum_{j=1}^\infty M f_j{}^u\right)^{\frac1u}\right\|_{{\rm w}{\mathcal M}^\varphi_q}
\ge 
\varphi(r_m)
\sum_{j=1}^m\left(\frac{1}{4^n}\right)^u
\sim 
\varphi(r_m) \ m^{\frac1u}.
\]
By our assumption, we have
\begin{align*}
\varphi(r_m)m^{\frac1u}
\lesssim 
\left\|\left(\sum_{j=1}^\infty M f_j{}^u\right)^{\frac1u}\right\|_{{\rm w}{\mathcal M}^\varphi_q}
\lesssim 
\left\|\left(\sum_{j=1}^\infty |f_j|^u\right)^{\frac1u}\right\|_{{\mathcal M}^\varphi_q}
\lesssim \varphi(r_m),
\end{align*}
or equivalently 
$$ m \le D$$
where $D$ does not depend on $m$. 
This contradicts
to the fact that $m\in {\mathbb N} \cap [2,\infty)$ is arbitrary.
\end{proof}

\begin{example}\label{example:180611-16}
Let $1\ \le q \le p<\infty$.
Then
\[
\left\|\left(\sum_{j=1}^\infty M f_j{}^u\right)^{\frac1u}\right\|_{{\rm w}m^p_q}
\lesssim
\left\|\left(\sum_{j=1}^\infty |f_j|^u\right)^{\frac1u}\right\|_{m^p_q}
\quad
(\{f_j\}_{j=1}^\infty \subset m^p_q({\mathbb R}^n))
\]
fails.
In particular,
\[
\left\|\left(\sum_{j=1}^\infty M f_j{}^u\right)^{\frac1u}\right\|_{{\rm w}L^q_{\rm uloc}}
\lesssim
\left\|\left(\sum_{j=1}^\infty |f_j|^u\right)^{\frac1u}\right\|_{L^q_{\rm uloc}}
\quad
(\{f_j\}_{j=1}^\infty \subset L^q_{\rm uloc}({\mathbb R}^n))
\]
fails.
In fact,
let $\varphi(t)=\max(t^{\frac{n}{p}},1)$ as before.
Then
$\varphi$ fails (\ref{eq:Nakai-19})
because
$\displaystyle
\int_1^\infty \frac{dr}{\varphi(r)}=\infty.
$
\end{example}
Proposition \ref{prop:180311-1}
led us to the conclusion that
$(\ref{eq:Nakai-19})$ is fundamental.
The following proposition will be fundamental
in the study of
the boundedness of the operators in generalized Morrey spaces.
\begin{theorem}{\rm \cite[Lemma 2]{Nakai94}}\label{prop:Nakai-a}
If a nonnegative locally integrable function
$\psi$ and a positive constant $D>0$
satisfy
\[
\int_r^\infty \psi(t)\frac{dt}{t}
\le D
\psi(r)
\quad (r>0)
\]
then
\begin{equation}\label{eq:Nakai-2}
\int_r^\infty \psi(t)t^{\varepsilon}\,\frac{dt}{t}
\le \frac{r^\varepsilon}{1-D\varepsilon}\cdot
\int_r^\infty \psi(t)t^{-1}\,dt
\le \frac{D}{1-D\varepsilon}\cdot
\psi(r)r^\varepsilon \quad (r>0)
\end{equation}
for all $0<\varepsilon<D^{-1}$.
\end{theorem}

\begin{proof}
Let
\[
\Psi(r)=
\int_r^\infty \psi(t)t^{-1}\,dt
\quad (r>0).
\]
For $0<r<R$
\begin{align*}
\int_r^R \psi(t)t^{\varepsilon}\,\frac{dt}{t}
=
[-\Psi(t)t^\varepsilon]^R_r
+
\int_r^R \Psi(t)\varepsilon t^{\varepsilon}\,\frac{dt}{t}
\le \Psi(r)r^{\varepsilon}
+
\varepsilon D\int_r^R \psi(t)t^{\varepsilon}\,\frac{dt}{t}.
\end{align*}
Therefore
\[
\int_r^R \psi(t)t^{\varepsilon}\,\frac{dt}{t}
\le \frac{1}{1-\varepsilon D}\Psi(r)r^{\varepsilon}
\le \frac{D}{1-\varepsilon D}
\psi(r)r^{\varepsilon}.
\]
It remains to let $R \to \infty$.
\end{proof}

We change variables to have the following variant:
\begin{theorem}{\rm \cite[Lemma 2]{Nakai94}}\label{thm:Nakai-a}
Let $\psi:(0,\infty) \to (0,\infty)$
be a measurable function satisfying
\[
\int_0^r \psi(t)\frac{dt}{t}
\le D
\psi(r)
\quad (r>0)
\]
for some $D>0$ independent of $r>0$.
If $0<\varepsilon<D^{-1}$,
then
\[
\int_0^r \psi(t)\frac{dt}{t^{1+\varepsilon}}
\le \frac{1}{1-D\varepsilon}
r^{-\varepsilon}
\int_0^{r} \psi(t)\frac{dt}{t}
\le \frac{D}{1-D\varepsilon}
r^{-\varepsilon}\psi(r).
\]
\end{theorem}

\begin{proof}
Set
\[
\eta(t)=\psi\left(\frac{1}{t}\right) \quad(t>0).
\]
Then our assumption reads as:
\[
\int_{r}^\infty \eta(t)\frac{dt}{t}
\le D
\eta(r)
\quad (r>0).
\]
Thus, 
\[
\int_{r}^\infty \eta(t)\frac{dt}{t^{1-\varepsilon}}
\le \frac{1}{1-D\varepsilon}
r^{\varepsilon}
\int_{r}^\infty \eta(t)\frac{dt}{t}
\le \frac{D}{1-D\varepsilon}
r^{\varepsilon}\eta(r)
\quad (r>0)
\]
according to Theorem \ref{prop:Nakai-a}.
If we express this inequality
in terms of $\psi$,
we obtain the desired result.
\end{proof}

In the next proposition,
we further characterize and apply our key assumption
(\ref{eq:Nakai-19}).
\begin{proposition}{\rm \cite[Proposition 2.7]{NNS15}}\label{prop:150312-1}
Let $\varphi$ be a nonnegative locally integrable function such that
$\varphi(s) \lesssim \varphi(r)$
for all $r,s>0$ with $\dfrac12 \le \dfrac{r}{s} \le 2$.
There exists a constant $\varepsilon>0$ such that
\begin{equation}\label{eq:Nakai-3}
\frac{t^\varepsilon}{\varphi(t)} 
\lesssim
\frac{r^{\varepsilon}}{\varphi(r)}
\quad (t \ge r)
\end{equation}
holds
if and only if
holds $(\ref{eq:Nakai-19})$, or equivalently,
$\varphi$ satisfies $(\ref{eq:Nakai-2})$
for some $\varepsilon>0$.
If one of these conditions is satisfied,
then
\begin{equation}\label{eq:Nakai-49}
\int_r^\infty \frac{ds}{\varphi(s)^us}
\lesssim \frac{1}{\varphi(r)^u} \quad (r>0)
\end{equation}
for all $0<u<\infty$,
where the implicit constant depends only on $u$.
\end{proposition}

\begin{proof}
The implication $(\ref{eq:Nakai-19}) \Longrightarrow (\ref{eq:Nakai-2})$
follows from Proposition \ref{prop:Nakai-a}.

Assume (\ref{eq:Nakai-2}).
Then we have
\[
\frac{t^\varepsilon}{\varphi(t)} 
{\lesssim}
\int_t^{2t}\frac{dv}{v^{1-\varepsilon}\varphi(v)}
{\lesssim}
\frac{r^\varepsilon}{\varphi(r)}
\]
thanks to the doubling property of $\varphi$,
proving (\ref{eq:Nakai-3}).

If we assume (\ref{eq:Nakai-3}), then we have
\[
\int_r^\infty \frac{ds}{\varphi(s)s}
=
\int_r^\infty \frac{s^\varepsilon}{\varphi(s)}
\frac{ds}{s^{1+\varepsilon}}
\lesssim
\int_r^\infty \frac{r^\varepsilon}{\varphi(r)}
\frac{ds}{s^{1+\varepsilon}}
= \frac{1}{\varepsilon\varphi(r)},
\]
which implies (\ref{eq:Nakai-19}).
Note that (\ref{eq:Nakai-3}) also implies (\ref{eq:Nakai-49})
because $\varphi^u$ satisfies (\ref{eq:Nakai-3}) as well.
\end{proof}

Let $0<u<\infty$ be fixed.
Inequality (\ref{eq:Nakai-49}) is 
necessary for (\ref{eq:Nakai-19});
simply apply
Proposition \ref{prop:150312-1}
to $\varphi^u$.

\begin{example}\label{example:180613-1}
Let $\varphi \in {\mathcal G}_q$ satisfy
$(\ref{eq:Nakai-49})$.
Then we have
\[
\int_r^\infty \frac{1}{\varphi(s)s^{1-\varepsilon}}\,ds \lesssim \frac{r^{\varepsilon}}{\varphi(r)}
\quad (r>0).
\]
Hence
\[
\frac{s^{\varepsilon}}{\varphi(s)} \lesssim \frac{r^{\varepsilon}}{\varphi(r)}
\]
whenever $0<r \le s<\infty$.
As a result,
\[
\int_0^1 \varphi(r)\frac{d r}{r}<\infty.
\]
\end{example}
We generalize condition
(\ref{eq:Nakai-19}) as follows:

\begin{definition}\label{defi:180611-13}
Let $\gamma \in {\mathbb R}$.
\begin{enumerate}
\item
The {\it $($upper$)$ Zygmund class} ${\mathbb Z}^\gamma$
is defined to be the set of all measurable functions
$\varphi:(0,\infty) \to (0,\infty)$ for which
$\displaystyle
\lim_{r \downarrow 0}\varphi(r)=0
$
and
\[
\int_0^r \varphi(t)t^{-\gamma-1}\,dt
\lesssim
\varphi(r)r^{-\gamma} \quad (r > 0),
\]
\index{${\mathbb Z}^\gamma$}
\item
The {\it $($lower$)$ Zygmund class} ${\mathbb Z}_\gamma$
is defined to be the set of all measurable functions
$\varphi:(0,\infty) \to (0,\infty)$ for which
$\displaystyle
\lim_{r \downarrow 0}\varphi(r)=0
$
and
\[
\int_r^\infty \varphi(t)t^{-\gamma-1}\,dt
\lesssim
\varphi(r)r^{-\gamma} \quad (r > 0).
\]
\index{${\mathbb Z}_\gamma$}
\end{enumerate}
\end{definition}

Note that (\ref{eq:Nakai-19})
reads as $\frac{1}{\varphi}\in {\mathbb Z}_0$.
\begin{example}\label{example:180611-17}
Let $\varphi(t)=t^p$ with $1<p<\infty$,
and let $\gamma \in {\mathbb R}$.
\begin{enumerate}
\item
$\varphi \in {\mathbb Z}^\gamma$
if and only if 
$p>\gamma$.
\item
$\varphi \in {\mathbb Z}_\gamma$
if and only if
$p<\gamma$.
\item
$1 \notin {\mathbb Z}^\gamma$
if and only if $\gamma<0$.
\end{enumerate}
\end{example}

We present an example of the functions in ${\mathcal M}^\varphi_q({\mathbb R}^n)$.
Let $0<\eta<\infty$.
We define the {\it powered Hardy--Littlewood maximal operator}
$M^{(\eta)}$ by
\[
M^{(\eta)}f(x)
\equiv \sup_{R>0}
\left(
\frac{1}{|B(x,R)|}\int_{B(x,R)}|f(y)|^\eta{\rm d}y
\right)^\frac{1}{\eta} \quad (x \in {\mathbb R}^n).
\]

\begin{example}{\rm \cite[Proposition 2.11]{NNS15}}\label{example:180310-A1}
Let $0<q<\infty$ and $\varphi \in {\mathcal G}_q$.
Define
\[
f \equiv 
\sum_{j=-\infty}^\infty \frac{\chi_{[2^{-j-1},2^{-j}]^n}}{\varphi(2^{-j})}, 
\quad
g \equiv 
\sup_{j \in {\mathbb Z}}\frac{\chi_{[0,2^{-j}]^n}}{\varphi(2^{-j})}.
\]
We claim that the following are equivalent;
\begin{enumerate}
\item[$(a)$]
$f \in {\mathcal M}^\varphi_q({\mathbb R}^n)$,
\item[$(b)$]
$g \in {\mathcal M}^\varphi_q({\mathbb R}^n)$,
\item[$(c)$]
$\frac{1}{\varphi} \in {\mathbb Z}^{\frac{n}{q}}$.
\end{enumerate}
Let $0<u<q$.
Note that
$f \le g \le 2^n M^{(u)}f$,
where $M^{(u)}$ denotes the powered Hardy--Littlewood maximal operator.
Observe that $M^{(u)}$ is bounded on ${\mathcal M}^\varphi_q({\mathbb R}^n)$.
Thus $(a)$ and $(b)$ are equivalent.
Since $f$ is expressed as $f=f_0(\|\cdot\|_\infty)$,
that is, there exists a function $f_0:[0,\infty) \to {\mathbb R}$ 
such that $f(x)=f_0(\|x\|_{\infty})$
for all $x\in{\mathbb R}^n$,
where $\|\cdot\|_\infty$ denotes the $\ell^\infty$-norm,
it follows that $(a)$ and $(c)$ are equivalent.
\end{example}

\begin{example}{\rm \cite[Proposition 2.11]{NNS15}}\label{example:180310-A2}
Let $0<q<\infty$ and $\varphi \in {\mathcal G}_q$.
Define a decreasing function $\varphi^\dagger$ by: 
\begin{equation}\label{eq:180310-201}
\varphi^\dagger(t){\equiv}\varphi(t)t^{-\frac{n}{q}}
\end{equation} 
for $t>0$.
Define
$\displaystyle
h \equiv 
\sum_{j=-\infty}^\infty \frac{\chi_{[0,2^{-j}]^n}}{\varphi(2^{-j})}.
$
Then
$\displaystyle
\sum_{j=l}^\infty \frac{1}{\varphi(2^{-j})} \lesssim \frac{1}{\varphi(2^{-l})}
$
and that
\[
\sum_{j=\infty}^l \frac{1}{\varphi^\dagger(2^{-j})} \lesssim \frac{1}{\varphi^\dagger(2^{-l})}
\]
for all $l \in {\mathbb Z}$ if and only if
$h \in {\mathcal M}^\varphi_q({\mathbb R}^n)$.

To verify this,
we let $f,g$ be as in Example \ref{example:180310-A1}.
Suppose first $h \in {\mathcal M}^\varphi_q({\mathbb R}^n)$.
Then
\[
\varphi(2^{-l})\left(\frac{1}{|[0,2^{-l}]^n|}\int_{[0,2^{-l}]^n}
\left(\sum_{j=l}^\infty \frac{1}{\varphi(2^{-j})}\right)^q\,dx\right)^{\frac{1}{q}}
\le {\|h\|_{{\mathcal M}^\varphi_q}}.
\]
Thus
$\displaystyle
\sum_{j=l}^\infty \frac{1}{\varphi(2^{-j})} \le 
\frac{\|h\|_{{\mathcal M}^\varphi_q}}{\varphi(2^{-l})}
$
for all $l \in {\mathbb Z}$.
This implies that $f \le g \le h \lesssim f$,
where $f$ and $g$ are defined in Example \ref{example:180310-A1}.
Thus from Example \ref{example:180310-A1}, 
$\displaystyle
\sum_{j=\infty}^l \frac{1}{\varphi^\dagger(2^{-j})} 
\lesssim \frac{1}{\varphi^\dagger(2^{-l})}
$
holds as well.

Conversely, 
assume that 
\begin{equation}\label{eq:150824-3}
\sum_{j=\infty}^l \frac{1}{\varphi^\dagger(2^{-j})} 
\lesssim \frac{1}{\varphi^\dagger(2^{-l})}
\end{equation}
and
\begin{equation}\label{eq:150824-4}
\sum_{j=l}^\infty \frac{1}{\varphi(2^{-j})} \lesssim \frac{1}{\varphi(2^{-l})}
\end{equation}
hold for all $l \in {\mathbb Z}$.
Then we have $h \sim f$ from {(\ref{eq:150824-4})}.
Thus $f\in {\mathcal M}^\varphi_q({\mathbb R}^n)$ by {(\ref{eq:150824-3})},
from which it follows that $h \in {\mathcal M}^\varphi_q({\mathbb R}^n)$.
\end{example}

We further present some examples of the functions
in ${\mathcal M}^\varphi_p({{\mathbb R}^n})$.
\begin{lemma}{\rm \cite[Lemma 2.4]{EGNS14}}\label{lem9}
Let $0<q<\infty$ and $\varphi \in {\mathcal G}_q \cap {\mathbb Z}^{-\frac{n}{q}}$.
Then the function $\psi(x)=\varphi(|x|)$
belongs to ${\mathcal M}^\varphi_q({{\mathbb R}^n})$.
\end{lemma}

\begin{proof}
First note that 
$\varphi \in {\mathbb Z}^{-\frac{n}{q}}$ is equivalent to
\begin{equation}\label{eq:120918-1}
\frac{1}{r^n}\int_0^r\varphi(t)^q t^{n-1}\,dt
\lesssim
\varphi(r)^q\quad(r>0).
\end{equation}
Note that $\varphi(|\cdot|)$
is radial decreasing,
so that
for
all
$a \in {\mathbb R}^n$ and $r>0$,
\begin{align}\label{eq:120918-24}
\left(\frac{1}{|B(a,r)|}\int_{B(a,r)} \varphi(|x|)^q dx
\right)^{\frac{1}{q}}
\le
\left(\frac{1}{|B(r)|}\int_{B(r)} \varphi(|x|)^q dx
\right)^{\frac{1}{q}}.
\end{align}
Combining (\ref{eq:120918-1}) and (\ref{eq:120918-24})
and using the spherical coordinate,
we obtain the desired result.
\end{proof}

\begin{remark}
See \cite[Theorem 2.1]{SMG12}
for the weak boundedness of the maximal operators,
where the integral conditions is assumed.
\end{remark}

\begin{remark}
See \cite[Theorem 4.2]{Guliyev09}
and
\cite[Theorem 3.4]{AGM12}
for the strong boundedness of the maximal operators,
where the integral conditions is assumed.
See \cite{Nakamura16, Nakai08-1,Sawano08-1}
for the strong boundedness of the maximal operators,
where the integral conditions is not assumed.
\end{remark}

\begin{remark}
See 
\cite[Theorem 1]{BuLi11},
\cite[Theorem 2.9]{ShTa12},
\cite[Theorem 1.9]{TaHe13}
and
\cite[Theorem 2.1]{YuTa14}
for the boundedness of the maximal operator
on generalized Morrey spaces
in the multilinear setting.
\end{remark}

\subsection{Singular integral operators on generalized Morrey spaces}

Let $T$ be a singular integral operator.
To define the function $T f$
for $f \in {\mathcal M}^\varphi_q({\mathbb R}^n)$
we follow the same strategy 
as the one for
$f \in {\mathcal M}^p_q({\mathbb R}^n)$.
To this end, we need to establish the following estimate:
\begin{lemma}\label{lem:180414-9}
Let $1<q<\infty$ and $\varphi \in {\mathcal G}_q$
satisfy $\frac{1}{\varphi}\in {\mathbb Z}_0$.
Then
$
\|T f \|_{{\mathcal M}^\varphi_q}
\lesssim
\|f\|_{{\mathcal M}^\varphi_q}
$
for all $f \in L^\infty_{\rm c}({\mathbb R}^n)$.
\end{lemma}

We note that
$\frac{1}{\varphi}\in {\mathbb Z}_0$
appeared once again.
\begin{proof}
Let $Q$ be a fixed cube.
Then we need to prove
\[
\varphi(\ell(Q))
\left(
\frac{1}{|Q|}\int_Q|T f(y)|^q\,dy
\right)^{\frac1q}
\lesssim
\|f\|_{{\mathcal M}^\varphi_q}.
\]
To this end, we decompose $f$
according to $2Q$:
$f_1=\chi_{2Q}f$,
$f_2=f-f_1$.
As for $f_1$,
we have
\begin{align*}
\varphi(\ell(Q))
\left(
\frac{1}{|Q|}\int_Q|T f_1(y)|^q\,dy
\right)^{\frac1q}
&\le
\varphi(\ell(Q))
\left(
\frac{1}{|Q|}\int_{{\mathbb R}^n}|T f_1(y)|^q\,dy
\right)^{\frac1q}\\
&\lesssim
\varphi(\ell(Q))
\left(
\frac{1}{|Q|}\int_{{\mathbb R}^n}|f_1(y)|^q\,dy
\right)^{\frac1q}\\
&\lesssim
\varphi(\ell(Q))
\left(
\frac{1}{|Q|}\int_{2Q}|f(y)|^q\,dy
\right)^{\frac1q}\\
&\lesssim
\|f\|_{{\mathcal M}^\varphi_q}.
\end{align*}
As for $f_2$, we use the size condition of $K$,
the integral kernel of $T$,
to have the local estimate:
\[
|T f_2(y)| \lesssim
\int_{{\mathbb R}^n \setminus 2Q}
\frac{|f(y)|\,dy}{|y-c(Q)|^n}
\lesssim
\int_{\ell(Q)}^\infty \left(\frac{1}{\ell^{n+1}}\int_{B(c(Q),\ell)}|f(y)|\,dy\right)\,d\ell.
\]
By the definition of the norm,
Lemma \ref{lem:180307-1} and $(\ref{eq:Nakai-19})$,
we obtain
\[
|T f_2(y)| \lesssim
\int_{\ell(Q)}^\infty \frac{1}{r\varphi(r)}\,dr
\cdot
\|f\|_{{\mathcal M}^\varphi_q}
\lesssim
\frac{1}{\varphi(\ell(Q))}
\|f\|_{{\mathcal M}^\varphi_q}.
\]
It remains to integrate this pointwise estimate.
\end{proof}

To carry out our program of proving the boundedness
of the singular integral operators,
we need to investigate the predual and its predual.
\begin{definition}\label{defi:180414-1}
Let $1<q<\infty$ and $\varphi \in {\mathcal G}_q$.
\begin{enumerate}
\item
{\rm \cite[Definition 2.3]{KoMi06}, \cite[Definition 4]{Shirai06-2}}
A $(\varphi,q)$-block is said to be a measurable function $A$
supported on a cube $Q$ satisfying
$\|A\|_{L^{q'}} \le |Q|^{-\frac1q}\varphi(\ell(Q))$.
In this case
$A$ is said to be a $(\varphi,q)$-block supported on $Q$.
\item
{\rm \cite[Definition 2.5]{KoMi06}, \cite[Definition 5]{Shirai06-2}}
The block space ${\mathcal H}^\varphi_{q'}({\mathbb R}^n)$
is the set of all measurable functions $f$
for which it can be written
\[
f=\sum_{j=1}^\infty \lambda_j A_j,
\]
for some sequence $\{A_j\}_{j=1}^\infty$ of $(\varphi,q)$-blocks
and $\{\lambda_j\}_{j=1}^\infty \in \ell^1({\mathbb R}^n)$.
The norm
$\|f\|_{{\mathcal H}^{\varphi}_{q'}}$
is the infimum of $\|\{\lambda_j\}_{j=1}^\infty\|_{\ell^1}$
where
$\{A_j\}_{j=1}^\infty$
and
$\{\lambda_j\}_{j=1}^\infty$ run over all expressions above.
\end{enumerate}
\end{definition}

\begin{example}\label{example:180611-29}
Let $1<q<\infty$ and $\varphi \in {\mathcal G}_q$.
Let $A$ be a non-zero $L^{q'}({\mathbb R}^n)$-function supported on a cube $Q$.
Then
$B=\dfrac{\varphi(\ell(Q))}{|Q|^{\frac1q}\|A\|_{L^{q'}}}A$
is a $(\varphi,q)$-block supported on $Q$.
\end{example}

\begin{proposition}\label{prop:180613-1}
Let $1<q<\infty$ and $\varphi \in {\mathcal G}_q$
be such that $\varphi(t) \gtrsim t^{\frac{n}{q}}$ for alll $t>0$.
Then a measurable function $f$ belongs to
${\mathcal H}^\varphi_{q'}({\mathbb R}^n)$
if and only if $f$ admits a decomposition:
\[
f=\lambda_0 B+\sum_{j=1}^\infty \lambda_j A_j,
\]
for some sequence $\{A_j\}_{j=1}^\infty$ of $(\varphi,q)$-blocks
supported on cubes of volume less than or equal to $1$
and $\{\lambda_j\}_{j=1}^\infty \in \ell^1({\mathbb R}^n)$
and $B \in L^{q'}({\mathbb R}^n)$ with unit norm.
Furthermore
the norm
$\|f\|_{{\mathcal H}^{\varphi}_{q'}}$
is the infimum of $\|\{\lambda_j\}_{j=0}^\infty\|_{\ell^1}$
where
$\{A_j\}_{j=1}^\infty$,
$B$
and
$\{\lambda_j\}_{j=1}^\infty$ run over all expressions above.
\end{proposition}

\begin{proof}
Let $A$ be a $(\varphi,q)$-block supported on $Q$ with $\ell(Q) \gg 1$.
Then
we can say that $A$
is a $(\varphi,q)$-block suppported on $Q$
is and only if $2A$ has the $L^{q'}({\mathbb R}^n)$-norm $1$
and $A$ is supported on $Q$.
So, in the decomposition in Definition \ref{defi:180414-1}
any block $A_j$ with the cube $Q_j$ satisfying $|Q_j|\gg 1$
can be combined into a block supported on \lq \lq ${\mathbb R}^n$".
\end{proof}

The following lemma justifies the definition above.
\begin{lemma}{\rm \cite[Lemma 2]{Shirai06-2}}\label{lem:180815-1}
Let $1<q<\infty$ and $\varphi \in {\mathcal G}_q$.
If $A$ is a $(\varphi,q)$-block and $f \in {\mathcal M}^\varphi_q({\mathbb R}^n)$,
then
$\|A\cdot f\|_{L^1} \le \|f\|_{{\mathcal M}^\varphi_q}$.
\end{lemma}

\begin{proof}
Since $A$ is a $(\varphi,q)$-block,
we can find a cube $Q$ such that
${\rm supp}(A) \subset Q$ and that
$\|A\|_{L^{q'}} \le |Q|^{-\frac1q}\varphi(\ell(Q))$.
By the H\"{o}lder inequality,
\[
\|A\cdot f\|_{L^1}
=\|A\cdot f\chi_Q\|_{L^1}
\le
\|A\|_{L^{q'}}
\|f\chi_Q\|_{L^q}
\le|Q|^{-\frac1q}\varphi(\ell(Q))
\|f\chi_Q\|_{L^q}
 \le \|f\|_{{\mathcal M}^\varphi_q},
\]
as required.
\end{proof}

About the definition above,
the following proposition is fundamental:
\begin{proposition}{\rm \cite[Lemma 4.2]{KoMi06}}\label{prop:180310-1111}
Let $1<q<\infty$ and $\varphi \in {\mathcal G}_q$.
Let $f \in {\mathcal M}^\varphi_q({\mathbb R}^n)$
and
$g \in {\mathcal H}^{\varphi}_{q'}({\mathbb R}^n)$.
Then
$\|f \cdot g\|_{L^1} \le \|f\|_{{\mathcal M}^\varphi_q}\|g\|_{{\mathcal H}^\varphi_{q'}}$.
\end{proposition}

\begin{proof}
Let $\varepsilon>0$ be fixed.
Then we can decompose $f$ as
$$
f=\sum_{k\in K}\lambda_kb_k
$$
where $K\subset{\mathbb N}$ is an index set,
\begin{equation}\label{eq:140814-5}
\sum_{k\in K}|\lambda_k| \le (1+\varepsilon)\|f\|_{{\mathcal H}^{\varphi}_{q'}}
\end{equation}
and each $b_k$ is a $(\varphi,q)$-block.
According to Lemma \ref{lem:180815-1},
we have
\[
\|f \cdot g\|_{L^1}
\le
\sum_{k \in K}|\lambda_k|\|f\|_{{\mathcal M}^\varphi_q}
\le(1+\varepsilon)\|f\|_{{\mathcal M}^\varphi_q}\|g\|_{{\mathcal H}^\varphi_{q'}}.
\]
Consequently
\[
\|f \cdot g\|_{L^1}
\le(1+\varepsilon)
\|f\|_{{\mathcal M}^\varphi_q}\|g\|_{{\mathcal H}^\varphi_{q'}}.
\]
Since $\varepsilon>0$ is arbitrary,
it follows that
\[
\|f \cdot g\|_{L^1}
\le
\|f\|_{{\mathcal M}^\varphi_q}\|g\|_{{\mathcal H}^\varphi_{q'}}.
\]
\end{proof}

It is easy to see that
${\mathcal H}^{\varphi}_{q'}({\mathbb R}^n)$
is a normed space.
Similar to the classical case,
we can prove the following theorem:
\begin{theorem}\label{theorem2.2-generalized}
Let $1<q<\infty$, $\varphi \in {\mathcal G}_q$, 
and let
$f\in{\mathcal H}^{\varphi}_{q'}({\mathbb R}^n)$.
Then $f$ can be decomposed as
$$
f=\sum_{Q\in{\mathcal D}}\lambda(Q)b(Q),
$$
where $\lambda(Q)$ is a non-negative number with
$$
\sum_{Q\in{\mathcal D}}\lambda(Q)\le 9^n
\|f\|_{{\mathcal H}^{\varphi}_{q'}}
$$
and $b(Q)$ is a $(\varphi,q)$-block supported in $Q$.
\end{theorem}

\begin{proof}
We suppose that $\|f\|_{{\mathcal H}^{\varphi}_{q'}}<1$.
Here and below we let $D \gg 1$ is a fixed constant depending on $\varphi$.
It suffices to find a decomposition 
$$
f=\sum_{Q\in{\mathcal D}}\lambda(Q)b(Q),
$$
where $\lambda(Q)$ is a non-negative number with
$$
\sum_{Q\in{\mathcal D}}\lambda(Q)<D.
$$
First, decompose $f$ as
$$
f=\sum_{k\in K}\lambda_kb_k
$$
where $K\subset{\mathbb N}$ is an index set,
\begin{equation}\label{eq:140814-5a}
\sum_{k\in K}|\lambda_k|<1
\end{equation}
and each $b_k$ is a $(\varphi,q)$-block.
We will divide $K$ into the disjoint sets
$K(Q)\subset{\mathbb N}$, $Q\in{\mathcal D}$,
as
$$
K=\bigcup_{Q\in{\mathcal D}}K(Q)
$$
so that
${\rm supp }(b_k)\subset 3Q$
and 
$|Q_k|\ge|Q|$
whenever 
$k\in K(Q)$.
We achieve this as follows:
Let
${\mathcal D}=\{Q^{(j)}\}_{j=1}^\infty$
be an emumeration of ${\mathcal D}$.
For each $k \in K$,
we write
\[
j_k\equiv
\min\{j\,:\,{\rm supp}(b_k) \subset 3Q, \, |Q_k| \ge |Q|\}
\]
for each $k$.
We set
\[
K(Q^{(j)})\equiv\{l \in L\,:\,j_k=j\}.
\]
We set
\begin{align*}
\lambda(Q)\equiv 3^n\sum_{k\in K(Q)}|\lambda_k|, \quad
b(Q)\equiv
\begin{cases}
\dfrac{1}{\lambda(Q)}
\sum_{k\in K(Q)}\lambda_kb_k&\lambda(Q) \ne 0,\\
0&\lambda(Q) \ne 0.
\end{cases}
\end{align*}

We now rewrite $f$ as
\begin{align*}
f
&=
\sum_{k\in K}\lambda_kb_k
=
\sum_{Q\in{\mathcal D}}
\left(\sum_{k\in K(Q)}\lambda_kb_k\right)
\\ &=
\sum_{Q\in{\mathcal D}}
\left\{3^n\sum_{k\in K(Q)}|\lambda_k|\right\}
\cdot
\left\{
\left(3^n\sum_{k\in K(Q)}|\lambda_k|\right)^{-1}
\sum_{k\in K(Q)}\lambda_kb_k
\right\}.
\end{align*}
By (\ref{eq:140814-5a}),
we have
$$
\sum_{Q\in{\mathcal D}}\lambda(Q)
=
3^n
\sum_{Q\in{\mathcal D}}
\left(\sum_{k\in K(Q)}|\lambda_k|\right)
=
3^n\sum_{k\in K}|\lambda_k|
<3^n.
$$
Since each $b_k$ is a $(\varphi,q)$-block,
we obtain
\begin{align*}
\left(3^n\sum_{k\in K(Q)}|\lambda_k|\right)^{-1}
\left\|\sum_{k\in K(Q)}\lambda_k b_k\right\|_{{q'}}
&\le
\left(3^n\sum_{k\in K(Q)}|\lambda_k|\right)^{-1}
\sum_{k\in K(Q)}
|\lambda_k|\cdot \|b_k\|_{{q'}}
\\ &\le
\left(3^n\sum_{k\in K(Q)}|\lambda_k|\right)^{-1}
|Q|^{\frac1p-\frac1q}
\sum_{k\in K(Q)}|\lambda_k|\\
&\lesssim_{p,q}|3Q|^{\frac1p-\frac1q},
\end{align*}
which implies that
$b(Q)$ is a $(\varphi,q)$-block supported in $3Q$
modulo a multiplicative constant.
These complete the proof.
\end{proof}

\begin{example}\label{example:180611-291}
Let $1<q<\infty$, $\varphi \in {\mathcal G}_q$, 
and let
$f\in{\mathcal H}^{\varphi}_{q'}({\mathbb R}^n)$.
Assume that $\inf_{t>0}t^{-\frac{n}{q}}\varphi(t)>0$,
so that Proposition \ref{prop:180613-1} is applicable.
Then $f$ can be decomposed as
$$
f=B+\sum_{Q\in{\mathcal D}, |Q| \le 1}\lambda(Q)b(Q),
$$
where $\lambda(Q)$ is a non-negative number with
$$
\|B\|_{L^{q'}}+
\sum_{Q\in{\mathcal D}}\lambda(Q)\lesssim
\|f\|_{{\mathcal H}^{\varphi}_{q'}}
$$
and $b(Q)$ is a $(\varphi,q)$-block supported in $Q$.
\end{example}

\begin{corollary}\label{cor:180414-16}
Let $1<q<\infty$ and $\varphi \in {\mathcal G}_q$.
Then every function in ${\mathcal H}^{\varphi}_{q'}({\mathbb R}^n)$
is locally integrable.
\end{corollary}

\begin{proof}
Simply combine
Lemma \ref{lem:180307-1},
Proposition \ref{prop:180310-1111}
and the fact that 
$\chi_Q \in {\mathcal M}^\varphi_q({\mathbb R}^n)$
for any cube $Q$.
\end{proof}

\begin{proposition}\label{prop:180310-1112}
Let
$1<q<\infty$ and $\varphi \in {\mathcal G}_q$.
Assume in addition that $\varphi$ satisfies 
$(\ref{eq:Nakai-19})$.
Suppose that
$f$ and $f_k$, $(k=1,2,\ldots)$,
are nonnegative,
that each $f_k \in {\mathcal H}^{\varphi'}_{q'}({\mathbb R}^n)$,
that $\|f_k\|_{{\mathcal H}^{\varphi'}_{q'}}\le 1$
and that $f_k\uparrow f$ a.e. Then
$f\in{\mathcal H}^{\varphi'}_{q'}({\mathbb R}^n)$ and
$\|f\|_{{\mathcal H}^{\varphi'}_{q'}}\le 1$.
\end{proposition}

\begin{proof}
By Theorem \ref{theorem2.2-generalized}
$f_k$ can be decomposed as
$$
f_k
=
\sum_{Q\in{\mathcal D}}\lambda_k(Q)b_k(Q),
$$
where $\lambda_k(Q)$ is a non-negative number with
\begin{equation}\label{2.1-generalized}
\sum_{Q\in{\mathcal D}}\lambda_k(Q)\le 2\cdot 9^n
\end{equation}
and $b_k(Q)$ is a $(\varphi,q)$-block supported in $Q$ and
\begin{equation}\label{2.2-generalized}
\|b_k(Q)\|_{{q'}}
\le
\frac{\varphi(\ell(Q))}{|Q|^{\frac1q}}.
\end{equation}
Using \eqref{2.1-generalized}, \eqref{2.2-generalized} and
the weak-compactness of the Lebesgue space
$L^{q'}(Q)$ 
we now apply a diagonalization argument and, hence,
we can select an 
increasing sequence 
$\{k_j\}_{j=1}^\infty$
of integers
that satisfies the following:
\begin{gather}
\label{2.3-generalized}
\lim_{j\to\infty}\lambda_{k_j}(Q)=\lambda(Q),
\\ 
\label{2.4-generalized}
\lim_{j\to\infty}b_{k_j}(Q)=b(Q)
\text{ in the weak-topology of }
L^{q'}(Q),
\end{gather}
where $b(Q)$ is a $(\varphi,q)$-block supported in $Q$.
We set
$$
f_0\equiv \sum_{Q\in{\mathcal D}}\lambda(Q)b(Q).
$$
Then, by the Fatou theorem and \eqref{2.1-generalized},
\begin{equation}\label{2.51-generalized}
\sum_{Q\in{\mathcal D}}\lambda(Q)
\le \liminf_{j\to\infty}
\sum_{Q\in{\mathcal D}}\lambda_{k_j}(Q)
\le 2\cdot 9^n,
\end{equation}
which implies
$f_0\in{\mathcal H}^{\varphi}_{q'}({\mathbb R}^n)$.

We will verify that
\begin{equation}\label{2.65-generalized}
\lim_{j\to\infty}
\int_{Q_0}f_{k_j}(x)\,dx
=
\int_{Q_0}f_0(x)\,dx
\end{equation}
for all $Q_0\in{\mathcal D}$.
Once \eqref{2.65-generalized} is established,
we will see that $f=f_0$ and hence
$f\in{\mathcal H}^{\varphi}_{q'}({\mathbb R}^n)$
by virtue
of the Lebesgue differentiation theorem
because at least we know that
$f_0$ locally in $L^{q'}({\mathbb R}^n)$.
By the definition of $f_{k_j}$,
we have
\begin{align*}
\int_{Q_0}f_{k_j}(x)\,dx
&=
\sum_{l=-\infty}^\infty
\sum_{\substack{Q \in {\mathcal D}_l\\ Q_0 \cap Q \ne \emptyset}}
\lambda_{k_j}(Q)
\int_{Q_0}b_{k_j}(Q)(x)\,dx.
\end{align*}
Note that
\begin{equation}\label{eq:131108-152-generalized}
\|b_{k_j}(Q)\|_{1}
\le
|Q_0 \cap Q|^{\frac1q}\|b_{k_j}(Q)\|_{{q'}}
\le 
\frac{\varphi(\ell(Q))|Q \cap Q_0|^{\frac1q}}{|Q|^{\frac1q}}
\le 
\frac{\varphi(\ell(Q))}{|Q|^{\frac1q}}
\end{equation}
for any cube $Q$ containing $Q_0$.
If 
\[
\lim_{t \to \infty}\varphi(t)t^{-\frac{n}{q}}=0,
\]
then 
for all $\varepsilon>0$
there exists $l \in {\mathbb N}$ such that
\[
\sum_{l=N}^\infty
\sum_{\substack{Q \in {\mathcal D}_l\\ Q_0 \cap Q \ne \emptyset}}
\left|\lambda_{k_j}(Q)
\int_{Q_0}b_{k_j}(Q)(x)\,dx\right|<\varepsilon.
\]
Thus, we are in the position of using the Lebesgue convergence theorem
based on Example \ref{example:180613-1}.
Thus
we obtain (\ref{2.65-generalized}).
If 
\[
\lim_{t \to \infty}\varphi(t)t^{-\frac{n}{q}}>0,
\]
then we go through a similar argument
using Example \ref{example:180611-291}
to obtain (\ref{2.65-generalized}).

Since $f_k\uparrow f$ a.e.,
we must have by \eqref{2.65-generalized}
$$
\int_{Q_0}f(x)\,dx
=
\int_{Q_0}f_0(x)\,dx
$$
for all $Q_0\in{\mathcal D}$.
This yields $f=f_0$ a.e.,
by the Lebesgue differentiation theorem, and, hence,
$f\in{\mathcal H}^{\varphi}_{q'}({\mathbb R}^n)$.
Since we have verified
$f\in{\mathcal H}^{\varphi}_{q'}({\mathbb R}^n)$,
it follows that
$$
\|f\|_{{\mathcal H}^{\varphi}_{q'}}
=
\sup\left\{
\left|\int_{{\mathbb R}^n}f_k(x)g(x)\,dx\right|:\,
k=1,2,\ldots,\,
\|g\|_{{\mathcal M}^\varphi_q}\le 1
\right\}
\le 1.
$$
This completes the proof of the theorem.

\end{proof}

The proof of the following theorem
is completely the same as the classical case
once Proposition \ref{prop:180310-1112} is proved.
\begin{theorem}\label{thm:180611-13}
Let
$1<q<\infty$, 
and let $\varphi \in {\mathcal G}_q$
satisfy $\frac{1}{\varphi} \in {\mathbb Z}_0$.
\begin{enumerate}
\item
The dual of ${\mathcal H}^\varphi_{q'}({\mathbb R}^n)$
is ${\mathcal M}^\varphi_q({\mathbb R}^n)$.
More precisely, we have the following mappings:
\begin{enumerate}
\item
Any $f \in {\mathcal M}^\varphi_q({\mathbb R}^n)$
defines a continuous functional $L_f$ by:
$$\displaystyle
L_f: g \mapsto \int_{{\mathbb R}^n} f(x)g(x)\,dx \in {\mathbb C}
$$
on ${\mathcal H}_{q'}^\varphi({\mathbb R}^n)$.
\item
Conversely, every continuous functional $L$
on ${\mathcal H}^\varphi_{q'}({\mathbb R}^n)$ can be realized
with $f \in {\mathcal M}^\varphi_q({\mathbb R}^n)$
as $L=L_f$.
\item
The correspondence
$f \in {\mathcal M}^\varphi_q({\mathbb R}^n)
\mapsto
L_f \in
({\mathcal H}^\varphi_{q'}({\mathbb R}^n))^*$
is an isomorphism.
Furthermore
\begin{equation}\label{eq:5-1}
\|f \|_{{\mathcal M}^\varphi_q}
=\sup\left\{
\left|\int_{{\mathbb R}^n} f(x)g(x)\,dx\right|
\,:\,g \in {\mathcal H}^\varphi_{q'}({\mathbb R}^n), 
\| g \|_{{\mathcal H}^\varphi_{q'}}=1
\right\}
\end{equation}
and
\begin{equation}\label{eq:5-1a}
\| g \|_{{\mathcal H}^\varphi_{q'}}
=\sup\left\{
\left|\int_{{\mathbb R}^n} f(x)g(x)\,dx\right|
\,:\,f \in {\mathcal M}^\varphi_q({\mathbb R}^n), 
\|f \|_{{\mathcal M}^\varphi_q}=1\right\}.
\end{equation}
\end{enumerate}
\item
The dual of 
$\widetilde{\mathcal M}^\varphi_q({\mathbb R}^n)$
is
${\mathcal H}^\varphi_{q'}({\mathbb R}^n)$
in the following sense.
\begin{enumerate}
\item
Any $f \in {\mathcal H}^\varphi_{q'}({\mathbb R}^n)$
defines a continuous functional $L'_f$ by:
$$\displaystyle
L_f: g \mapsto \int_{{\mathbb R}^n} f(x)g(x)\,dx \in {\mathbb C}
$$
on $\widetilde{\mathcal M}_{q}^\varphi({\mathbb R}^n)$.
\item
Conversely, every continuous functional $L$
on $\widetilde{\mathcal M}^\varphi_{q}({\mathbb R}^n)$ can be realized
with $f \in {\mathcal H}^\varphi_{q'}({\mathbb R}^n)$
as $L=L'_f$.
\item
The correspondence
$f \in {\mathcal H}^\varphi_{q'}({\mathbb R}^n)
\mapsto
L'_f \in
(\widetilde{\mathcal M}^\varphi_{q}({\mathbb R}^n))^*$
is an isomorphism.
Furthermore
\begin{equation}\label{eq:5-1p}
\|f \|_{\widetilde{\mathcal M}^\varphi_q}
=\sup\left\{
\left|\int_{{\mathbb R}^n} f(x)g(x)\,dx\right|
\,:\,g \in {\mathcal H}^\varphi_{q'}({\mathbb R}^n), 
\| g \|_{{\mathcal H}^\varphi_{q'}}=1
\right\}
\end{equation}
and
\begin{equation}\label{eq:5-1a p}
\| g \|_{{\mathcal H}^\varphi_{q'}}
=\sup\left\{
\left|\int_{{\mathbb R}^n} f(x)g(x)\,dx\right|
\,:\,f \in \widetilde{\mathcal M}^\varphi_q({\mathbb R}^n), 
\|f \|_{\widetilde{\mathcal M}^\varphi_q}=1\right\}.
\end{equation}
\end{enumerate}
\end{enumerate}
\end{theorem}

\begin{proof}
\
\begin{enumerate}
\item
\begin{enumerate}
\item[{\it (a)}]
This is a consequence of Proposition \ref{prop:180310-1111}.
\item[{\it (b)}]
We let $Q_j=2^{j}[-1,1]^n$.
For the sake of the simplicity we denote
$L^{q'}(Q_j)$ by the set of $L^{q'}$ functions
supported on $Q_j$.
The functional
$\displaystyle g \mapsto L(g)$
is well defined and bounded on $L^{q'}(Q_j)$.
Thus by the duality $L^{p'}$-$L^p$ 
there exists $f_j$ such that
$$L(g)=\int_{Q_j} f_j(x)g(x)\,dx$$
for all $g \in L^{q}(Q_j)$.
By the uniqueness of this theorem
we can find an $L^q_{\rm loc}({\mathbb R}^n)$ function $f$
such that $f|_{Q_j}=f_j$ a.e.
for any $j$.

We will prove $f \in {\mathcal M}^p_q({\mathbb R}^n)$.
For this purpose, we take an arbitrary $Q$ and estimate;
\begin{equation}
I\equiv \varphi(\ell(Q))
\left(\frac{1}{|Q|}\int_Q |f(x)|^{q}\,dx\right)^{\frac{1}{q}}.
\end{equation}

For a fixed cube $Q$ and a fixed function $f$ we set
$$
g(x)\equiv \chi_Q(x){\rm sgn}(f(x))|f(x)|^{q-1}
\quad x \in {\mathbb R}^n.
$$
Then we can write
\begin{equation}
I=\varphi(\ell(Q))
\left(\frac{1}{|Q|}\int_Qf(x)g(x)\,dx\right)^{\frac{1}{q}}
=\varphi(\ell(Q))
\left(\frac{1}{|Q|}L(g)\right)^{\frac{1}{q}}.
\end{equation}
Notice that the function
$\frac{|Q|^{\frac{1}{p}-\frac{1}{q}}}{\|g\|_{{q'}}}g$
is a $(p',q')$-block.
Hence, we have
$|L(g)|\le \varphi(\ell(Q))^{-1}\|L\|_*|Q|^{\frac{1}{q}}\|g\|_{q'}$.
As a result we have $I \le \|L\|_*$.
This is the desired result.
\item[{\it (c)}]
Carefully reexamine the proof
of {\it (a)} and {\it (b)}.
The proof of {\it (c)} is already included in them.
\end{enumerate}
\item
The heart of the matters is to prove {\it (b)};
{\it (a)} and {\it (c)} are obtained similarly to (1).
Let $L':\widetilde{\mathcal M}^\varphi_q({\mathbb R}^n)
\to {\mathbb C}$ be a bounded linear mapping.
Since $\varphi^{-1} \in {\mathbb Z}_0$,
we have $\varphi(t) \lesssim t^{1/P}, 0<t \le 1$ for some $P>1$.
For such $P>0$,
we have $L^P({\mathbb R}^n) \cap L^0_{\rm c}({\mathbb R}^n)
\hookrightarrow {\mathcal M}^\varphi_q({\mathbb R}^n)$.
As a consequence, for each $k \in {\mathbb N}$
the mapping $f \in L^P({\mathbb R}^n) \mapsto L'(\chi_{[-k,k]^n}f) \in {\mathbb C}$
is a bounded linear mapping that can be realized by
some $f_k \in L^{P'}({\mathbb R}^n)$.
Furthermore,
$f_k=f_{k+1}\chi_{[-k,k]^n}$
from the $L~{P}({\mathbb R}^n)$-$L^{P'}({\mathbb R}^n)$-duality.
So,
$\displaystyle f=\lim_{k \to \infty}f_k$ exists almost everywhere.
Let $h \in L^\infty_{\rm c}({\mathbb R}^n)$
with ${\rm supp}(h) \subset [-k,k]^n$ for some $k \in {\mathbb N}$.
As a result 
\[
\int_{{\mathbb R}^n}|h(x)f_k(x)|\,dx
=
L(h\overline{\rm sgn}(f_k))
\le
\|L\|_{\widetilde{\mathcal M}^\varphi_{q} \to{\mathbb C}}\|h\|_{\widetilde{\mathcal M}^\varphi_q}.
\]
Since $h$ is arbitrary,
it follows that 
$\|f_k\|_{{\mathcal H}^\varphi_{q'}} \le 
\|L\|_{\widetilde{\mathcal M}^\varphi_{q} \to{\mathbb C}}$.
Using $f_k=f_{k+1}\chi_{[-k,k]^n}$
and Proposition \ref{prop:180310-1112},
we see that
$f \in {\mathcal H}^\varphi_{q'}({\mathbb R}^n)$
and that
$L_f'|L^\infty_{\rm c}({\mathbb R}^n)=L|L^\infty_{\rm c}({\mathbb R}^n)$.
Since
$L^\infty_{\rm c}({\mathbb R}^n)$
is dense in
$\widetilde{\mathcal M}^\varphi_q({\mathbb R}^n)$,
it follows that 
$L=L_f'$,
as required.
\end{enumerate}
\end{proof}

As we discussed for classical Morrey spaces,
we have the following conclusion.
\begin{theorem}{\rm \cite[Theorem 2]{Nakai94}}\label{thm:180611-12}
Let $1<q<\infty$
and let $\varphi \in {\mathcal G}_q$
satisfy $\frac{1}{\varphi} \in {\mathbb Z}_0$.
Then any singular integral operator,
which is initially defined for $L^\infty_{\rm c}({\mathbb R}^n)$-functions,
can be naturally extended to a bounded linear operator on
${\mathcal M}^\varphi_q({\mathbb R}^n)$.
More precisely, $T$ has the following properties:
\begin{enumerate}
\item
For all $f \in L^\infty_{\rm c}({\mathbb R}^n)$,
$T f \in \widetilde{\mathcal M}^\varphi_q({\mathbb R}^n)$
and
satisfies
$\|T f\|_{\widetilde{\mathcal M}^\varphi_q} \lesssim \|f\|_{\widetilde{\mathcal M}^\varphi_q}$.
In particular $T$ extends to a bounded linear operator
on $\widetilde{\mathcal M}^\varphi_q({\mathbb R}^n)$.
\item
The adjoint operator $T^*$ is bounded
on ${\mathcal H}^\varphi_{q'}({\mathbb R}^n)$.
\item
The adjoint $T^{**}$ of $T^{*}$
is a bounded linear operator
on ${\mathcal M}^\varphi_q({\mathbb R}^n)$.
\end{enumerate}
\end{theorem}

Theorem \ref{thm:180611-12}
extends the result by Peetre in \cite{Peetre69}
from Morrey spaces
to generalized Morrey spaces.
The proof is based on \cite{RoSc16}.
\begin{proof}
Since $\frac{1}{\varphi} \in {\mathbb Z}_0$,
$L^P({\mathbb R}^n) \cap L^\infty_{\rm c}({\mathbb R}^n)$
is a subset of $\widetilde{\mathcal M}^\varphi_q({\mathbb R}^n)$.
Let $f \in L^\infty_{\rm c}({\mathbb R}^n)$.
Then there exists $R>0$ such that $f$ is supported on $Q(R)$.
We can decompose
\[
|T f|\lesssim\chi_{Q(2R)}|T f|+(R+|\cdot|)^{-n}\int_{Q(R)}|f(y)|\,dy
\in \widetilde{\mathcal M}^\varphi_q({\mathbb R}^n),
\]
since
\[
(R+|\cdot|)^{-n\eta} \lesssim (M\chi_{Q(R)})^\eta \in {\mathcal M}^\varphi_q({\mathbb R}^n) 
\quad (q^{-1}<\eta<1)
\]
and hence
\[
(R+|\cdot|)^{-n}=\lim_{L \to \infty}
\chi_{Q(L)}
(R+|\cdot|)^{-n}
\]
in ${\mathcal M}^\varphi_q({\mathbb R}^n)$.
It thus remains to show that
$\|T f\|_{\widetilde{\mathcal M}^\varphi_q} \lesssim \|f\|_{\widetilde{\mathcal M}^\varphi_q}$
for all $f \in L^\infty_{\rm c}({\mathbb R}^n)$.
We argue as follows:
Let $Q$ be a cube.
We need to show
\[
\varphi(\ell(Q))\left(
\frac{1}{|Q|}
\int_{Q}|T f(y)|^q\,dy\right)^{\frac1q}
\lesssim
\|f\|_{{\mathcal M}^\varphi_q}.
\]
We decompose this estimate according to $3Q$:
Since $T$ is known to be $L^q({\mathbb R}^n)$-bounded,
\begin{align*}
\varphi(\ell(Q))\left(
\frac{1}{|Q|}
\int_{Q}|T[\chi_{3Q}f](y)|^q\,dy\right)^{\frac1q}
&\le
\varphi(\ell(Q))\left(
\frac{1}{|Q|}
\int_{{\mathbb R}^n}|T[\chi_{3Q}f](y)|^q\,dy\right)^{\frac1q}\\
&\lesssim
\varphi(\ell(Q))\left(
\frac{1}{|Q|}
\int_{3Q}|f(y)|^q\,dy\right)^{\frac1q}\\
&\le
\varphi(3\ell(Q))\left(
\frac{1}{|Q|}
\int_{3Q}|f(y)|^q\,dy\right)^{\frac1q}\\
&\lesssim
\|f\|_{{\mathcal M}^\varphi_q}.
\end{align*}
Meanwhile, using the size condition of $T$,
we have
\begin{align*}
\varphi(\ell(Q))\left(
\frac{1}{|Q|}
\int_{Q}|T[\chi_{{\mathbb R}^n \setminus 3Q}f](y)|^q\,dy\right)^{\frac1q}
&\le
\varphi(\ell(Q))
\sup_{x \in Q}|T[\chi_{{\mathbb R}^n \setminus 3Q}f](x)|\\
&\lesssim
\varphi(\ell(Q))
\int_{{\mathbb R}^n \setminus 3Q}
\frac{|f(y)|}{|y-c(Q)|^n}\,dy\\
&\lesssim
\varphi(\ell(Q))
\sum_{m=1}^\infty
\frac{1}{|2^m Q|}
\int_{2^{m+1}Q \setminus 2^{m}Q}|f(y)|\,dy.
\end{align*}
Since $\frac{1}{\varphi} \in {\mathbb Z}_0$,
we obtain
\begin{align*}
\lefteqn{
\varphi(\ell(Q))\left(
\frac{1}{|Q|}
\int_{Q}|T[\chi_{{\mathbb R}^n \setminus 3Q}f](y)|^q\,dy\right)^{\frac1q}
}\\
&\lesssim
\varphi(\ell(Q))
\sum_{m=1}^\infty
\frac{1}{\varphi(2^m \ell(Q))}
\frac{\varphi(2^m\ell(Q))}{|2^m Q|}
\int_{2^{m+1}Q \setminus 2^{m}Q}|f(y)|\,dy\\
&\lesssim
\varphi(\ell(Q))
\sum_{m=1}^\infty
\frac{1}{\varphi(2^m \ell(Q))}
\|f\|_{{\mathcal M}^\varphi_q}\\
&\lesssim
\varphi(\ell(Q))
\|f\|_{{\mathcal M}^\varphi_q}
\int_{\ell(Q)}^\infty\frac{1}{\varphi(s)s}\,ds\\
&\lesssim
\|f\|_{{\mathcal M}^\varphi_q}.
\end{align*}
Thus we have
$\|T f\|_{\widetilde{\mathcal M}^\varphi_q} \lesssim \|f\|_{\widetilde{\mathcal M}^\varphi_q}$
for all $f \in L^\infty_{\rm c}({\mathbb R}^n)$.
\end{proof}

We did not use
$(\ref{eq:Nakai-19})$
for the proof of boundedness of the Hardy--Littlewood maximal operator.
However
for the proof of boundedness of the singular integral operators,
$(\ref{eq:Nakai-19})$ is absolutely necessary as the following proposition shows:

\begin{proposition}\label{prop:180310-1113}
Let $1<q<\infty$ and $\varphi \in {\mathcal G}_q$.
If 
$\|R_1 f\|_{{\rm w}{\mathcal M}^\varphi_q} \lesssim \|f\|_{{\mathcal M}^\varphi_q}$
for all $f \in L^\infty_{\rm c}({\mathbb R}^n)$,
then $\frac{1}{\varphi} \in {\mathbb Z}_0$,
where $R_1$ denotes the first Riesz transform.
\end{proposition}

\begin{proof}
Let $V=\{x=(x_1,x_2,\ldots,x_n) \in {\mathbb R}^n\,:\,2x_1>|x|\}$.
Assume that 
$\frac1\varphi \notin {\mathbb Z}_0$,
so that
for any $m \in {\mathbb N}\cap [3,\infty)$
there exists $r_m>0$ such that
$\varphi(2^m r_m) \le 2\varphi(r_m)$.
Then, consider 
$f_m=\chi_{V \cap B( 2^{m-1}r_m) \setminus B(2r_m)}$.
Let $x \in V \cap B(r_m)$.
If $y \in V \cap B( 2^{m-1}r_m) \setminus B(2r_m)$, then $x-y\in V$ and $r_m \le |x-y| \le 2^{m}r_m$.
Thus,
\begin{align}
R_1 f_m(x)
&=\int_{{\mathbb R}^n}\frac{x_1-y_1}{|x-y|^{n+1}}f_m(y)\,dy\notag\\
&=\int_{V \cap B( 2^{m-1}r_m) \setminus B(2r_m)} \frac{y_1}{|y|^{n+1}}\,dy\notag\\
&\ge C\int_{V \cap B( 2^{m-1}r_m) \setminus B(2r_m)} \frac{1}{|y|^{n}}\,dy\notag\\
\label{eq:180707-41}
&=C\log m.
\end{align}
Since $V$ is a cone, we have 
\begin{equation}\label{eq:180707-42}
\chi_{B(r_m)} \lesssim M\chi_{V\cap B(r_m)}.
\end{equation}
We use this estimate and the boundedness of $M$ on 
${\mathcal M}^\varphi_q({\mathbb R}^n)$ to obtain
\[
\varphi(r_m)
\lesssim \|\chi_{B(r_m)}\|_{{\mathcal M}^\varphi_q}
\lesssim \|M\chi_{V\cap B(r_m)}\|_{{\mathcal M}^\varphi_q}
\lesssim \|\chi_{V \cap B(r_m)}\|_{{\mathcal M}^\varphi_q}.
\]
By using the inequality 
$\log m \lesssim |R_1 f_m(x)|$ for $x\in V\cap B(r_m)$
and the boundedness of $R_1$ on ${\mathcal M}^\varphi_q(\ell^u)$, we have
\begin{align*}
\log(m) \varphi(r_m)
&\lesssim 
\|R_1 f_m\|_{{\rm w}{\mathcal M}^\varphi_q}\\
&\lesssim
\|f_m\|_{{\mathcal M}^\varphi_q}\\
&\le
\|\chi_{B(2^m r_m)}\|_{{\mathcal M}^\varphi_q}\\
&\lesssim
\varphi(2^mr_m) \\
&\sim \varphi(r_m).
\end{align*}
This implies $\log m \le D$ where $D$ is independent of $m$, 
contradictory to the fact that $m\ge 3$ is arbitrary.
Hence, there exists some $m_0 \in {\mathbb N}$ such that $\varphi\left(2^{m_0}r \right)>2\varphi(r)$.
Thus the integral condition $(\ref{eq:Nakai-19})$ holds.
\end{proof}

We disprove that $T$ can not be exteded to a bounded linear operator
on $m^p_q({\mathbb R}^n)$.
\begin{example}\label{example:180611-28}
Let $1 \le q \le p<\infty$.
Then
\[
\|R_1 f\|_{{\rm w}m^p_q}
\lesssim
\|f\|_{m^p_q}
\quad (f \in L^\infty_{\rm c}({\mathbb R}^n))
\]
and
\[
\|R_1 f\|_{{\rm w}L^q_{\rm uloc}}
\lesssim
\|f\|_{L^q_{\rm uloc}}
\quad (f \in L^\infty_{\rm c}({\mathbb R}^n))
\]
fail.
In fact,
let $\varphi(t)=\min(t^{\frac{n}{p}},1)$,
$t>0$ as in Example \ref{example:180611-1}.
Then
$\varphi$ fails (\ref{eq:Nakai-19})
because
$\displaystyle
\int_1^\infty \frac{dr}{\varphi(r)}=\infty.
$
\end{example}

We end this section with extension to the vector-valued inequality.
\begin{theorem}\label{thm:180611-1}
Let $1<q<\infty$, $1<u<\infty$ and $\varphi \in {\mathcal G}_q$.
Let also $T$ be a singular integral operartor.
Assume in addition that
$(\ref{eq:Nakai-19})$ holds.
Then
for all $\{f_j\}_{j=1}^\infty \subset {\mathcal M}^\varphi_q({\mathbb R}^n)$,
\[
\left\|\left(\sum_{j=1}^\infty |T f_j|^u\right)^{\frac1u}\right\|_{{\mathcal M}^\varphi_q}
\lesssim
\left\|\left(\sum_{j=1}^\infty |f_j|^u\right)^{\frac1u}\right\|_{{\mathcal M}^\varphi_q}
\]
for all $\{f_j\}_{j=1}^\infty \subset {\mathcal M}^\varphi_q({\mathbb R}^n)$.
\end{theorem}

\begin{proof}
Similar to Theorem \ref{thm:180524-1}.
We also use a well-known inequality:
\[
\left\|\left(\sum_{j=1}^\infty |T f_j|^u\right)^{\frac1u}\right\|_{L^q}
\lesssim
\left\|\left(\sum_{j=1}^\infty |f_j|^u\right)^{\frac1u}\right\|_{L^q}
\]
for all $\{f_j\}_{j=1}^\infty \subset L^q({\mathbb R}^n)$.
\end{proof}

\begin{remark}
One can consider the singular integral operators
having rough kernel.
Let $\Omega:S^{n-1}=\{|x|=1\} \to {\mathbb C}$
be a measurable function having enough integrability
and having integral zero.
Then define $T_\Omega$ by
\[
T_\Omega f(x)=\int_{{\mathbb R}^n}
\frac{\Omega(x-y)}{|x-y|^n}f(y)\,dy.
\]
See \cite{BGGS17, Guliyev13-2} for example.
\end{remark}

\begin{remark}
N. Samko introduced a method
of defining singular integral operators
on generalized Morrey spaces
by considering $p$-admissible singular integral operators
in \cite[Definition 3.3]{Sa(N)13-1}.
See also 
\cite[Theorem 4.5]{GAK11}
and
\cite[Theorem 4.3]{AGM12}.
See
\cite{KGS14} for a similar approach in the weighted setting.
\end{remark}

\begin{remark}
See \cite{Peetre69},
where Peetre proved the boundedness
of the singular integral operators on generalized Morrey--Campanato spaces.
\end{remark}

\begin{remark}
See \cite[Theorem 6.2]{Guliyev09}
and \cite{Softova06}
for the boundedness of the singular integral operators.
\end{remark}

\begin{remark}
See 
\cite[Theorems 5.3, 5.6 and 6.3]{GAKS11}
and
\cite[Theorem 5.6]{GAK11}
for the boundedness of the commutators
generated by singular integral operators and BMO.
\end{remark}

\begin{remark}
Chen and Ding dealt with the parabolic singular integral operators
in \cite{ChDi14}.
\end{remark}

\begin{remark}
Pang, Li and Wang dealt with the oscillatory integral operators
in \cite{PLW13}.
\end{remark}

\begin{remark}
See 
\cite[Theorems 2.3, 2.4 and 3.1]{ShTa12},
\cite[Theorem 1.7]{TaHe13}
and
\cite[Theorem 3.1]{YuTa14}
for the boundedness of the singular integral operator
on generalized Morrey spaces
in the multilinear setting.
\end{remark}

\begin{remark}
\
\begin{enumerate}
\item
Wang discussed the boundedness of the intrinsic square functions
in \cite{Wang14-1},
where Wang did not have to resort to the duality argument.
See also 
\cite[Corollary 6.7]{GAK11},
\cite[Theorems 5 and 6]{GaWu14},
\cite[Section 5.2]{KGS14} and
\cite[Theorem 1.4]{GOS15}
for similar approaches.
See
\cite{CDW10, WuZh14}
for the case of the commutators.
\item
Karaman, Guliyev and Serbetci handled
the pseudo-differential operators in \cite[Section 5.1]{KGS14}.
\item
Karaman, Guliyev and Serbetci handled
the Marcinkiewicz operators in \cite[Section 5.3]{KGS14}.
\item
Karaman, Guliyev and Serbetci handled
the Bochner--Riesz operators in \cite[Section 5.4]{KGS14}.
\end{enumerate}
\end{remark}

\begin{remark}
Let $T$ be a singular integral operator.
In \cite[Theorem 3.2]{KoMi06}
Komori and Mizuhara considered
the operator of the form $(f,g) \in
{\mathcal H}^\varphi_{q'}({\mathbb R}^n) \times 
{\mathcal M}^\varphi_{q}({\mathbb R}^n) \mapsto f \cdot T g-T f \cdot g
\in H^1({\mathbb R}^n)$
and obtained a factorization theorem
as an application of the ${\mathcal M}^\varphi_{q'}({\mathbb R}^n)$-boundednes of
the commutators.
\end{remark}

\subsection{Generalized fractional integral operators in generalized Morrey spaces}
\label{subsection:Boundedness of the generalized fractional integral operator on generalized Morrey spaces}

To consider the operator like $(1-\Delta)^{-1}$,
we are oriented to considering
$$
I_{\rho}f(x)
=
\int_{{\mathbb R}^n}f(y)\frac{\rho(|x-y|)}{|x-y|^n}\,dy
$$
for any suitable function $f$ on ${{\mathbb R}^n}$,
where $\rho:(0,\infty) \to (0,\infty)$ is a suitable measurable function.
Generalized Morrey spaces allow us to consider
more general fractional integral operators.

Let us discuss what condition we need in order to guarantee
that $I_\rho$ enjoys some boundedness property.
As we did in \cite[p. 761]{EGNS14},
we always assume that $\rho$ satisfies the ``$Dini\ condition$"
for $I_\rho$.
\begin{equation}\label{150513-1}
\int_{0}^{1}\frac{\rho(s)}{s}ds<\infty,
\end{equation}
so that $I_\rho\chi_Q(x)$ is finite for any cube $Q$
and $x \in {\mathbb R}^n$. 

In addition, we also assume that $\rho$ satisfies the ``$growth\ condition$":
there exist constants $C>0$ and $0<2k_1<k_2<\infty$ such that
\begin{equation}\label{150513-2}
\sup_{\frac{r}{2}<s\le r}\rho(s)\lesssim \int_{k_1r}^{k_2r}\frac{\rho(s)}{s}ds,\quad r>0,
\end{equation}
as was proposed by Perez \cite{Perez95}.
Condition (\ref{150513-2}) is weaker than the usual {\it doubling condition}:
there exists a constant $D>0$ such that
\begin{equation}\label{eq:D}
\frac{1}{D}\le
\frac{\rho(r)}{\rho(s)}\le D
\end{equation}
whenever
$r>0$ and $s>0$ satisfy
$r \le 2s \le 4r$.
\begin{proposition}\label{prop:180611-22}
If $\rho:(0,\infty) \to (0,\infty)$ satisfies the doubling condition
$(\ref{eq:D})$,
then
\[
\sup_{\frac{r}{2} \le s \le r}\rho(s)
\le 2D^2\int_{\frac{r}{2}}^r \rho(s)\frac{ds}{s}.
\]
\end{proposition}

\begin{proof}
Keeping in mind $\displaystyle\int_{\frac12r}^{r}\frac{ds}{s}=\log 2<1$,
we calculate
$$\displaystyle
\sup_{\frac{r}{2} \le s \le r}\rho(s)
\le D\rho(r)
=
\frac{D}{\log 2}
\int_{\frac{r}{2}}^r \rho(r)\frac{ds}{s}
\le
\frac{D^2}{\log 2}
\int_{\frac{r}{2}}^r \rho(s)\frac{ds}{s}
\le
2D^2
\int_{\frac{r}{2}}^r \rho(s)\frac{ds}{s}.
$$
\end{proof}
\begin{example}\label{example:180414-111}
Let $0<\alpha<\infty$.
\begin{enumerate}
\item
$\rho(t)=t^\alpha$,
which generates $I_\alpha$, 
satisfies the doubling condition.
\item
$\rho(t)=\dfrac{t^\alpha}{\log(e+t)}$ satisfies the doubling condition.
\item
$\rho(t)=\dfrac{t^\alpha}{1+t^M}$ satisfies the doubling condition.
See \cite{KNS00} for an application to Schr\"{o}dinger equations.
\item
$\rho(t)=t^\alpha e^{-t}$ satisfies the growth condition
but fails the doubling condition.
\item
Let $0 \le \gamma<\infty$ and $\beta_1,\beta_2 \in {\mathbb R}$.
We set
\[
\ell^{\mathbb B}(r)
\equiv
\begin{cases}
(1+|\log r|)^{\beta_1}&(0< r \le 1),\\
(1+|\log r|)^{\beta_2}&(1< r<\infty).
\end{cases}
\]
as before.
Then
$\rho(t)=t^\gamma \ell^{\mathbb B}(t)$
satisfies
(\ref{150513-1})
if and only if
$\gamma=0>-1>\beta_1$
or
$\gamma>0$.
Meanwhile
(\ref{eq:D})
is always satisfied.
Noteworthy is the fact that we can tolerate the case
$\gamma=0$ if $\beta_1<-1$.
\end{enumerate}
\end{example}

To check that our example is not so artificial

\begin{definition}
One defines
$(1-\Delta)^{-\frac{s}{2}}f$
by
\[
(1-\Delta)^{-\frac{s}{2}}f=G_s*f,
\]
where $G_s$ is given by:
\[
G_s(x)=\lim_{\varepsilon \downarrow 0}
\frac{1}{(2\pi)^n}
\int_{{\mathbb R}^n}
\frac{\exp(-|\varepsilon \xi|^2)e^{i x \cdot \xi}}{(1+|\xi|^2)^{\frac{s}{2}}}\,d\xi
\quad (x \in {\mathbb R}^n).
\]
The function $G_s$ is called the Bessel kernel.
\end{definition}
The above examples are natural in some sense
but somewhat artificial because
the second example and the third one do not appear
naturally in the context of other areas of mathematics.
Here we present some other examples related
to partial differential equations.
\begin{example}\label{example:180611-27}
Note that the solution to $(1-\Delta)f=g$,
where $f$ is an unknown function and 
$g$ is a give function is given by:
\[
g=(1-\Delta)^{-1}f.
\]
\end{example}

Although it is impossible to find $I_\rho \chi_{B(x,r)}(y), y \in {\mathbb R}^n$,
we still have a partial but important estimate.

\begin{lemma}\label{lem1}
Let $\rho:(0,\infty) \to (0,\infty)$ be a measurable function.
Then
inequality
$
\tilde{\rho}(R/2)\lesssim \,I_\rho \chi_{B(R)}(x)
$
holds whenever $x\in B(R/2)$ and $R>0$.
\end{lemma}

\begin{proof}
Take $x\in B(R/2)$. We write the integral in full:
\begin{align*}
I_\rho \chi_{B(R)}(x)
= \int_{{\mathbb R}^n} \frac{\rho(|x-y|)}{|x-y|^n}\,\chi_{B(R)}(y)\,dy
= \int_{B(R)} \frac{\rho(|x-y|)}{|x-y|^n}\,dy.
\end{align*}
A geometric observation shows that $B(x,R/2)\subseteq B(R)$. Hence, we have
\begin{align*}
I_\rho \chi_{B(R)}(x)
\ge \int_{B(x,R/2)} \frac{\rho(|x-y|)}{|x-y|^n}\,dy
= C \int_0^{R/2} \frac{\rho(s)}{s}\,ds.
\end{align*}
Note that we only use the spherical coordinates to obtain the last integral.
\end{proof}
In the case of the radially symmetric functions,
we can calculate $I_\rho g_R(x)$ for $x$ small.
\begin{lemma}{\rm \cite[Lemma 2.2]{EGNS14}}\label{lem8}
For every $R>0$ and a measurable function
$\theta:(0,\infty) \to [0,\infty)$
satisfying the doubling condition
\begin{equation}\label{eq:Nakai-4}
\theta(s) \sim \theta(r)
\quad (0<r \le s \le 2r),
\end{equation}
the inequality
\[
\int_{2R}^\infty \frac{\theta(t)\rho(t)}{t}\,dt
\lesssim
I_\rho g_R(x) \lesssim
\int_{2R/3}^\infty \frac{\theta(t)\rho(t)}{t}\,dt
\]
holds whenever $x\in \displaystyle B\left(\frac{R}{3}\right)$,
where $g_R(x)\equiv \theta(|x|)\chi_{B(R)^{\rm c}}(x)$.
\end{lemma}

\begin{proof}
We prove the right-hand inequality,
the left-hand inequality being similar.
A geometric observation shows that
$|x-y| \sim |y|$
for all $x \in \displaystyle B\left(\frac{R}{3}\right)$ and
$y \in {\mathbb R}^n \setminus \displaystyle B\left(\frac{2R}{3}\right)$.
Since $\theta$ satisfies $(\ref{eq:Nakai-4})$,
then
\begin{align*}
I_\rho g_R(x)
&=
\int_{{{\mathbb R}^n} \setminus B(R)}\frac{\theta(|y|)\rho(|x-y|)}{|x-y|^n}\,dy\\
&\le
\int_{{{\mathbb R}^n} \setminus B(x,2R/3)}\theta(|y|)
\frac{\rho(|x-y|)}{|x-y|^n}\,dy\\
&=
\int_{{{\mathbb R}^n} \setminus \displaystyle B\left(\frac{2R}{3}\right)}
\theta(|x-y|)\frac{\rho(|y|)}{|y|^n}\,dy\\
&\lesssim
\int_{2R/3}^\infty \frac{\theta(t)\rho(t)}{t}\,dt
\mbox{ for } x \in \displaystyle B\left(\frac{R}{3}\right).
\end{align*}
It remains to write the most right-hand side
in terms of the spherical coordinates.
\end{proof}

For covenience, write 
\begin{equation}
\tilde{\rho}(r)\equiv \int_0^r \frac{\rho(t)}{t}\,dt.
\end{equation}
Sometimes, we are interested in the case
where matters are reduced to the classical fractional integral operators.
\begin{proposition}\label{prop:180611-23}
Let $\rho:(0,\infty) \to (0,\infty)$ be a measurable function
satisfying
$(\ref{150513-2})$.
Then the following are equivalent{\rm:}
\begin{enumerate}
\item[$(a)$]
$\rho(r) \lesssim \,r^{\alpha}$ for all $r>0$.
\item[$(b)$]
$\tilde{\rho}(r)\lesssim r^{\alpha}$ for all $r>0$.
\end{enumerate}
\end{proposition}

\begin{proof}
Clearly $(a)$ implies $(b)$,
since
$\displaystyle \int_0^r s^{\alpha-1}\,ds=\frac{1}{\alpha}r^{\alpha}$.
Let us see $(b)$ implies $(a)$.
Combining
$(\ref{150513-2})$
with
$(b)$, we obtain
\[
\rho(r) \le
\sup_{\frac{r}{2}<s\le r}\rho(s)\lesssim \int_{k_1r}^{k_2r}\frac{\rho(s)}{s}ds
\lesssim \int_0^{k_2 r}\frac{\rho(s)}{s}ds
\lesssim
(k_2r)^{\alpha}
\sim
r^{\alpha}.
\]
\end{proof}

Now we present three different criteria
for the boundedness of $I_\rho $.
We prove the following three theorems
on the boundedness of $I_\rho$
on generalized Morrey spaces.

For the case of $q=1$,
we have the following simple result:
\begin{theorem}{\rm \cite[Theorem 1.3]{EGNS14}}\label{thm:Nakai-2}
Let $1 \le p<\infty$,
and let $\varphi \in {\mathcal G}_p$ and $\psi \in {\mathcal G}_1$.
Let also $\rho:(0,\infty) \to (0,\infty)$ be a measurable function
satisfying
$(\ref{150513-2})$.
Then $I_\rho$ is bounded from
${\mathcal M}^\varphi_p({{\mathbb R}^n})$
to
${\mathcal M}^\psi_1({{\mathbb R}^n})$
only if 
\begin{equation}\label{eq:Nakai-10}
\frac{1}{\varphi(r)}\int_0^r \frac{\rho(t)}{t}\,dt
+
\int_r^\infty \frac{\rho(t)}{t\varphi(t)}\,dt
\lesssim
\frac{1}{\psi(r)} \quad(r>0).
\end{equation}
Furthermore, if $(\ref{eq:Nakai-10})$ is satisfied,
then
$I_\rho$ is bounded from
${\mathcal M}^\varphi_1({{\mathbb R}^n})$
to
${\mathcal M}^\psi_1({{\mathbb R}^n})$.
\end{theorem}

Note that
the left-hand side of (\ref{eq:Nakai-10}) equals
$$
\frac{1}{\varphi(r)}\int_0^r \frac{\rho(t)}{t}\,dt
+
\int_r^\infty \frac{\rho(t)}{t\varphi(t)}\,dt
=
\int_0^\infty \frac{\rho(t)}{t\varphi(\max(r,t))}\,dt.
$$
\begin{proof}[Proof of Theorem \ref{thm:Nakai-2}$
(${\rm Necessity}$)$]
Assume that $I_\rho$ is a bounded linear operator
from ${\mathcal M}^\varphi_p({{\mathbb R}^n})$ to
${\mathcal M}^{\psi}_1({{\mathbb R}^n})$.
Let $r>0$.
By Lemma \ref{lem1} and the doubling property of $\psi$,
we obtain
\begin{align*}
\tilde{\rho}(r)
\lesssim
\frac{1}{r^n}\int_{B(r/2)}I_\rho \chi_{B(r)}(x)\,dx
\le
\frac{1}{r^n}\int_{B(r/2)}I_\rho \chi_{B(r)}(x)\,dx
\le \frac{1}{\psi(r)}
\|I_\rho \chi_{B(r)}\|_{{\mathcal M}^{\psi}_1}
\end{align*}
Since $\psi \in {\mathcal G}_1$ and $I_\rho$ is assumed bounded
from ${\mathcal M}^\varphi_p({{\mathbb R}^n})$ to
${\mathcal M}^{\psi}_1({{\mathbb R}^n})$,
it follows that
$$
\tilde{\rho}(r)
\lesssim 
\frac{1}{\psi(r)}
\|\chi_{B(r)}\|_{{\mathcal M}^\varphi_p}.
$$
Since
$\|\chi_{B(r)}\|_{{\mathcal M}^\varphi_p} \sim \varphi(r)$,
we conclude
$$
\tilde{\rho}(r)
\lesssim
\frac{\varphi(r)}{\psi(r)}.
$$
Let $g_r(x)=\frac{\chi_{B(r)^c}(x)}{\varphi(|x|)}$
for $x \in {\mathbb R}^n$.
By Lemma \ref{lem8} with $\theta=\dfrac{1}{\varphi}$,
we have
\[
\int_r^\infty \frac{\rho(t)}{t\varphi(t)}\,dt
\lesssim
\psi\left(\frac{r}{6}\right)^{-1}\|I_\rho g_r\|_{{\mathcal M}^{\psi}_1}
\lesssim
\psi(r)^{-1}\|g_r\|_{{\mathcal M}^\varphi_1}
\lesssim
\psi(r)^{-1}.
\]
Thus Theorem \ref{thm:Nakai-2} is proved.
\end{proof}

\begin{proof}[Proof of Theorem \ref{thm:Nakai-2}$(${\rm Sufficiency}$)$]
For a ball $B(z,r)$, we let $f_1 \equiv f\chi_{B(z,2r)}$ and $f_2 \equiv f-f_1$.
Then a geometric observation shows
$B(z,r) \subset B(y,3r)$
for all $y \in B(z,2r)$. Hence by the Fubini theorem and the normalization,
\begin{align*}
\int_{B(z,r)}|I_\rho f_1(x)|\,dx
&\le
\int_{B(z,r)}
\left(\int_{B(z,2r)}|f(y)|\frac{\rho(|x-y|)}{|x-y|^n}\,dy\right)\,dx\\
&\le
\int_{B(z,2r)}\left(\int_{B(y,3r)}|f(y)|
\frac{\rho(|x-y|)}{|x-y|^n}\,dx\right)\,dy\\
&=
\int_{B(z,2r)}|f(y)|\,dy
\times
\int_{B(3r)}
\frac{\rho(|x|)}{|x|^n}\,dx.
\end{align*}
By the use of the definition of the Morrey norm,
(\ref{eq:Nakai-10}) and the doubling condition of $\psi$,
we obtain
\begin{align*}
\int_{B(z,r)}|I_\rho f_1(x)|\,dx
&\lesssim \tilde{\rho}(3r)\varphi(2r)^{-1}r^n\\
&\lesssim \tilde{\rho}(3r)\varphi(3r)^{-1}r^n\\
&\lesssim \psi(3r)^{-1}r^n\\
&\lesssim \psi(r)^{-1}r^n.
\end{align*}
Thus the estimate for $f_1$ is valid.
As for $f_2$, we let $x \in B(z,r)$.
Then we have
\[
|I_\rho f_2(x)|
\le
\int_{B(z,2r)^{\rm c}}|f(y)|\frac{\rho(|x-y|)}{|x-y|^n}\,dy
\le
\int_{B(x,r)^{\rm c}}|f(y)|\frac{\rho(|x-y|)}{|x-y|^n}\,dy
\]
and decomposing the right-hand side dyadically
we obtain
\begin{align*}
|I_\rho f_2(x)|
\le \sum_{j=1}^\infty
\int_{B(x,2^jr) \setminus B(x,2^{j-1}r)}
|f(y)|\frac{\rho(|x-y|)}{|x-y|^n}\,dy
\lesssim \int_{2k_1r}^\infty \frac{\rho(t)}{t\varphi(t)}\,dt.
\end{align*}
If we use (\ref{eq:Nakai-10}) once again
and the doubling condition on
$\psi$, then we obtain
$|I_\rho f_2(x)| \lesssim \psi(r)^{-1}$.
Thus the estimate for $f_2$ is valid as well.
\end{proof}

In the following example,
we consider why we need generalized Morrey spaces.

\begin{example}{\rm \cite[Theorem 5.1]{SaWa13}, \cite[Example 5.1]{EGNS14}}\label{p1.3}
Let $s \in (0,n)$ and $\kappa>0$.
Define
$\displaystyle
\psi(r)\equiv \frac{(1+r)^s}{\max(1,\log r^{-1})}
\quad (r>0).
$
Let
$
\rho(r)=
r^{s}\exp(-\kappa r), \,
\varphi(r)=r^{s}
$
for $r>0$.
Then
$\rho$ is a measurable function
satisfying
$(\ref{150513-2})$.
Furthermore, if $0<r<1$,
$\displaystyle
\frac{1}{\varphi(r)}\int_0^r \frac{\rho(t)}{t}\,dt
+
\int_r^\infty \frac{\rho(t)}{t\varphi(t)}\,dt
\sim
\frac{1}{\psi(r)}
$
and if $r\ge 1$,
$\displaystyle
\frac{1}{\varphi(r)}\int_0^r \frac{\rho(t)}{t}\,dt
+
\int_r^\infty \frac{\rho(t)}{t\varphi(t)}\,dt
\lesssim
\frac{1}{\psi(r)}.
$
Thus
$
\|(1-\Delta)^{-\frac{s}{2}}f\|_{{\mathcal M}^{\psi}_1}
\lesssim_{s}
\|I_\rho f\|_{{\mathcal M}^{\psi}_1}
\lesssim
\|f\|_{{\mathcal M}_1^{\frac{n}{s}}}
$
 for all $f\in {\mathcal M}_1^{\frac{n}{s}}({\mathbb R}^n)$.
This calculation shows
we cannot delete $\max(1,\log r^{-1})$
and that
$(1-\Delta)^{-\frac{s}{2}}$ does not map ${\mathcal M}^p_q({\mathbb R}^n)$
to $L^\infty({\mathbb R}^n)$ when
$\frac{n}{p}=s$ and $1<q \le p<\infty$.
\end{example}

Example \ref{p1.3} convince us that
generalized Morrey spaces occur naturally.

\begin{example}\label{example:180611-24}
Let $\dfrac1s=\dfrac1p-\dfrac{\alpha}{n}$
with $1<p<s<\infty$ and $0<\alpha<n$.
Then
$I_\alpha$ does not map
$m^p_1({\mathbb R}^n)$
to
$m^s_1({\mathbb R}^n)$,
since
$\rho(t)=t^\alpha$
and
$\varphi(t)=\max(t^{\frac{n}{p}},1)$
satisfies
\[
\int_r^\infty \frac{\rho(t)}{t\varphi(t)}\,dt=\infty
\]
instead of
\[
\int_r^\infty \frac{\rho(t)}{t\varphi(t)}\,dt\lesssim \frac{1}{\varphi(t)^{\frac{s}{p}}}.
\]
If we consider the truncated fractional maximal operator
$i_\alpha$ given by
\[
i_\alpha f(x)=\int_{|y| \le 1}\frac{f(x-y)}{|y|^{n-\alpha}}\,dy,
\]
then $i_\alpha$ maps 
$m^p_1({\mathbb R}^n)$
to
$m^s_1({\mathbb R}^n)$,
since
\[
\frac{1}{\varphi(r)}\int_0^r \frac{\rho(t)\chi_{(0,1)}(t)}{t}\,dt
+
\int_r^\infty \frac{\rho(t)\chi_{(0,1)}(t)}{t\varphi(t)}\,dt
\lesssim
\frac{1}{\varphi(t)^{\frac{s}{p}}}.
\]
\end{example}

\begin{example}{\rm \cite[Theorem 5.1]{SaWa13}, \cite[Example 5.1]{EGNS14}}\label{ex:5.5}
Let $0<s<n$.
Define
$\varphi(r)\equiv r^{s}$ and $\psi(r)\equiv (1+r)^{-s}\ell^{(-1,0)}(r)$
for $r>0$.
Let $\rho(r)\equiv r^n G_{s}(r)$,
where $G_{s}$ denotes the Bessel kernel,
the kernel of $(1-\Delta)^{\frac{s}{2}}$.
Observe that
$\tilde{\rho}(r) \sim \min(r^{s},1)
$
and hence
$\frac{\tilde{\rho}(r)}{\varphi(r)} \sim \min(1,r^{-s}).$
Note also that
$$
\int_r^\infty \frac{\rho(t)}{t\varphi(t)}\,dt
\sim
\begin{cases}
\log(e/r)&(r<1),\\
r^{n-s}G_{s}(r)&(r \ge 1).
\end{cases}
$$
Then 
$$
\frac{\tilde{\rho}(r)}{\varphi(r)}
+\int_r^\infty \frac{\rho(t)}{t\varphi(t)}\,dt
\sim
\frac{1}{\psi(r)} \quad (r>0).
$$
Hence it follows from Theorem \ref{thm:Nakai-2}
that $\|I_\rho f\|_{{\mathcal M}^\psi_1}
\lesssim \|f\|_{{\mathcal M}^\varphi_1}$,
extending Proposition \ref{p1.3}.
This triple $(\rho,\varphi,\psi)$ fulfills the assumption
$(\ref{eq:Nakai-10})$.
However,
$\frac{\rho}{\varphi} \notin {\mathbb Z}_0$
since $\frac{\rho(r)}{\varphi(r)}={\rm o}(1)$
as $r\downarrow 0$
and $(\ref{eq:Nakai-1})$ fails.
\end{example}

We give a result, which improves Example \ref{p1.3}.

We move on to the Adams type estimate.
\begin{theorem}{\rm \cite[Theorems 1.1 and 1.2]{EGNS14}}\label{thm:Hendra}
Let $1<p<q<\infty$ and $\varphi \in {\mathcal G}_p$.
Assume
that
$\rho:(0,\infty) \to (0,\infty)$ 
satisfies
$(\ref{150513-2})$.
\begin{enumerate}
\item
The operator $I_\rho $ is bounded
from
${\mathcal M}^\varphi_p({{\mathbb R}^n})$
to
${\mathcal M}^{\varphi^{p/q}}_{q}({{\mathbb R}^n})$
if 
\begin{equation}\label{eq:Nakai-1}
\frac{1}{\varphi(r)}\int_0^r \frac{\rho(t)}{t}\,dt
+
\int_r^\infty \frac{\rho(t)}{t\varphi(t)}\,dt
\lesssim\frac{1}{\varphi(r)^{\frac{p}{q}}}
\end{equation}
for all $r>0$.
If $\varphi \in {\mathbb Z}^{-\frac{n}{p}}$,
then $(\ref{eq:Nakai-1})$ is necessary
for the boundedness of $I_\rho $ 
from
${\mathcal M}^\varphi_p({{\mathbb R}^n})$
to
${\mathcal M}^{\varphi^{p/q}}_{q}({{\mathbb R}^n})$.
\item
Assume
$\frac{\rho}{\varphi} \in {\mathbb Z}_0$.
Then $I_\rho$ is bounded
from
${\mathcal M}^\varphi_p({{\mathbb R}^n})$ to
${\mathcal M}^{\varphi^{p/q}}_{q}({{\mathbb R}^n})$ if and only if
\begin{equation}\label{eq:180423-22}
\tilde{\rho}(r) \lesssim
\varphi(r)^{1-p/q} \quad (r>0).
\end{equation}
So,
if
$\frac{\rho}{\varphi} \in {\mathbb Z}_0$,
condition $(\ref{eq:Nakai-1})$
simplies to $(\ref{eq:180423-22})$.
\end{enumerate}
\end{theorem}

\begin{remark}
\
\begin{enumerate}
\item
The first half of \lq \lq only if" part
$(\ref{eq:180423-22})$ is clear
from Theorem \ref{thm:Nakai-2}
with $\psi=\varphi^{p/q}$.
\item
Once we 
assume
$\frac{\rho}{\varphi} \in {\mathbb Z}_0$,
it is eacy to check that
$(\ref{eq:180423-22})$ implies
$(\ref{eq:Nakai-1})$.
Indeed,
if we use
$\frac{\rho}{\varphi} \in {\mathbb Z}_0$
and
$\varphi \in {\mathcal G}_p$,
then we have
\[
\int_{r}^\infty \frac{\rho(s)}{s\varphi(s)}\,ds
\lesssim
\frac{\rho(r)}{\varphi(r)}.
\]
Since $\rho$ satisfies the growth condition,
we have
\[
\int_{r}^\infty \frac{\rho(s)}{s\varphi(s)}\,ds
\lesssim
\frac{\tilde{\rho}(k_2 r)}{\varphi(r)}.
\]
If we use (\ref{eq:180423-22})
and the doubling condition on $\varphi$,
then we obtain
\[
\int_r^\infty \frac{\rho(t)}{t\varphi(t)}\,dt
\lesssim\frac{1}{\varphi(r)^{\frac{p}{q}}}.
\]
\item
For the \lq \lq if" part
we only need the following estimate
of Hedberg-type,
see Lemma \ref{lem:150824-20} below.
\end{enumerate}
\end{remark}

\begin{proof}[Proof of Theorem \ref{thm:Hendra}, necessity]
According to Theorem \ref{thm:Nakai-2},
we have only to show
$$
\int_{2R}^\infty \frac{\rho(t)}{t\varphi(t)}\,dt
\lesssim
\varphi(2R)^{{-p/q}}.
$$
By virtue of Lemma \ref{lem8},
we obtain
\[
\int_{2R}^\infty \frac{\rho(t)}{t\varphi(t)}\,dt \sim
\left(\frac{1}{R^n}\int_{B\left(\frac{R}{3}\right)}I_\rho
g(x)^q\,dx\right)^{\frac{1}{q}} \lesssim \varphi(R)^{{-p/q}}
\|I_\rho g_R\|_{{\mathcal M}^{\varphi^{p/q}}_q}.
\]
Since $I_{\rho}$ is bounded,
we obtain
$$
\int_{2R}^\infty \frac{\rho(t)}{t\varphi(t)}\,dt
\lesssim
\varphi(R)^{{-p/q}}\| g_R\|_{{\mathcal M}^\varphi_p}
\lesssim
\varphi(R)^{{-p/q}}
\left\| \frac{1}{\varphi(|\cdot|)}\right\|_{{\mathcal M}^\varphi_p}.
$$
Recall that we are assuming
$\varphi \in {\mathbb Z}^{-\frac{n}{p}}$.
Now we invoke Lemma \ref{lem9} to conclude
$$
\int_{2R}^\infty \frac{\rho(t)}{t\varphi(t)}\,dt
\lesssim
\varphi(R)^{{-p/q}}
\lesssim
\varphi(2R)^{{-p/q}}.
$$
Thus necessity is proven.
\end{proof}

As we have mentioned,
we want an estimate of Hedberg-type.
We may ask ourselves whether 
$\inf_{r>0}\frac{1}{\varphi(r)^{\frac{p}{q}}}$
can be removed, that is,
we may assume
$\sup_{t>0} \varphi(t)=\infty$.
However, it can happen that
$\sup_{t>0} \varphi(t)<\infty$
as example below shows.
\begin{lemma}{\rm \cite[Lemma 3.1]{EGNS14}}\label{lem:150824-20}
Let $1 \le p<q<\infty$, 
and let $\varphi \in {\mathcal G}_p \cap W$
satisfy $(\ref{eq:Nakai-1})$.
Let $\rho:(0,\infty) \to (0,\infty)$ be a measurable function
satisfying
$(\ref{150513-2})$.
If a measurable function $f$ satsfies $\|f\|_{{\mathcal M}^\varphi_p}=1$,
then 
\begin{equation}\label{eq:180525-11}
|I_\rho f(x)|
\lesssim \,
[Mf(x)]^{p/q}+\inf_{r>0}\varphi(r)^{-\frac{p}{q}}
\end{equation}
for $x\in{{\mathbb R}^n}$.
\end{lemma}

Once this estimate is satisfied,
we can conclude the proof of Theorem \ref{thm:Hendra}
as follows:
We choose an arbitrary ball $B=B(z,r)$.
If we integrate Lemma \ref{lem:150824-20},
then we have
\[
\frac{1}{|B|}\int_B |I_\rho f(x)|^q\,dx
\lesssim \,
\frac{1}{|B|}\int_B [Mf(x)]^{p}\,dx
+\inf_{u>0}\varphi(u)^{-p}.
\]
If we multiply both sides by $\varphi(r)^p$,
then we have
\[
\frac{\varphi(r)^p}{|B|}\int_B |I_\rho f(x)|^q\,dx
\lesssim \,
\left(
\frac{\varphi(r)^p}{|B|}\int_B [Mf(x)]^{p}\,dx
+1\right)
\lesssim 1
\]
by virtue of the boundedness of the maximal operator $M$
on ${\mathcal M}^\varphi_p({\mathbb R}^n)$.
The ball $B$ being arbitrary, we obtain the desired result.
\begin{proof}
Recall that $k_1$ and $k_2$ appeared in the condition
$(\ref{150513-2})$ on $\rho$.
Let $$\rho^*(r) \equiv \int_{k_1r}^{k_2r} \frac{\rho(s)}{s}\,ds.$$
We have
$$
|I_\rho f(x)|
\lesssim \sum_{j=-\infty}^{-1}+ \sum_{j=0}^\infty
\frac{\rho^*(2^jr)}{(2^jr)^n} \int_{|x-y|<2^jr} |f(y)|\,dy
$$
for given $x\in{{\mathbb R}^n} $ and $r>0$. Let
$\Sigma_I$ and $\Sigma_{II}$ be the first and second summations
above.
Now we invoke the {\it overlapping property}:
\begin{equation}\label{overlapping}
\sum_{j=-\infty}^{-1}
\chi_{[2^jk_1r,2^jk_2r]}
\lesssim
\chi_{(-\infty,2^{-1}k_2r]},
\quad
\sum_{j=0}^\infty
\chi_{[2^jk_1r,2^jk_2r]}
\lesssim
\chi_{[k_1r,\infty)}.
\end{equation}
As a result,
we have
\begin{align*}
\sum_{j=-\infty}^{-1} \rho^*(2^jr)
\le
\sum_{j=-\infty}^{-1}
\int_{2^jk_1r}^{2^jk_2r}\frac{\rho(s)}{s}\,ds
\lesssim
\int_{0}^{k_2r}
\frac{\rho(s)}{s}\,ds= \tilde{\rho}(k_2r)
\end{align*}
and
\begin{align*}
\sum_{j=0}^\infty \frac{\rho^*(2^jr)}{\varphi(2^jr)}
\lesssim
\int_{k_1r}^\infty
\left(
\sum_{j=0}^\infty
\chi_{[2^jk_1r,2^jk_2r]}(s)
\right)
\frac{\rho(s)}{s\varphi(s)}\,ds
\lesssim
\int_{k_1r}^\infty
\frac{\rho(s)}{s\varphi(s)}\,ds.
\end{align*}
Thus
thanks to (\ref{eq:180423-22})
\begin{align*}
\Sigma_I
&\lesssim \sum_{j=-\infty}^{-1} \rho^*(2^jr)Mf(x)\le
C\,\tilde{\rho}(k_2r)Mf(x)\lesssim \,\varphi(r)^{1-p/q}Mf(x).
\end{align*}
Meanwhile
\begin{align*}
\Sigma_{II}
&\lesssim \sum_{j=0}^\infty \frac{\rho^*(2^jr)}{\varphi(2^jr)}\|f\|_{{\mathcal M}^\varphi_p}
\lesssim
\int_{k_1r}^\infty \frac{\rho(s)}{s\varphi(s)}\,ds.
\end{align*}
We use
$\frac{\rho}{\varphi} \in {\mathbb Z}_0$
or
(\ref{eq:Nakai-1}) now.
If we use
(\ref{eq:Nakai-1}),
then we have
\[
\int_r^\infty \frac{\rho(t)}{t\varphi(t)}\,dt
\lesssim\frac{1}{\varphi(r)^{\frac{p}{q}}}.
\]
By the doubling property of $\varphi$,
we obtain
$
\Sigma_{II}
\lesssim \,\frac{1}{\varphi(r)^{\frac{p}{q}}}.
$
Hence,
\begin{equation}\label{eq:3.0}
|I_\rho f(x)| \lesssim
,\varphi(r)^{1-p/q}\left(Mf(x)+\frac{1}{\varphi(r)}\right)
\end{equation}
for all $r>0$.

First assume $\displaystyle Mf(x) \le \inf_{r>0}\frac{1}{\varphi(r)}$.
Then,
the conclusion is immediate
from $(\ref{eq:3.0})$.

Next, we assume $\displaystyle Mf(x)>\inf_{r>0}\frac{1}{\varphi(r)}$.
Since $\|f\|_{{\mathcal M}^\varphi_p}=1$,
we have
\[
1\ge
\varphi(r)\left(
\frac{1}{|Q(x,r)|}\int_{Q(x,r)}|f(y)|^p\,dy\right)^{\frac{1}{p}}\ge
\frac{\varphi(r)}{|Q(x,r)|}\int_{Q(x,r)}|f(y)|\,dy.
\]
Hence
\[
\frac{1}{|B(x,r)|}\int_{B(x,r)}|f(y)|\,dy
\le \frac{1}{\varphi(r)}
\]
for all $r>0$. This implies
\[
\frac{1}{|B(x,r)|}\int_{B(x,r)}|f(y)|\,dy
\le \sup_{R>0}\frac{1}{\varphi(R)}
\]
for all $r>0$. Since $r>0$ and $x \in {\mathbb R}^n$ are arbitrary,
it follows that
$Mf(x) \le \sup_{r>0} \frac{1}{\varphi(r)}$.
We can thus find $R>0$ such that
$Mf(x)=2\varphi(R)$
and, with this $R$, we can obtain the desired estimate.
\end{proof}

In order that
$I_\rho$ be bounded
from
${\mathcal M}^\varphi_p({{\mathbb R}^n})$ to
${\mathcal M}^{\varphi^{p/q}}_{q}({{\mathbb R}^n})$,
we must have
\[
\frac{\tilde{\rho}(r)}{\varphi(r)}
\lesssim\frac{1}{\varphi(r)^{\frac{p}{q}}}
\]
according to
Theorem \ref{thm:Nakai-2}
with $\psi=\varphi^{p/q}$.

We note that if $\rho(r)=r^\alpha$, with $0<\alpha<n$, then
$I_\rho =I_{\alpha}$ is the classical fractional integral operator, also
known as the Riesz potential,
which is bounded
from $L^p({{\mathbb R}^n})$ to $L^q({{\mathbb R}^n})$
if and only if
$\frac{1}{p}-\frac{1}{q}=\frac{\alpha}{n}$, where $1<p,\,q<\infty$. 
The necessary part is usually proved by using the
scaling arguments.

Theorem \ref{thm:Hendra} characterizes the
kernel function $\rho$ for which $I_\rho $ is bounded
from $L^p({{\mathbb R}^n})$ to $L^q({{\mathbb R}^n})$
for $1<p<q<\infty$.
We have the following result:
\begin{corollary}{\rm \cite[Corollary 1.5]{EGNS14}}\label{thm1}
Let $1<p<q<\infty$. The operator $I_\rho$ is bounded
from $L^p({{\mathbb R}^n})$ to $L^q({{\mathbb R}^n})$
if and only if $\rho(r) \lesssim \,r^{\frac{n}{p}-\frac{n}{q}}$ for all
$r>0$.
\end{corollary}

For $\rho(r)=r^\alpha$,
Corollary \ref{thm1} further reads that the operator $I_\rho$ is bounded
from $L^p({{\mathbb R}^n})$ to $L^q({{\mathbb R}^n})$ if and only if
$\alpha=\frac{n}{p}-\frac{n}{q}$, where $1<p<q<\infty$.

With Theorems \ref{thm:Nakai-2}--\ref{thm:Hendra}
we can characterize the function $\rho$ for which
$I_\rho $ is bounded from one Morrey space to another.

The next corollary generalizes the previous characterization in
Corollary \ref{thm1}.

\begin{corollary}{\rm \cite[Corollary 1.6]{EGNS14}}\label{thm2}
Assume that the parameters $p,q,s,t$ and $\alpha$ satisfy
\[
1<q \le p<\infty, \quad
1<t \le s<\infty, \quad
0<\alpha<n
\]
and
\[
\frac{1}{s}=\frac{1}{p}-\frac{\alpha}{n}, \quad
\frac{t}{s}=\frac{q}{p}.
\]
Let $\rho:(0,\infty) \to (0,\infty)$ be a function satisfying the growth condition.
Then the generalized fractional integral operator $I_\rho $ is
bounded from ${\mathcal M}^p_q({{\mathbb R}^n})$
to ${\mathcal M}^s_t({{\mathbb R}^n})$
precisely when 
$\rho(r) \lesssim r^\alpha$.
\end{corollary}

We show by examples that 
two statements
in
Theorem
\ref{thm:Hendra} are of independent interest.
As before we write
\[
\ell^{\mathbb B}(r)
\equiv
\begin{cases}
(1+|\log r|)^{\beta_1}&(0< r \le 1),\\
(1+|\log r|)^{\beta_2}&(1< r<\infty).
\end{cases}
\]
This function is used to describe
the \lq \lq log"-growth and \lq \lq log"-decay properties.
Also, we fix $p$ and $q$ so that $1<p<q<\infty$.
The key properties
we are interested in are summarized in the following table:
\begin{center}
\begin{small}
\tabcolsep1.5pt
\begin{tabular}{llllllll}
 & $\frac{\rho}{\varphi} \in {\mathbb Z}_0$ 
 & $\varphi \in{\mathbb Z}^{-\frac{n}{p}}$
 & $(\ref{eq:Nakai-1})$
 &$(\ref{eq:180423-22})$\\
\hline
\hline
Example \ref{ex:5.2} 
&$+$ & $+$ & $+$& $+$\\
\hline
Example \ref{ex:5.1} 
&$+$ & $-$ & $+$& $+$ \\
\hline
Example \ref{ex:5.3} 
&$-$ & $-$ & $+$& $+$ \\
\hline
Example \ref{ex:5.4} 
&$-$ & $-$ & $+$& $+$ \\
\hline
\end{tabular}
\end{small}
\end{center}

\begin{example}{\rm \cite[Example 2]{EGNS14}}\label{ex:5.2}
Let $\lambda<0$ satisfy $0<\left(\frac{p}{q}-1\right)\lambda<n$ and
$-\frac{n}{p}<\lambda$. Take $\mu_1,\mu_2$ arbitrarily. Set
$\beta_i\equiv \left(\frac{p}{q}-1\right)\mu_i$ for $i=1,2$. Define 
$\varphi(r)\equiv r^{-\lambda}\ell_{-\mu_1-,\mu_2}(r)$
and
$\rho(r)=\varphi(r)^{1-\frac{p}{q}}$ for $r>0$. Then this
pair $(\rho,\varphi)$ fulfills the assumptions
$\frac{\rho}{\varphi} \in {\mathbb Z}_0$
and
$\varphi \in{\mathbb Z}^{-\frac{n}{p}}$ in Theorem
{\rm \ref{thm:Hendra}}. Indeed, 
for $r>0$
we have
$\tilde{\rho}(r) \sim \rho(r)=\varphi(r)^{1-\frac{p}{q}}$
and
$$\int_r^\infty
\frac{\rho(t)}{t\varphi(t)}\,dt \sim \frac{\rho(r)}{\varphi(r)}.
$$
\end{example}

Example \ref{ex:5.1} is an endpoint case
of the above example.
\begin{example}{\rm \cite[Example 1]{EGNS14}}\label{ex:5.1}
Let $\mu_1,\mu_2$ satisfy $\mu_1,\mu_2 \ge 0. $ Set
$\alpha\equiv \frac{n}{p}-\frac{n}{q}$ and
$\beta_i\equiv \left(\frac{p}{q}-1\right)\mu_i$ for $i=1,2$. Define 
$\varphi(r)\equiv r^{\frac{n}{q}}\ell_{-\mu_1,-\mu_2}(r)$ for $r>0$
and
$\rho=\varphi^{1-\frac{p}{q}}$. 
We note that
$\tilde{\rho} \sim \rho$.
Then
this pair $(\rho,\varphi)$ fulfills the assumptions
$\frac{\rho}{\varphi}=\varphi^{-\frac{p}{q}} \in {\mathbb Z}_0$ and 
$(\ref{eq:Nakai-1})$
but $\varphi \notin{\mathbb Z}^{-\frac{n}{p}}$
since $\ell_{-\mu_1,-\mu_2} \notin {\mathbb Z}^0$. 
\end{example}

The next example concerns
the case where the spaces are close to $L^\infty({\mathbb R}^n)$
and the smoothing order of $I_\rho $ is \lq \lq almost $0$".

\begin{example}{\rm \cite[Example 3]{EGNS14}}\label{ex:5.3}
Let $\mu_1,\mu_2<0$. Set
$\beta_1\equiv \left(\frac{p}{q}-1\right)\mu_1+1 \in(1,\infty)$ and
$\beta_2\equiv \left(\frac{p}{q}-1\right)\mu_2-1\in(-1,\infty)$.
Define
$\rho\equiv \ell^{\mathbb B}$ as we did in Example \ref{example:180611-12}
and let
$\varphi\equiv \ell_{\mu_1,\mu_2}$. Then this pair
$(\rho,\varphi)$ fulfills $\varphi \notin{\mathbb Z}^{-\frac{n}{p}}$
and assumption
$(\ref{eq:Nakai-1})$ but
$\frac{\rho}{\varphi}=\ell_{\beta_1-\mu_1,\beta_2-\mu_2} \notin {\mathbb Z}_0$.
More precisely, 
we have $\tilde{\rho} \sim
\ell_{\beta_1-1,\beta_2+1}$ since $\beta_1>1$, and 
$$ 
\int_r^\infty
\frac{\rho(t)}{t\varphi(t)}\,dt \sim
\ell_{\mu_1+\beta_1-1,\mu_2+\beta_2+1}(r)
\quad (r>0). 
$$
\end{example}

We consider a case where the target space is close
to $L^\infty({\mathbb R}^n)$.
\begin{example}{\rm \cite[Example 4]{EGNS14}}\label{ex:5.4}
Let $1<p,q<\infty$.
Let $\alpha,\beta_1,\mu_1,\mu_2$ satisfy $ 0<\alpha<\frac{n}{p},
\, \mu_1+\beta_1<1, \, \mu_2<0$. Set
$\beta_2\equiv \left(\frac{p}{q}-1\right)\mu_2-1 \in (-1,\infty).$
Define $ \rho(r)\equiv \min(1,r^\alpha)\ell^{\mathbb B}(r) $
as we did in Example \ref{example:180611-12}
 and let
$
\varphi(r)\equiv \max(1,r^{-\alpha})\ell_{\mu_1,\mu_2}(r) $ for $r>0$. Then
this pair $(\rho,\varphi)$ fulfills 
$\varphi \notin{\mathbb Z}^{-\frac{n}{p}}$
and
assumption $(\ref{eq:Nakai-1})$ but 
$\frac{\rho}{\varphi} \notin {\mathbb Z}_0$ More precisely, 
\[
\frac{\tilde{\rho}(r)}{\varphi(r)} \sim
\ell_{\mu_1+\beta_1,\mu_2+\beta_2+1}(r)
\] 
and 
\[ 
\int_r^\infty
\frac{\rho(t)}{t\varphi(t)}\,dt \sim
\ell_{\mu_1+\beta_1-1,\mu_2+\beta_2+1}(r) 
\] 
for $r>0$.
\end{example}

Based upon these preliminary results and Lemma \ref{lem8},
we will prove Theorems \ref{thm:Hendra}--\ref{thm:Nakai-2}.

We remark that $(\ref{eq:Nakai-1})$ includes
$(\ref{eq:180423-22})$.
We prove an estimate.
Once we prove Lemma \ref{lem:150824-20} below,
we can obtain the boundedness of $I_\rho $ from ${\mathcal M}^\varphi_p({{\mathbb R}^n})$ to
${\mathcal M}^{\varphi^{p/q}}_q({{\mathbb R}^n})$ as we will see below.
Here we use the fact that the Hardy--Littlewood maximal operator $M$ is
bounded on ${\mathcal M}^\varphi_p({{\mathbb R}^n})$,
if $p>1$ and $\varphi$ is almost decreasing;
see Theorem \ref{thm:140525-3}.

We end this section with comparison of our results
with the existing results.
We move on to the case of Spanne type.

\begin{remark}
See \cite[Theorem 5.2]{Guliyev09}
for the Spanne-type boundedness of $I_\alpha$.
See 
\cite[Theorem 5.4]{GAKS11},
\cite[Theorem 1.8]{ShTa09},
\cite[Theorems 2.6, 2.7, 3.4 and 3.6]{ShTa12}
and
\cite[Theorem 1.5]{TaHe13}
for the case of the multilinear setting.
\end{remark}

\begin{remark}
See \cite[Theorem 2.3]{SMG12}
for the weak boundedness of the maximal operators
(on nonhomogeneous spaces),
where the integral conditions is assumed.
\end{remark}

\begin{remark}
See \cite[Theorem 5.5]{Guliyev09}
and
\cite[Theorem 5.7]{GuSh13}
for the Adams-type boundedness of $I_\alpha$.
Persson and Samko obtainted the  Adams-type boundedness of $I_\alpha$
using the weighted Hardy operator;
see \cite[Theorem 5.4]{PeSa11}.
Guliyev and Shukurov also considered a similar situation
in \cite[Theorem 3.6]{GuSh12} and \cite[Theorem 5.7]{GuSh13}.
See \cite[Theorem 2]{BuLi11}
for the multilinear case.
\end{remark}

\begin{remark}\
\begin{enumerate}
\item
See \cite[Theorem 3]{Nakai94},
\cite{GuEr09} 
for the boundedness of
$I_\alpha$ on generalized Morrey spaces.
\item
See \cite[Theorem 7.1]{Nakai08-1}
for the boundedness of the generalized fractional integral operators
on generalized Orlicz Morrey spaces
(of the first kind).
\item
See \cite{EGN04,Gunawan03,GuEr09,MNOS10,Sugano11,SuTa03}
for the study of the boundedness of $I_\rho$.
\item
See \cite{SaSh13-3}
for a different type of generalization of the form:
\[
I f(x)=\int_{{\mathbb R}^n}K(x,y)f(y)\,dy.
\]
\end{enumerate}
\end{remark}

\begin{remark}
In some special case,
some authors obtained
the necessity of the boundedness
of $I_\alpha=I_\rho$ from
generalized Morrey spaces to other generalized Morrey spaces.
See
\cite[Theorems 2.3 and 3.2]{EUG12}
as well as
\cite{Gunawan03},
\cite{SST11-1}.
\end{remark}

\begin{remark}
Kurata and Sugano pointed out that
the operator of the form
$V^\gamma(-\Delta+V)^\beta$
with a potentail $V$ satisfying the reverse H\"{o}lder inequality
falls under the scope of the results in this section
\cite{KuSu00-1}.
Here $\beta,\gamma$ are suitable real parameters.
See \cite{KuSu00-1} for more details.
\end{remark}

\begin{remark}
Many researchers handled various operators.
\begin{enumerate}
\item
In \cite{Eroglu13}
Eroglu dealt with fractional oscillatory integral operators
and their commutators.
\item
In \cite{LiSh11}
Liu and Shi considered the boundedness of the commutator
generated by BMO and the fractional integral operators.
See also 
\cite[Theorem 7.1]{GAKS11},
\cite[Theorem 7.11]{GuSh13}
for the Spanne type result
and
\cite[Theorem 7.13]{GuSh13}
for the Adams type result.
\end{enumerate}
\end{remark}

\subsection{Generalized fractional maximal operators in generalized Morrey spaces}
\label{subsection:Generalized fractional maximal operators}

We discuss the boundedness property of the generalized fractional maximal operator, defined by:
\[
M_{\rho}f(x)=\sup_{r>0} \frac{\rho(r)}{|B(x,r)|} \int_{B(x,r)} |f(y)| \ dy
\quad (x \in {\mathbb R}^n),
\]
where $f\in L^1_{\mathrm{loc}}({\mathbb R}^n)$ and $\rho$ is 
a suitable function from $(0,\infty)$ to $[0,\infty)$.

\begin{example}\label{example:180611-23}
Let $0 \le \alpha<n$.
\begin{enumerate}
\item
If we let
$\rho(t)=t^\alpha$,
then we obtain the 
fractional maximal operator $M_\alpha$;
$M_\rho=M_\alpha$.
\item
If we let
$\rho(t)=\min(t^\alpha,1)$,
then we obtain the 
local fractional maximal operator $m_\alpha$;
$M_\rho=m_\alpha$,
where
\[
m_\alpha f(x)=\sup_{0<r\le 1} \frac{r^\alpha}{|B(x,r)|} \int_{B(x,r)} |f(y)| \ dy
\]

\end{enumerate}
\end{example}

What
$M_\rho$ is to $I_\rho$
is
what
$M_\alpha$ is to $I_\alpha$.
So, we are interested in when
 $M_{\rho}$ is bounded from
${\mathcal M}^{\varphi}_{q}({\mathbb R}^n)$ to ${\mathcal M}^{\psi}_{t}({\mathbb R}^n)$.
We start with the following necessary condition:
\begin{proposition}{\rm \cite[Theorem 1]{HNS15}}\label{prop:180526-1}
Let $1\le q<\infty$ and $(\varphi,\psi) \in {\mathcal G}_q \times {\mathcal G}_1$.
Assume that $M_{\rho}$ is bounded from
${\mathcal M}^{\varphi}_{q}({\mathbb R}^n)$ to ${\rm w}{\mathcal M}^{\psi}_1({\mathbb R}^n)$.
Then
$\rho \lesssim \frac{\varphi}{\psi}$.
In particular, 
if $M_{\rho}$ is bounded from
${\mathcal M}^{\varphi}_{q}({\mathbb R}^n)$ to ${\mathcal M}^{\psi}_1({\mathbb R}^n)$,
then
$\rho \lesssim \frac{\varphi}{\psi}$.
\end{proposition}

\begin{proof}
Let $R>0$ be fixed.
We utilize the pointwise estimate 
$\rho(R)\chi_{B(R)}\le M_{\rho} \chi_{B(2R)}$, 
and the doubling condition of $\varphi$
to obtain 
\begin{align*}
\rho(R) 
\lesssim \frac{\|\rho(R)\chi_{B(R)}\|_{{\rm w}{\mathcal M}^{\psi}_1}}{\psi(R)}
\lesssim \frac{\|M_{\rho} \chi_{B(2R)}\|_{{\mathcal M}^{\psi}_1}}{\psi(R)}
\lesssim \frac{\|\chi_{B(2R)} \|_{{\mathcal M}^\varphi_{q}}}{\psi(R)}
\sim \frac{\varphi(R)}{\psi(R)}.
\end{align*}
\end{proof}

Our first result completely 
characterizes the boundedness of $M_{\rho}$ 
on generalized Orlicz--Morrey spaces.
\begin{theorem}{\rm \cite[Theorem 1]{HNS15}}\label{thm:150206-1}
Let $0<a<1<q<\infty$.
Let $\varphi\in{\mathcal G}_{q}$.
Assume that 
$\displaystyle\lim_{t\downarrow 0} \varphi(t)=0$ and
$\displaystyle\lim_{t\to \infty} \varphi(t)=\infty$. 
Then, $M_{\rho}$ is bounded from
${\mathcal M}^{\varphi}_{q}({\mathbb R}^n)$ to ${\mathcal M}^{\varphi^a}_{a^{-1}q}({\mathbb R}^n)$ 
if and only if 
$\rho$ and $\varphi$ satisfy the inequality 
\begin{equation}\label{eq:150206}
\rho(R)\lesssim \varphi(R)^{1-a}
\end{equation}
for all $R>0$.
\end{theorem}

The proof hinges on the following Hedberg inequality:
\begin{lemma}{\rm \cite[(15)]{HNS15}}\label{lem:180526-1}
Let $0<a<1<q<\infty$.
Let $\varphi\in{\mathcal G}_{q}$.
Assume that 
$\displaystyle\lim_{t\downarrow 0} \varphi(t)=0$ and
$\displaystyle\lim_{t\to \infty} \varphi(t)=\infty$. 
Then for any $f \in {\mathcal M}^{\varphi}_{q}({\mathbb R}^n)$ 
with $\|f\|_{{\mathcal M}^\varphi_q} \le 1$,
\begin{align}\label{eq:150206-2}
M_\rho f(x)
\lesssim
M f(x)^a
\quad (x \in {\mathbb R}^n).
\end{align}
\end{lemma}

Once Lemma \ref{lem:180526-1} is proved,
we have only to resort to the scaling law
(Lemma \ref{lem:150318-1})
and the boundedness of $M$ on ${\mathcal M}^\varphi_q({\mathbb R}^n)$.
\begin{proof}[Proof of Lemma \ref{lem:180526-1}]
Remark that both $\varphi$ is bijective. 
Let $R>0$.
By using the definition of $M$, 
we obtain
\begin{align*}
\frac{\rho(R)}{|B(x,R)|} \int_{B(x,R)} |f(y)| \ dy
\le
\rho(R) Mf(x)
\lesssim 
\varphi(R)^{1-a} Mf(x)
\end{align*}
and
\begin{align*}
\frac{\rho(R)}{|B(x,R)|} \int_{B(x,R)} |f(y)| \ dy
\lesssim
\rho(R)
\frac{\|f\|_{{\mathcal M}^\varphi_1}}{\varphi(R)}
\lesssim
\varphi(R)^{-a}. 
\end{align*}
Thus, it follows that
\begin{align*}
\frac{\rho(R)}{|B(x,R)|} \int_{B(x,R)} |f(y)| \ dy
&\lesssim
\min\left\{ 
\varphi(R)^{1-a} Mf(x), 
\varphi(R)^{-a} \right\}
\\
&\le 
\sup_{t>0}
\min\left\{
t^{1-a} M f(x), 
t^{-a} \right\}
\\
&=
M f(x)^a.
\end{align*}
Since $R>0$ being arbitrary, we obtain
(\ref{eq:150206-2}).
\end{proof}

The weak boundedness of $M_\rho$ can be characterized in a similar way.
\begin{corollary}{\rm \cite[Corollary 1]{HNS15}}\label{cor:180414-15}
Let $0<a <1 \le q<\infty$.
Let $\varphi\in{\mathcal G}_{q}$
satisfy
$\inf_{t>0} \varphi(t)=0$
and
$\sup_{t>0} \varphi(t)=\infty$. 
Then, $M_{\rho}$ is bounded from
${\mathcal M}^{\varphi}_{q}({\mathbb R}^n)$ 
to ${\rm w}{\mathcal M}^{\varphi^a}_{a^{-1}q}({\mathbb R}^n)$ 
if and only if 
$\rho$ and $\varphi$ satisfy 
$(\ref{eq:150206})$
for all $R>0$.
\end{corollary}

We move on to the vector-valued inequality 
for $M_{\rho}$ on generalized Orlicz--Morrey spaces 
and generalized weak Orlicz--Morrey spaces.
\begin{theorem}{\rm \cite[Theorem8]{HNS15}}\label{thm:150206-1A}
Let $0<a<1<q<\infty$ and let $1 \le u<\infty$.
Let $\varphi\in{\mathcal G}_{q}$.
Assume that 
$\displaystyle\lim_{t\downarrow 0} \varphi(t)=0$ and
$\displaystyle\lim_{t\to \infty} \varphi(t)=\infty$. 
\begin{enumerate}
\item
If
$\rho$ and $\varphi$ satisfy 
$(\ref{eq:Nakai-19})$ and $(\ref{eq:150206})$,
then 
for $\{f_j\}_{j=1}^\infty \subset {\mathcal M}^{\varphi}_{q}({\mathbb R}^n)$
\[
\left\|
\left(
\sum_{j=1}^\infty M_\rho f_j{}^u
\right)^{\frac{1}{u}}
\right\|_{{\mathcal M}^{\varphi^a}_{a^{-1}q}}
\lesssim
\left\|
\left(
\sum_{j=1}^\infty |f_j|^u
\right)^{\frac{1}{u}}
\right\|_{{\mathcal M}^{\varphi}_{q}}.
\]
\item
Conversely, if 
\begin{equation}\label{eq:150108-55}
\left\|
\left(
\sum_{j=1}^\infty M_\rho f_j{}^u
\right)^{\frac{1}{u}}
\right\|_{{\rm w}{\mathcal M}^{\varphi^a}_1}
\lesssim
\left\|
\left(
\sum_{j=1}^\infty |f_j|^u
\right)^{\frac{1}{u}}
\right\|_{{\mathcal M}^{\varphi}_{q}}
\end{equation}
for $\{f_j\}_{j=1}^\infty \subset {\mathcal M}^{\varphi}_{q}({\mathbb R}^n)$,
then $\rho$, $\varphi$ and $\psi$ satisfy \eqref{eq:150206}.
Moreover, under the assumption that $\rho\sim\varphi/\psi$,
inequality \eqref{eq:150108-55} holds if and only if
$\varphi$ satisfies $(\ref{eq:Nakai-19})$.
\end{enumerate}
\end{theorem}

\begin{proof}
\
\begin{enumerate}
\item
Using
$(\ref{eq:150206})$,
we may assume that
$\rho=\varphi^{1-a}$.
Then since $\varphi$ is a doubling function and $0<a<1$,
we have
\[
M_{\varphi^{1-a}}f_j \lesssim I_{\varphi^{1-a}}|f_j|.
\]
Thus,
\[
\left(
\sum_{j=1}^\infty M_\rho f_j{}^u
\right)^{\frac{1}{u}}
\lesssim
\left(
\sum_{j=1}^\infty (I_{\varphi^{1-a}}|f_j|)^u
\right)^{\frac{1}{u}}
\lesssim
I_{\varphi^{1-a}}
\left[\left(
\sum_{j=1}^\infty |f_j|^u
\right)^{\frac{1}{u}}\right].
\]
It remains to resort to the boundedness
of $I_{\varphi^{1-a}}$
from ${\mathcal M}^\varphi_q({\mathbb R}^n)$
to ${\mathcal M}^{\varphi^a}_{a^{-1}q}({\mathbb R}^n)$.
\item
We let
$$
 f_j=
\begin{cases}
 f, \quad j=1,\\
 0, \quad j\ge 1,
\end{cases}
\quad f\in {\mathcal M}^\varphi_1({\mathbb R}^n).
$$ 
Then
we have the boundedness of $M_{\rho}$ on ${\mathcal M}^\varphi_1({\mathbb R}^n)$.
Hence, 
by Theorem \ref{thm:150206-1}, 
we conclude that the inequality (\ref{eq:150206}) holds.

Finally, under the assumption that $\rho\sim\varphi/\psi$,
we prove that 
inequality \eqref{eq:150108-55} holds if and only if
$\varphi$ satisfies \eqref{eq:Nakai-19}.
To do this, it is enough to show that 
\eqref{eq:Nakai-19} follows from \eqref{eq:150108-55}.
Now, 
assume that 
the integral condition \eqref{eq:Nakai-19} fails.
Then, for any $m\in {\mathbb N}$, there exists $r_m>0$ such that 
$$
 \varphi(2^mr_m)\le 2\varphi(r_m).
$$
Letting 
$
 f_j=
\chi_{[1,m]}(j)
 \chi_{B(2^jr_m) \setminus B(2^{j-1}r_m)},
j \in {\mathbb N},
$ 
we have 
\begin{align}\label{eq:150211-1}
\|f_{j}\|_{{\mathcal M}^\varphi_\Phi(\ell^u)}
 \le \|\chi_{B(2^mr_m)}\|_{{\mathcal M}^\varphi_\Phi}
 \sim \varphi(2^mr_m)
 \le 2\varphi(r_m).
\end{align}
Since $\theta\in{\mathcal G}_n$ and
$\rho\sim\varphi/\psi=\varphi/\theta(\varphi)$,
$\rho(r)\lesssim \rho(s)$ for all $r\le s$.
Due to this fact and the inequality 
$M_{\rho} f_j \gtrsim \rho(2^jr_m)\chi_{B(r_m)}$,
we have
\begin{align}\label{eq:150211-2}
\|M_{\rho}f_{j}\|_{{\rm w}{\mathcal M}^\psi_1(\ell^u)} \nonumber
&\gtrsim
\left\|\left( \sum_{j=1}^m \rho(2^jr_m) ^u\right)^\frac{1}{u} \chi_{B(r_m)} 
\right\|_{{\rm w}{\mathcal M}^\psi_1}\\ \nonumber
&\gtrsim 
\rho(2r_m) \psi(2r_m) m^{\frac{1}{u}}\\
&\gtrsim \varphi(r_m) m^{\frac{1}{u}}.
\end{align}
We combine the inequalities (\ref{eq:150211-1}) and (\ref{eq:150211-2}) 
with the boundedness of $M_{\rho}$ from ${\rm w}{\mathcal M}^\varphi_q(\ell^u)$ to
${\mathcal M}^\psi_1(\ell^u)$ to obtain $m\le D$ where $D$ is independent of $m$, 
contradictory to the fact that $m\in {\mathbb N}$ is arbitrary. 
Thus the integral condition (\ref{eq:Nakai-19}) holds.
\end{enumerate}
\end{proof}

\begin{remark}
One may ask ourselves
how different $I_\alpha$ and $M_\alpha$.
See
\cite[Theorem 1.10]{GoMu12}
or
compare
\cite[Theorem 1.3]{SST11-1}
with
\cite[Theorem 1.7]{SST11-1}
and
\cite[Proposition 4.1]{SST11-1}
to see  the gap
between $I_\alpha$ and $M_\alpha$.
See also
\cite[Theorems 5.1 and 5.2]{GoMu12}
to see that when they are the same.
\end{remark}

\begin{remark}
See \cite[Theorem 5.2]{Guliyev09}
and
\cite[Theorem 4.3]{GuSh13}
for the Spanne-type boundedness of $M_\alpha$,
where the integral condition is assumed.
\end{remark}

\begin{remark}
See \cite[Theorem 5.5]{Guliyev09}
and
\cite[Theorem 4.4]{GuSh13}
for the Adams-type boundedness of $M_\alpha$,
where the integral condition is assumed.
\end{remark}
\begin{remark}
See \cite[Theorem 4.1]{YuTa14}
for the boundedness of the fractional maximal operator
on generalized Morrey spaces
in the multilinear setting.
\end{remark}

\section{Overview of other types of generalizations}
\label{s5}

\subsection{Morrey spaces for general Radon measures
and Morrey spaces over metric measure spaces}

In addition to the genealization of $p$ into functions,
one can replace the Lebesgue measure by general Radon measures.
Here we work on a metric measure space
$(X,d,\mu)$.
We refer to \cite{YYH13} for an exhaustive account
of the analysis on metric measure spaces.
Let $k>0$ and $0< q \le p<\infty$.
We define the Morrey space ${\mathcal M}^p_q(k,\mu)$ as
\[
{\mathcal M}^p_q(k,\mu)\equiv
\left\{f \in L^q_{\rm loc}(\mu):\,
\|f \|_{{\mathcal M}^p_q(k,\mu)}<\infty\right\},
\]
where
\begin{equation} \label{norm:def}
\|f \|_{{\mathcal M}^p_q(k,\mu)}\equiv
\sup_{B(x,r) \in {\mathcal B}(\mu)}\mu(B(x,k r))^{\frac{1}{p}-\frac{1}{q}}
\left(\int_{B(x,r)}|f(z)|^q\,d\mu(z)\right)^{\frac{1}{q}}.
\end{equation}
Here
${\mathcal B}(\mu)$ stands for the set of all
balls having positve $\mu$-measure.
In the Euclidean space ${\mathbb R}^n$,
${\mathcal Q}(\mu)$ stands for the set of all
cubes having positve $\mu$-measure.
Clearly we have $L^p(\mu)={\mathcal M}^p_p(k,\mu)$, and
by applying H\"{o}lder's inequality to \eqref{norm:def}
we have
$\|f \|_{{\mathcal M}^p_{q_1}(k,\mu)} 
\ge \|f \|_{{\mathcal M}^p_{q_2}(k,\mu)}$
for all $p \ge q_1 \ge q_2>0$ and $k \ge 1$.
Thus the following inclusions hold:
\begin{proposition}
Let
$p \ge q_1 \ge q_2>0$ and $k \ge 1$.
Then
\[
L^p(\mu) = {\mathcal M}^p_p(k,\mu)
\subset {\mathcal M}^p_{q_1}(k,\mu) \subset {\mathcal M}^p_{q_2}(k,\mu).
\]
\end{proposition}
A remarkable property of 
${\mathcal M}^p_q(k,\mu)$
is that 
the space
${\mathcal M}^p_q(k,\mu)$ does not depend on $k>1$.
\begin{proposition}\label{prop:180925-1}
Let $(X,d,\mu)$ be the Euclidean space
${\mathbb R}^n$
with the Euclidean distance and the Lebesgue measure.
Then for all 
${\mathcal M}^p_q(k,\mu)={\mathcal M}^p_q(2,\mu)$
$p \ge q>0$ and $k>1$.
\end{proposition}
We do not recall its proof
whose proof hinges on a geometric structure of ${\mathbb R}^n$;
see \cite[Proposition 1.1]{SaTa05}.
It can happen that
${\mathcal M}^p_q(1,\mu)$
is a proper subset
${\mathcal M}^p_q(2,\mu)$
in Proposition \ref{prop:180925-1},
as was shown in \cite{SST18}.
Instead of the norm above,
we can use
\[
\|f \|_{{\mathcal L}^\nu_q(k,\mu)}\equiv
\sup_{B(x,r) \in {\mathcal B}(\mu)}r^\nu\mu(B(x,k r))^{-\frac{1}{q}}
\left(\int_{B(x,r)}|f(z)|^q\,d\mu(z)\right)^{\frac{1}{q}}.
\]
See \cite{MSS09}.
See also \cite{SaSh13}.
We can also define
the generalized Morrey spaces
with Radon measures.
We work on a metric measure space
$(X,d,\mu)$.
Let $\varphi:(0,\infty) \to (0,\infty)$ be a function.
Then define
\[
{\mathcal M}^\varphi_q(k,\mu)\equiv
\left\{f \in L^q_{\rm loc}(\mu):\,
\|f \|_{{\mathcal M}^\varphi_q(k,\mu)}<\infty\right\},
\]
where
\begin{equation} \label{norm:def2}
\|f \|_{{\mathcal M}^p_q(k,\mu)}\equiv
\sup_{B(x,r) \in {\mathcal B}(\mu)}
\varphi(\mu(B(x,k r)))
\left(\frac{1}{\mu(B(x,k r))}\int_{B(x,r)}|f(z)|^q\,d\mu(z)\right)^{\frac{1}{q}}.
\end{equation}
See
\cite{GuSa13,Sawano08-1}
for generalized Morrey spaces with general Radon measures
and
\cite{Sawano06-1}
for weak Morrey spaces with general Radon measures.
See
\cite{Ho17,HSS16, SaSh13-2, SaSh17-1,SaSh17-2}
for this direction of approaches.
We refer to \cite{YLY15}
for Morrey spaces with gradient in this type of setting.

We can consider some concrete cases.
See
\cite{EGN18,GEM13,GuMa13} for genrealized Morrey spaces 
on Heisenberg group, 
\cite{GuRa12} for parabolic genealized Morrey spaces
and
\cite{DzKh12,GuMu11} for anisotropic genealized Morrey spaces.
See \cite{EGA17, GAM13} and \cite{Volosivets12}
for generalized Morrey spaces on Carnot groups
and
for generalized Morrey spaces for the $p$-adic fields.

\subsection{Local generalized Morrey spaces}

Motivated by the works 
\cite{Guliyev94,GORS18},
define the generalized local Morrey space
${\rm L}{\mathcal M}^\varphi_q({\mathbb R}^n)$
to be the set of all measurable functions $f$ such that
\begin{equation}\label{eq:131115-111a}
\|f\|_{{\rm L}{\mathcal M}^\varphi_q}
\equiv
\sup_{r>0}
\varphi(r)\left(
\frac{1}{|Q(r)|}
\int_{Q(r)}|f(y)|^q\,dy\right)^{\frac1q}<\infty.
\end{equation}
The space
${\rm L}{\mathcal M}^\varphi_q({\mathbb R}^n)$ 
is sometimes referred to as
the $B_\sigma$-spaces.
See \cite{NaSo16,SaYo18}
for  $B_\sigma$-spaces.

\subsection{Generalized Orlicz--Morrey spaces}

Instead of $p$,
we can generalize $q$ using Young functions.
In the definition below
we exclude the case where
$\varphi(t)=\infty$ for some $t \in(0,\infty)$.
\begin{definition}[Young function]
\label{defi:170428-113}
A function $\Phi:[0,\infty) \to [0,\infty)$
is said to be a Young function,
if there exists an increasing function
$\varphi$ which is right-continuous such that
\begin{equation}\label{eq:161231-701}
\Phi(t)=\int_0^t \varphi(s){\rm d}s \quad (t \ge 0).
\end{equation}
Equality $(\ref{eq:161231-701})$ is called
the {\it canonical representation} of a Young function $\Phi$.
By convention 
define $\Phi(\infty)\equiv \infty$.
\index{canonical representation@canonical representation}
\index{Young function@Young function}
\end{definition}

There are three types of generalizations.
We start with the Orlicz--Morrey spaces
of the first kind.
To this end
we start with the defintion of the $(\varphi,\Phi)$-average over $Q$.
\begin{definition}[$(\varphi,\Phi)$-average]\label{defi:108707-1}
\index{phi Phi average@$(\varphi,\Phi)$-average}
\index{$\|\cdot\|_{(\varphi,\Phi);Q}$}
Let $\varphi:(0,\infty) \to (0,\infty)$
be a function
and
$\Phi:(0,\infty) \to (0,\infty)$
a Young function.
For a cube $Q$ and $f \in L^0(Q)$, define the $(\varphi,\Phi)$-average over $Q$ by:
\[
\|f\|_{(\varphi,\Phi);Q}
\equiv
\inf\left\{
\lambda>0\,:\,
\frac{\varphi(\ell(Q))}{|Q|}\int_{Q}\Phi\left(\frac{|f(x)|}{\lambda}\right)\,dx
\le 1
\right\}.
\]
\end{definition}

In \cite{Nakai08-1}
Nakai defined
the generalized Orlicz--Morrey space of the first kind
via the $(\varphi,\Phi)$-average.
\begin{definition}[Generalized Orlicz--Morrey spaces of the first kind]
\label{defi:150824-259}
\index{L Phi phi@${\mathcal L}^{\varphi}_{\Phi}({\mathbb R}^n)$}
\index{generalized Orlicz--Morrey spaces of the first kind@generalized Orlicz--Morrey spaces of the first kind}
Suppose thatwe have a function
 $\varphi:(0,\infty) \to (0,\infty)$
and a Young function
$\Phi:[0,\infty) \to [0,\infty)$.
For a measurable function $f$
define
\[
\|f\|_{{\mathcal L}^{\varphi}_{\Phi}}
\equiv
\sup_{Q \in {\mathcal Q}}
\|f\|_{(\varphi,\Phi);Q}.
\]
The function space ${\mathcal L}^{\varphi}_{\Phi}({\mathbb R}^n)$,
the generalized Orlicz--Morrey space of the first kind,
is defined to be the set of all measurable functions
$f$ for which the norm $\|f\|_{{\mathcal L}^{\varphi}_{\Phi}}$
is finite.
Likewise
the function space ${\rm W}{\mathcal L}^{\varphi}_{\Phi}({\mathbb R}^n)$,
the weak generalized Orlicz--Morrey space of the first kind,
\index{w L Phi phi@${\rm W}{\mathcal L}^{\varphi}_{\Phi}({\mathbb R}^n)$}
\index{weak generalized Orlicz--Morrey spaces of the first kind@weak generalized Orlicz--Morrey spaces of the first kind}
is defined to be the set of all measurable functions
$f$ for which the norm $\|f\|_{{\rm W}{\mathcal L}^{\varphi}_{\Phi}}
=\sup_{\lambda>0}\lambda\|\chi_{(\lambda,\infty]}(|f|)\|_{{\mathcal L}^{\varphi}_{\Phi}}$
is finite.
\end{definition}

We move on to Orlicz--Morrey spaces of the second kind.
To define  Orlicz--Morrey spaces of the second kind,
we need another notion of the average.
\begin{definition}\label{defi:180426-1}
\index{Phi average@$\Phi$-average}
Let
$\Phi:[0,\infty) \to (0,\infty)$
a Young function.
For a cube $Q$, define its $\Phi$-average over $Q$ by:
\begin{equation}\label{eq:180313-1}
\|f\|_{\Phi;Q}
\equiv
\inf\left\{
\lambda>0\,:\,
\frac{1}{|Q|}\int_{Q}\Phi\left(\frac{|f(x)|}{\lambda}\right)\,dx
\le 1
\right\}.
\end{equation}
\end{definition}

With this new definition of the average in mind,
Sawano, Sugano and Tanaka define generalized Orlicz--Morrey spaces of the second kind
\cite{SST12}.
\begin{definition}[Generalized Orlicz--Morrey spaces of the second kind]\label{defi:108707-2}
\index{${\mathcal M}^{\varphi}_{\Phi}({\mathbb R}^n)$}
Suppose that we have a Young function
$\Phi:[0,\infty) \to [0,\infty)$ 
and a function
$\varphi:(0,\infty) \to [0,\infty)$.
Let $f \in L^0({\mathbb R}^n)$.
\begin{enumerate}
\item
Define
\[
\|f\|_{{\mathcal M}^{\varphi}_{\Phi}}
\equiv
\sup_{Q \in {\mathcal Q}}
\varphi(\ell(Q))\|f\|_{\Phi;Q}.
\]
The  generalized Orlicz--Morrey space
of the second kind ${\mathcal M}^{\varphi}_{\Phi}({\mathbb R}^n)$
is defined to be the
the set of all measurable functions
$f$ for which the norm $\|f\|_{{\mathcal M}^{\varphi}_{\Phi}}$
is finite.
\item
We define
\[
\|f\|_{{\rm W}{\mathcal M}^{\varphi}_{\Phi}}
\equiv\sup_{\lambda>0}\lambda\|\chi_{(0,\lambda)}(|f|)\|_{{\mathcal M}^{\varphi}_{\Phi}}.
\]
The function space ${\rm W}{\mathcal M}^{\varphi}_{\Phi}({\mathbb R}^n)$
is defined to be the weak generalized Orlicz--Morrey space
of the second kind
as the set of all measurable functions
$f$ for which the norm $\|f\|_{{\rm W}{\mathcal M}^{\varphi}_{\Phi}}$
is finite.
\end{enumerate}
\end{definition}
Finally, 
Deringoz, Samko and Guliyev define generalized Orlicz-Morrey space
of the third kind as follows:
\begin{definition}
The generalized Orlicz-Morrey space
${\mathcal Z}^{\varphi}_{\Phi}({\mathbb R}^n)$ of the third kind
is defined as the set of all measurable functions
$f$ for which the norm
\[
\|f\|_{{\mathcal Z}^{\varphi}_{\Phi}}
\equiv
\sup_{Q \in {\mathcal Q}} ~
\varphi(\ell(Q))\Phi^{-1}\left(\frac{1}{|Q|}\right) ~ \|f\|_{L^{\Phi}(Q)}
\]
is finite.
\end{definition}
We do not go into the details of these function spaces;
here we content ourselves with mentioning that
the first kind and the second kind are different
and that
the second kind and the third kind are different
according to \cite{GST15}.
See \cite{DGS14,GHSN16,HaSa16,SaSh18, SST11-2,SuWa12}
for more about these function spaces.

\subsection{Generalization of $\varphi$ to the function depending also on $x$}

One can consider the case where
$\varphi$ depends on $x$ not only on $r$.
\begin{equation}\label{eq:131115-111p}
\|f\|_{{\mathcal M}^\varphi_q}
\equiv
\sup_{x \in {\mathbb R}^n, r>0}
\varphi(x,r)\left(
\frac{1}{|Q(x,r)|}
\int_{Q(x,r)}|f(y)|^q\,dy\right)^{\frac1q}<\infty.
\end{equation}
See
\cite{ArNa18,Guliyev09,GuSa13,Nakai94}
for the approach to this direction.

\subsection{Martingale Morrey spaces}
One can use martingales
to generalize Morrey spaces \cite{NaSa12-1, NaSa17}.

Let $(\Omega,\Sigma,P)$ be a probability space, 
and ${\mathcal F}=\{{\mathcal F}_n\}_{n\ge0}$ a nondecreasing sequence of sub-$\sigma$-algebras 
of $\Sigma$ such that
$\Sigma=\sigma(\bigcup_{n}{\mathcal F}_n)$.
For the sake of simplicity, let ${\mathcal F}_{-1}={\mathcal F}_0$.
The set $B\in{\mathcal F}_n$ is called atom, more precisely $({\mathcal F}_n,P)$-atom, 
if any $A\subset B$, $A\in{\mathcal F}_n$, satisfies $P(A)=P(B)$ or $P(A)=0$.
Denote by $A({\mathcal F}_n)$ the set of all atoms in ${\mathcal F}_n$.

The expectation operator and the conditional expectation operators relative to ${\mathcal F}_n$
are denoted by $E$ and $E_n$, respectively.
It is known as the Doob theorem that, if $p\in(1,\infty)$, 
then any $L_{p}$-bounded martingale converges in $L_{p}$.
Moreover, if $p\in[1,\infty)$,
then, for any $f \in L_{p}$, its corresponding martingale $\{f_n\}_{n=1}^\infty$ with $f_n=E_nf$ 
is an $L_{p}$-bounded martingale and converges to $f$ in $L_p$
(see for example \cite{Neveu1975}).
For this reason 
a function $f\in L_1$ and the corresponding martingale $\{f_n\}_{n=1}^\infty$ 
will be denoted by the same symbol $f$.

Let ${\mathcal M}$ be the set of all martingales 
such that $f_0=0$.
For $p\in[1,\infty]$, 
let $L^0_p(\Omega,{\mathcal F},P)$ be the set of all $f\in L^p(\Omega,{\mathcal F},P)$ such that $E_0f=0$.
For any $f \in L^0_p(\Omega,{\mathcal F},P)$, 
its corresponding martingale $(f_n)$ with $f_n=E_nf$ 
is an $L_{p}$-bounded martingale in ${\mathcal M}$.
For this reason we regard as $L^0_p(\Omega,{\mathcal F},P)\subset{\mathcal M}$.

Let ${\mathcal B}=\{{\mathcal B}_n\}_{n\ge0}$ be sub-families of ${\mathcal F}=\{{\mathcal F}_n\}_{n\ge0}$ with
${\mathcal B}_n\subset{\mathcal F}_n$ for each $n\ge 0$.
We denote by ${\mathcal B}\subset{\mathcal F}$ this relation of ${\mathcal B}$ and ${\mathcal F}$. 

In this paper we always postulate the following condition on $\mathcal{B}$:
\begin{equation}\label{full}
\text{There exists a countable subset $\mathcal{B}'\subset\mathcal{B}_0$ 
such that }P\left(\bigcup_{B\in\mathcal{B}'}B\right)=1.
\end{equation}
We first define generalized martingale Morrey-Campanato spaces with respect to ${\mathcal B}$
as the following:

\begin{definition}\label{defn:gMC}
Let ${\mathcal B}\subset{\mathcal F}$, $p\in [1,\infty)$ and $\varphi:(0,1]\to(0,\infty)$.
For $f\in L_1$, let 
\begin{align*}
  \|f\|_{L^\varphi_p}
  =\|f\|_{L^\varphi_p({\mathcal B})}
  &=
  \sup_{n\ge0}\sup_{B\in \mathcal{B}_n}
  \varphi(P(B))\left(\frac{1}{P(B)}\int_B|f|^p\,dP\right)^{1/p}, 
\\
  \|f\|_{{\mathcal L}^\varphi_{p}}
  =\|f\|_{{\mathcal L}^\varphi_{p}({\mathcal B})}
  &=
  \sup_{n\ge0}\sup_{B\in \mathcal{B}_n}
  \varphi(P(B))\left(\frac{1}{P(B)}\int_B|f-E_nf|^p\,dP\right)^{1/p},
\\
  \|f\|_{({\mathcal L}^\varphi_{p})^{-}}
  =\|f\|_{({\mathcal L}^\varphi_{p})^{-}({\mathcal B})}
  &=
  \sup_{n\ge0}\sup_{B\in \mathcal{B}_n}
  \varphi(P(B))\left(\frac{1}{P(B)}\int_B|f-E_{n-1}f|^p\,dP\right)^{1/p},
\end{align*}
and define
\begin{align*}
  L^\varphi_p 
  =L^\varphi_p({\mathcal B}) 
  &= \{f\in L^0_p: \|f\|_{L^\varphi_p}<\infty \}, 
\\
  {\mathcal L}^\varphi_{p} 
  ={\mathcal L}^\varphi_{p}({\mathcal B}) 
  &= \{f\in L^0_p: \|f\|_{{\mathcal L}^\varphi_{p}}<\infty \}, 
\\
  ({\mathcal L}^\varphi_{p})^{-} 
  =({\mathcal L}^\varphi_{p})^{-}({\mathcal B}) 
  &= \{f\in L^0_p: \|f\|_{({\mathcal L}^\varphi_{p})^{-}}<\infty \}. 
\end{align*}
\end{definition}

\subsection{Replacing $\sup$ in the Morrey norm by other norms}
Instead of taking the supremum,
Fueto considered to take the $L^p$-norm in \cite{Fueto14}.
Although Feuto worked in the weighted setting,
we describe it in the unweighted setting.
\begin{definition}
Let $0<q \le p<\infty$ and $0<r\le \infty$.
One defines
${\mathcal M}^p_{q,r}$ to be the set of all
$f \in L^q_{\rm loc}$
for which
\[
\|f\|_{{\mathcal M}^p_{q,r}}
=
\left\|
\left\{
|Q_{\nu m}|^{\frac{1}{p}-\frac{1}{q}}
\left(\int_{Q_{\nu m}}|f(y)|^q\,dy\right)^{\frac{1}{q}}
\right\}_{\nu \in {\mathbb Z}, m \in {\mathbb Z}^n}
\right\|_{\ell^r}
\]
is finite. 
\end{definition}
See
\cite{Bourgain91,MaSe18}
for applications to partial differential equations.

\subsection{Grandification of the parameter $q$}

In addition to generalization of the parameter $p$,
we can also grandify the parameter $q$;
for $f \in L^0({\mathbb R}^n)$,
we define
\[
\|f\|_{{\mathcal M}^p_{q),\theta}}
\equiv
\sup_{x \in {\mathbb R}^n, r>0}
\sup_{\varepsilon \in (0,q-1)}
\varepsilon^\theta
|Q(x,r)|^{\frac1p-\frac1{q-\varepsilon}}
\left(\int_{Q(x,r)}|f(y)|^{q-\varepsilon}\,dy\right)^{\frac1{q-\varepsilon}}.
\]
The space 
${\mathcal M}^p_{q),\theta}({\mathbb R}^n)$
collects all $f \in L^0({\mathbb R}^n)$
for which
$\|f\|_{{\mathcal M}^p_{q),\theta}}$ is finite.
See \cite{KMR13,MeSa18,MiOh14-1}
for more details and variants.

\subsection{The case of the variable expoenent}

By a variable exponent
we mean any measurable function 
from ${\mathbb R}^n$ to a subset of $(-\infty,\infty]$.
We define
$\|\cdot\|_{L^{p(\cdot)}}$
which is called
the variable Lebesgue norm
or
the Nakano--Luxenburg norm.
\begin{definition}[Variable Lebesgue spaces, Variable exponent Lebesgue space]
\label{defi:170428-116}
Let
$$p(\cdot) : {\mathbb R}^n \to [1,\infty ]$$
be a measurable function. 
Then define
the {\it variable exponent Lebesgue space}
$L^{p(\cdot)}({\mathbb R}^n)$ {\it with variable exponents} 
by
\[
L^{p(\cdot)}({\mathbb R}^n) \equiv 
\bigcup_{\lambda>0}
\{ f \in L^0({\mathbb R}^n) \, : \, 
 \rho_p(\lambda^{-1}f) < \infty \} ,
\]
where 
\[
\rho_p(f) \equiv \|\chi_{p^{-1}(0,\infty)} |f|^{p(\cdot)}\|_1
+\| f \|_{L^{\infty}(p^{-1}(\infty))} .
\]
Moreover,
for $f \in L^{p(\cdot)}({\mathbb R}^n)$ one defines the 
{\it variable Lebesgue norm}
by
\[ \| f \|_{L^{p(\cdot)}({\mathbb R}^n)}\equiv 
\inf\left( \left\{ \lambda \in(0,\infty) \,:\, \rho_p(\lambda^{-1}f) \le 1 \right\} \cup \{\infty\}\right).
\]
\index{$\| \cdot \|_{L^{p(\cdot)}({\mathbb R}^n)}$}
\index{Nakano--Luxenburg norm@Nakano--Luxenburg norm}
\index{variable exponent Lebesuge space@variable exponent Lebesgue space}
\end{definition}
Using this techinique,
we can generalize the exponent $q$.
See \cite{GHS10, LPSW13, LoHa16} for example.

\section*{Acknowledgement}

The author is thankful to Professors Eiichi Nakai
and Vagif Guliyev
for their careful reading.

\end{document}